\documentclass[11pt,reqno]{amsart}
\usepackage{fullpage}
\date{}

\usepackage{amsfonts,amsmath,amsthm,amssymb,latexsym,mathrsfs,stmaryrd}
\usepackage{hyperref}
\usepackage{mathtools}
\usepackage{graphicx}
\usepackage{subcaption}
\usepackage{bbm}
\usepackage[capitalize]{cleveref}
\usepackage{soul} 
\usepackage{amsmath,amssymb,amscd}
\usepackage{amsthm}
\usepackage{enumitem}
\usepackage{listings,galois}
\usepackage{geometry}
\usepackage{mathrsfs}
\usepackage{CJK}
\usepackage{amssymb}
\usepackage{tikz-cd}
\usepackage{hyperref}
\usepackage{tikz}
\usepackage{graphicx}

\usepackage{comment}

\usepackage{supertabular}

\title{}
\author{}
\newtheorem{theorem}[subsubsection]{Theorem}
\newtheorem{lemma}[subsubsection]{Lemma}
\newtheorem{proposition}[subsubsection]{Proposition}

\newtheorem{corollary}[subsubsection]{Corollary}
\newtheorem*{theorem*}{Theorem}
\newtheorem*{proposition*}{Proposition}
\newtheorem*{corollary*}{Corollary}

\theoremstyle{definition}
\newtheorem{definition}[subsubsection]{Definition}
\newtheorem{notation}[subsubsection]{Notation}
\newtheorem{assumption}[subsubsection]{Assumption}
\newtheorem{construction}[subsubsection]{Construction}
\newtheorem{setup}[subsubsection]{Setup}

\newtheorem{example}[subsubsection]{Example}
\newtheorem{claim}[subsubsection]{Claim}

\theoremstyle{remark}
\newtheorem{remark}[subsubsection]{Remark}

\newcommand{\cL}{\mathcal{L}}

\newcommand{\bA}{\mathbb{A}}
\newcommand{\bL}{\mathbb{L}}
\newcommand{\bZ}{\mathbb{Z}}
\newcommand{\bQ}{\mathbb{Q}}
\newcommand{\bR}{\mathbb{R}}
\newcommand{\bC}{\mathbb{C}}
\newcommand{\bK}{\mathbb{K}}
\newcommand{\bF}{\mathbb{F}}
\newcommand{\bI}{\mathbb{I}}

\newcommand{\cW}{\mathcal{W}}
\newcommand{\cF}{\mathcal{F}}
\newcommand{\cD}{\mathcal{D}}

\DeclareMathOperator{\Frob}{Frob}
\newcommand{\cris}{\mathrm{cris}}
\DeclareMathOperator{\Spf}{Spf}
\DeclareMathOperator{\Spec}{Spec}

\DeclareMathOperator{\Fil}{Fil}

\newcommand{\mpr}{\textcolor{purple}}
\newcommand{\bfx}{\mathbf{x}}
\newcommand{\bfy}{\mathbf{y}}

\newcommand{\sF}{\mathcal{F}}
\newcommand{\sO}{\mathcal{O}}
\newcommand{\dR}{\mathrm{dR}}
\newcommand{\into}{\hookrightarrow}
\newcommand{\gr}{\mathrm{gr}}
\newcommand{\tor}{\mathrm{tor}}
\def\wt{\widetilde}
\newcommand{\tensor}{\otimes}

\newcommand{\shS}{\mathscr{S}}

\newcommand{\cA}{\mathcal{A}}

\newcommand{\cZ}{\mathcal{Z}}

\newcommand{\bH}{\mathbb{H}}
\newcommand{\bN}{\mathbb{N}}
\newcommand{\End}{\mathrm{End}}

\newcommand{\Sh}{\mathrm{Sh}}

\newcommand{\Cl}{\mathrm{Cl}}

\usepackage{todonotes}

\usepackage{multirow}
\usepackage{float}
\author{Ruofan Jiang, Ananth N. Shankar, Ziquan Yang}
\newcommand{\<}{\langle}
\renewcommand{\>}{\rangle}
\title{Picard rank jumps for families of K3 surfaces in positive characteristic}
\begin{document}
\maketitle
\begin{abstract}
Let $\mathscr{X}/C$ be a non iso-trivial family of K3 surfaces over a curve $C$ defined over characteristic $p>2$ field. We show that if $\mathscr{X}$ avoids a necessary and structural obstruction coming from Frobenius, and satisfies a big monodromy condition, then there are infinitely may geometric fibers that have larger Picard rank than the geometric generic fiber.

\end{abstract}
\tableofcontents
\section{Introduction}

In this paper, we address the question of the variation of Picard ranks in families of K3 surfaces over a base in positive characteristic $p > 2$. This question is very well understood in characteristic zero. Indeed, let $\mathscr{X}\rightarrow \Delta$ be a holomorphic family of K3 surfaces. Define the \emph{Noether--Lefschetz locus} (succinctly, the NL locus) to be the set of points $s\in \Delta$ such that the Picard rank of $\mathscr{X}_s$ is greater than the generic Picard rank of $\mathscr{X}$. Work of Oguiso (\cite{Oguiso}) and Green (\cite[Thm.~3.5]{Voisin}) shows that if the family is not isotrivial, the NL locus is dense in $\Delta$, and in particular is infinite. 

The NL locus is can be defined for families of K3 surfaces over more general bases. Specifically, let $\mathscr{X}/C$ be an algebraic family of K3 surfaces over some irreducible scheme $C$ with generic point $\eta$. Define the NL locus to be the set of points $s\in C$ such that the Picard rank of $\mathscr{X}_{\bar{s}}$ is greater than that of $\mathscr{X}_{\bar{\eta}}$. The works of Green and Oguiso shows that if $C$ is a positive dimensional scheme over $\bQ$ and $\mathscr{X}$ is isotrivial, then the exceptional locus is Zariski-dense in $C$. Studying the NL locus becomes significantly more challenging when the $C$ is a scheme over an arithmetic base. For instance, let $\mathscr{X}/\mathcal{O}_K[1/N]$ be a K3 surface, where $\mathcal{O}_K$ is the ring of integers of a number field $K$. In this setting, the NL locus is a union of points of the form $\Spec \bF_q$. The series of works \cite{Ch18}, \cite{ST}, \cite{SSTT} and \cite{Tay22} shows that the NL locus is infinite, and therefore Zariski dense in $C$.  

\subsection{Exceptional families: counterexamples in positive characteristic}\label{subsec: first counterexample}

The setting where $C$ is defined over a field of positive characteristic is more subtle, and even the answer depends on characteristic $p$ invariants that do not exist in characteristic zero. Note that the question reduces to the case where $C$ is a curve. So in what follows, we will always assume that $C$ is a curve. Moreover, up to replacing $C$ by a finite cover, we assume that line bundles on the geometric generic fiber are already defined on the generic fiber.

The work \cite{MST} (along with \cite{Tay22} to treat the case of non-projective $C$) shows that the NL locus is Zariski dense if the generic K3 surface parameterized by $C$ is ordinary. But as noted in \cite{MST}, the ordinariness condition is very important. For example, the generic Picard rank and the Picard rank of every specialization is 22 for a family of supersingular K3 surfaces, and so the NL locus is empty. This is not the only obstruction to the exceptional locus not being dense. Consider any non-isotrivial family of elliptic curves $\mathscr{E}/C$, and let $E$ be a supersingular elliptic curve. Then, consider the Kummer family $\mathcal{K}/C$ associated to the abelian surface $\mathscr{E}\times E /C$. It is easy to check that the Picard rank of $\mathcal{K}_{\bar{s}}$ is greater than the generic Picard rank if and only if $s\in C$ is a point of supersingular specialization of $\mathscr{E}$. As there are only finitely many such points in $C$, it follows that the NL locus is finite. The above example is not an isolated one, and perists even for families of K3 surfaces with smaller Picard rank. We note that being non-ordinary alone is not sufficient for the NL locus to be finite. For instance, \cite{RJ23} shows that families of non-ordinary abelian surfaces with trivial endomorphism ring are split at infinitely many fibers, a question that is intimately related to the distribution of the NL locus.

 To explain the phenomenon underlying these examples, we will need to work with moduli spaces of lattice-polarized K3 surfaces. These spaces are (open parts of) GSpin Shimura varieties $S_L$ associated to a quadratic lattice $L$ having signature $(b,2)$ with ($b\leq 19$). These Shimura varieties admit families of \emph{special divisors} which are themselves GSpin Shimura varieties associated to sublattices of $L$ having signature $(b-1,2)$. We make the standing assumption that the discriminant of $L$ (or equivalently the discriminant of $\mathrm{Pic}(\mathscr{X}_\eta)$) is prime to $p$. By \cite{Kis10}, this implies the existence of smooth integral models $\mathscr{S}_L$ of $S_L$ at $p$. 
    Under this assumption, there is a minimum-rank lattice $L$ and a map $C\rightarrow \mathscr{S}_L$ that induces $\mathscr{X}$, and the assumption on the Picard lattice of the generic fiber of $\mathscr{X}$ implies that $L$ has prime-to-$p$ discriminant. The mod $p$ special fiber of $\mathscr{S}_L$ admits a Newton stratification, with the open stratum being the ordinary locus and the smallest stratum being the supersingular locus ($\mathscr{S}_{\textrm{ss}}$). We consider the case when $b$ is even -- indeed, the NL locus in the odd case necessarily consists of every $\overline{\bF}_p$-point as the the Picard rank of a K3 surface over a finite field is necessarily even. This is also true in the more general setting of $\textrm{GSpin}$ Shimura varieties and is a consequence of the group-theoretic fact that the rank of $\textrm{SO}_{2n-1}$ equals the rank of $\textrm{SO}_{2n}$. The geometry and combinatorics of the Newton strata only depend on the isomorphism class of $L\otimes \bZ_p$. There are two cases -- when $L\otimes \bZ_p$ is a direct sum of hyperbolic planes, and when it isn't. In the first case, the dimension of $\mathscr{S}_{\textrm{ss}}$ is smaller than in the second case, and has dimension $b/2 -1$ (whereas, in the second case, $\mathscr{S}_{\textrm{ss}}$ has dimension $b/2$). In the first case, the fact that $\mathscr{S}_{\textrm{ss}}$ is smaller than expected has the crucial consequence that \emph{every special divisor} contains the $\mathscr{S}_{\textrm{ss}}$. Further, the unique ``almost supersingular'' Newton stratum, which we call $\mathscr{S}_{\textrm{as}}$ (whose closure $\overline{\mathscr{S}_{\textrm{as}}}$ equals $\mathscr{S}_{\textrm{as}} \cup \mathscr{S}_{\textrm{ss}}$), does not intersect any special divisor! In other words, the intersection of $\overline{\mathscr{S}_{\textrm{as}}}$ with every special divisor is contained in $\mathscr{S}_{\textrm{ss}}$. In particular, if $\mathscr{X}/C$ is induced by a map $C\rightarrow \overline{\mathscr{S}_{\textrm{as}}}$, then the NL locus of $C$ is necessarily contained in the supersingular locus of $C$, and thus must be finite. It is exactly this phenomenon that shows up in the Kummer example. 

\begin{definition}\label{def: exceptional family} We say that the family of K3 surfaces $\mathscr{X}/C$ is \textit{exceptional} either if it is supersingular, or if $L \tensor \bZ_p$ is a direct sum of hyperbolic planes and the classifying map $C \to \mathscr{S}_L$ factors through $\overline{\mathscr{S}}_{\mathrm{as}}$.

\end{definition}
\begin{remark}
\begin{enumerate}
      \item Consider the relative second crystalline cohomology of $ \mathscr{X}/C$. This $F$-crystal has a quadratic pairing that makes it a \textit{K3 crystal}; cf. \cite[Def.~5.1]{NO}. Its transcendental part, i.e., the orthogonal complement of the image of $\mathrm{Pic}(\mathscr{X}_\eta)$, is also a K3 crystal. Let us call it $\bL(\mathscr{X})$. It has at most three slopes at the generic point, and we have that the family $\mathscr{X}/C$ is exceptional if and only if $\bL(\mathscr{X})$ generically has one or two slopes.
      


    \item A generically ordinary non-isotrivial family of K3 surfaces can never be exceptional. Indeed, a generically ordinary family $\mathscr{X} \rightarrow C$ would be exceptional if and only if the associated Shimura variety $S_L$ satisfies $b = 0$, and is therefore zero-dimensional. This implies that $\mathscr{X}\rightarrow C$ is isotrivial.  
\end{enumerate}
\end{remark}

\subsection{Statement of results in the setting of K3 surfaces}
Our first result is in the setting of non-isotrivial families $\mathscr{X}/C$ which are not exceptional, and which satisfy a monodromy result. 

\begin{theorem}\label{main}
    Let $\mathscr{X}\rightarrow C$ be a non iso-trivial non-exceptional family of K3 surfaces in characteristic $p>2$ such that the discriminant of the generic Picard lattice is prime-to-$p$. Let $L$ be the minimal rank lattice with $\mathscr{X}\rightarrow C$ induced by a map $C\rightarrow \mathscr{S}_L$. Suppose that the $\ell$-adic monodromy equals that of $\mathscr{S}_L$. Then the NL locus is dense.
\end{theorem}

We will elaborate on the big monodromy assumption -- indeed, there are non-exceptional examples where the NL locus is provably finite when the family has smaller monodromy, which we describe in Section \ref{sec: resof scalars}. But first, we state a stronger theorem that implies the above result. 

\begin{theorem}\label{mainbrauer}
    Let $\mathscr{X}\rightarrow C$ be a non iso-trivial non-exceptional family of K3 surfaces in characteristic $p>2$ such that the discriminant of the generic Picard lattice is prime-to-$p$. Let $\eta$ be the generic point of $C$. Suppose that the connected-étale exact sequence of the enlarged formal Brauer group of $\mathscr{X}_\eta$ does not split. Then the NL locus is dense. 
\end{theorem}
The notion of enlarged formal Brauer group traces back several years; cf. \cite[\S IV]{Artin-Mazur} and \cite[\S 3]{Nygaard-Ogus}, where the group scheme is defined at least for non-supersingular K3 surfaces over finite fields. We will define the enlarged formal Brauer group for  $\mathscr{X}_\eta$ in Section \ref{sec: algebraization}. Roughly speaking, it is a $p$-divisible group characterized by 
the property that (up to a Tate twist) its Dieudonné crystal equals the quotient of the K3 crystal of $\mathscr{X}_\eta$ by the subcrystal of the smallest slope. Note that when $\mathscr{X}$ is generically ordinary, the connected-étale exact sequence of the enlarged formal Brauer group admits a global splitting if and only if the family $\mathscr{X}\rightarrow C$ is isotrivial. 

 A consequence of the enlarged formal Brauer group admitting such a splitting is that the generic Kuga--Satake abelian scheme admits extra endomorphisms. The above result therefore has the following corollary. 

\begin{corollary}\label{thm: main Kuga}
    Let $\mathscr{X}\rightarrow C$ be a non iso-trivial non-exceptional family of K3 surfaces in characteristic $p>2$ such that the discriminant of the generic Picard lattice is prime-to-$p$. Let $L$ be the minimal rank lattice with $\mathscr{X}\rightarrow C$ induced by a map $C\rightarrow \mathscr{S}_L$, and let $\cA / \mathscr{S}$ be the universal Kuga--Satake abelian scheme over $\mathscr{S}$. Then the NL locus is dense if $\End(\cA|_C)\otimes \bQ = \End(\cA) \otimes \bQ$.  
\end{corollary}

\begin{remark}
These theorems are stated for K3 surfaces only for concreteness. By combining a general formulation \cref{mainShimura} below with results from \cite{HYZ, GY25}, they extend verbatim to other varieties with $h^{2,0} = 1$ whose moduli satisfy mild assumptions---for example, elliptic surfaces of height $1$ over genus $1$ curves. We leave the precise formulations to the interested reader.    
\end{remark}

\subsection{Counterexamples with small monodromy}\label{sec: resof scalars}
We note that the case of exceptional K3 surfaces are not the only counterexamples to the NL locus being dense. For example, let $F/\bQ$ be a real quadratic field, and let $V'/F$ be a 4-dimensional vector space along with a bilinear form that is positive definite at one real place and has signature $(2,2)$ at the other. The Shimura variety $S_{V'}$ associated to quadratic space is defined over $F$, and embeds in $S_L$ where $L$ is a rank 8 quadratic $\bZ$-lattice with signature $(6,2)$. By arguments of Geeman and Nikulin (see \cite[Prop.~3.3]{Geeman}), one can arrange that $L$ embeds into the K3 lattice so that $S_L$ parametrizes K3 surfaces with real multiplication by $F$. 

It is easy to see that there are infinitely many primes $p$ such that $p$ splits in $F$ with $v_1,v_2 \mid p$ the places above $p$, and such that $V'_{v_i}$ (for $i = 1,2$) is a direct sum of two hyperbolic planes over $\bQ_p$. Fix such a $p$ and $v = v_i$ for some $i\in \{1,2 \}$. The mod $v$ Shimura variety $\mathscr{S}_{V',v}$ admits three Newton strata: the ordinary stratum, the supersingular stratum, and one intermediate stratum. This intermediate stratum does \emph{not} intersect $\mathscr{S}_{\textrm{as}}\subseteq \mathscr{S}_L $. However, it behaves like the 1-dimensional Newton stratum in the moduli space of a pair of elliptic curves in that the NL locus intersects the closure only along the supersingular locus. Therefore, the NL locus is always empty for a family of K3 surfaces pulled back from this intermediate Newton stratum. It is easy to construct similar examples where the K3 surface in question has smaller Picard ranks.

\subsection{General GSpin Shimura varieties}
In this section, we state the results in the general setting of GSpin Shimura varieties that go beyond the K3 case. We will always fix $p$ to be an odd prime. 

Let $L$ be a lattice of signature $(b,2)$ which is self-dual at $p$, and let $\mathscr{S}_{L}$ be the canonical integral model of the associated GSpin Shimura variety at $p$. Though $\mathscr{S}_L$ does not parametrize K3 surfaces in general, one can still equip $\mathscr{S}_L$ with a universal Kuga--Satake abelian scheme $\mathcal{A}/\mathscr{S}_L$ via a suitable Hodge embedding. Similar to the K3 case, $\mathscr{S}_L$ is equipped with a family of special divisors $\{\mathcal{Z}(m)\}_{m\geq 1}$ which are Zariski closure in $\mathscr{S}$ of GSpin Shimura subvarieties associated to sublattices of $L$ having signature $(b-1,2)$. The special divisors parameterize points with so-called special endomorphisms (see Section \ref{subsec: set up SV} for definitions). The notion of formal Brauer groups also generalizes; cf, \S\ref{sec: algebraization}.

Similar to the case of K3 surfaces, when $b$ is even, there is a dichotomy between $L\otimes \bZ_p$ being a direct sum of hyperbolic planes, and otherwise. In the former case, the supersingular locus has dimension $b/2 - 1$ and $\overline{\mathscr{S}_{\textrm{as}}} \cap \cZ(m)$ is contained in the supersingular locus for every $m$, and thus the NL locus cannot be dense for curves generically mapping to this locus. Now we let $\iota:C\rightarrow \mathscr{S}_{L}$ be a nontrivial morphism, where the discriminant of $L$ is prime to $p$. We further assume that $\mathscr{S}_L$ is the smallest $\textrm{GSpin}$ Shimura varieties containing the image of $\iota$, i.e. $\iota$ does not factor through any $\mathcal{Z}(m)$.  
Similar to the K3 case, we say that $\iota$ is \textit{exceptional}, if its image is supersingular, or $L\otimes \bZ_p$ is a direct sum of hyperbolic planes and $\iota$ factors through $\overline{\mathscr{S}_{\textrm{as}}}\subseteq \mathscr{S}_{L}$.



\begin{theorem}\label{mainShimura}
    Let $C\rightarrow \mathscr{S}_{L}$ be as above, and suppose that it is non-exceptional. Then the NL locus is dense if one of the following is true: 
    \begin{enumerate}[label=\upshape{(\alph*)}]
        \item the connected-étale exact sequence of the enlarged formal Brauer group over the generic point of $C$ does not split. 
        \item $\End(\cA|_C)\otimes \bQ = \End(\cA) \otimes \bQ$.
        \item The $l$-adic monodromy of $\mathcal{A}|_C$ equals that of $\mathcal{A}|_{\mathscr{S}_L}$. 
    \end{enumerate}
\end{theorem}


\subsection{Outline of proof}
The overarching outline follows \cite{COhypersymmetric}, \cite{MST1} and \cite{MST}, and is a local-global argument. The idea (first used in \cite{COhypersymmetric}) is to use global inputs to estimate the global intersection  $C . \cZ(m)$ as $m$ varies, and to show that the local contribution $i_P(C. \cZ(m))$ (i.e., the intersection multiplicity of $C$ with $\cZ(m)$ at $P$) from any one point $P\in C$ is small relative to $C.\cZ(m)$. The difficulty, of course, is in proving good enough bounds for $i_P(C . \cZ(m))$. Our proof of this consists of three parts having distinct flavors. The first part is mostly global, and in particular shows that $i_P(C.\cZ(m))$ is small enough as long as $P$ is not a supersingular point. For supersingular points, we reduce the main theorem to a purely local statement in the context of a formal curve $\Spec \bF[\![t]\!] \rightarrow \mathscr{S}$. This technique first appeared in \cite{SSTT} and \cite{MST}, and the lack of ordinarity introduces no additional difficulty. Part 2 is purely local, and we establish a combinatorial characterization of ``decaying'' lattices of special endomorphisms, which is needed to control $i_P(C.\cZ(m))$ when $P$ is a supersingular point. The fact that the curve avoids the ordinary locus introduces significant difficulties, which take new ideas to overcome. Part 3 is an algebraization result used to handle an extreme situation occurring in Part 2 where ``maximal amount'' of special endomorphisms resist decay (hence the local intersections can not be effectively controlled). When this happens, we show that the connected-étale exact sequence of the enlarged formal Brauer group over the generic point of $C$ must split. This implies that the generic Kuga--Satake abelian scheme admits extra algebraic endomorphisms. 

We now describe these three parts in more detail.
\subsubsection{Local-global comparison} Borcherd's theory tells us the asymptotic growth rate of the global intersection number $C . Z(m)$ as $m \to \infty$. More precisely, $C . Z(m)$ is of the form $|q_L(m)|(C \cdot \omega) + o(m^{b/2})$ and $q_L(m) \asymp m^{(b + 2) / 2}$. Here, $\omega$ is the Hodge bundle of $\mathscr{S}$. 

Let $P$ be a $\overline{\bF}_p$-point on $C$ and let $t$ be a uniformizer of $C$ at $P$, so that the formal completion of $C$ at $P$ is $\Spf(\overline{\bF}_p[\![t]\!])$. Let $\cA_r$ be the restriction of $\cA$ to $\overline{\bF}_p[t]/ (t^r)$, and let $L_r := L(\cA_r)$. Note that we have a nested sequence of lattices $L(\cA_P) = L_1 \supseteq L_2 \supseteq L_3 \supseteq \cdots$, such that a special endomorphism $v \in L(\cA_P)$ deforms to $\overline{\bF}_p[t]/ (t^r)$ if and only if $v \in L_r$. By the moduli interpretation of $\cZ(m)$, the local intersection number $i_P(C . \cZ(m))$ is computed by $\sum_{r = 1}^\infty \# \{ v \in L_r : v^2 = m \}$. It is proved in \cite{MST} (crucially using inputs from \cite{SSTT}) that if $P$ is non-supersingular, then for large $X$, $\sum_{m \le 2X} i_P(C \cdot Z(m)) = O(X^{b/2} \log X)$, so the growth rate is not on the same order as the global intersection number. By the work of Tayou \cite{Tay22}, a similar statement holds when $P$ is a boundary point. When $P$ is supersingular, we have $i_P(C . \cZ(m)) \asymp m^{(b + 2) / 2}$ as well, so we need to compare it with $|q_L(m)|(C \cdot \omega)$ more carefully, and this is done in Part 2. 




\subsubsection{Supersingular points and decaying lattices of special endomorphisms} Now let $P$ be supersingular. Roughly speaking, to bound $i_P(C . \cZ(m))$ we need $L_r$ to become small quickly as $r$ grows. This is done by examining for each $n$, the ``$n$-th order decay rate'' $d_n(v) = \max \{ r : p^{n - 1} v \in L_r \}$ for $v \in L(\cA_P)$. There are two kinds of results that we need:


\begin{enumerate}
    \item There is a saturated sublattice (having large enough rank) of $L(\cA_P)$ consisting of vectors $v$ that decays rapidly. See sections \ref{sec: decay superspecial} and \ref{sec: supersingular} for the definition of rapid decay, which roughly means that $d_n(v)$'s are subject to certain bounds.
    \item The first time a primitive endomorphism $v$ in this sublattice decays is sufficiently early (i.e., $d_1(v)$ is sufficiently small). See section \ref{sec: first step superspecial}.
\end{enumerate}

We need a lot more control for both steps as compared to the ordinary case. In particular, ``large enough rank'' just means rank 2 in the ordinary case, and we are able to establish this in a comparatively ad-hoc way in \cite{MST} and \cite{MST1}, whereas our treatment in the case of general Newton strata is (and needs to be) more systematic. Indeed, we are able to precisely relate the rates of higher-order decay to the degree of tangency of the curve to deeper \emph{Newton} strata. This takes up all of Section \ref{sec: decay superspecial} and part of \ref{sec: supersingular}.

When $C$ generically maps into the ordinary stratum, one only requires \emph{one} special endomorphism to decay ``early'' to first order, which can be deduced just from the fact that curve mapping non-constantly to $\mathscr{S}$. However, the deeper the Newton stratum into which $C$ generically maps, the more special endomorphisms are required to decay ``sufficiently early'' to first order. Even quantifying what sufficiently ``early'' (i.e., what bounds to put on $d_n(v)$'s) requires a significantly more involved setup. It turns out that the degree of tangency of the curve with the \emph{Artin} strata is what governs the first order decay, and we end up proving that a large-rank isotropic space decays early to first order. Proving this takes up all of Section \ref{sec: first step superspecial} and part of \ref{sec: supersingular}. 


Taking both these steps together, it is possible to obtain a comprehensive description of the lattices $\{ L_r \}_{r \ge 1}$. We find that the decay of special endomorphisms must follow some fascinating stringent patterns, which we were not able to find a satisfactory conceptual explanation. For curious reader, we provide a typical pattern of decay:
\begin{example}\label{exp:Rondo}
   When the point $P$ is superspecial, $L_1\otimes \bZ_p$ admits a (carefully chosen) basis $$\{e',f',e_1,...,e_m,f_1,...,f_m\},$$ such that $\langle e_i,f_j\rangle=1$, $\langle e_i,e_j\rangle=\langle f_i,f_j\rangle=0$ and  $\mathrm{Span}_{\bZ_p}\{e',f'\}$ is orthogonal to $\mathrm{Span}_{\bZ_p}\{e_i,f_i\}_{i=1}^m$. Consider the decaying lattices $L_r\otimes \bZ_p$, $r\geq 1$. The decay obeys the following beautiful pattern (when $r$ increases, the vectors decay in an order following the arrows: $e_1$ first decays, then $e_2$, ..., then $e_m$, then $f_m$, then $f_{m-1}$, \textit{etc}). 
   \begin{center}

\tikzset{every picture/.style={line width=0.75pt}} 

\begin{tikzpicture}[x=0.75pt,y=0.75pt,yscale=-1,xscale=1]

\draw    (104.37,499.93) -- (104.18,542.93) ;
\draw [shift={(104.17,544.93)}, rotate = 270.25] [color={rgb, 255:red, 0; green, 0; blue, 0 }  ][line width=0.75]    (10.93,-3.29) .. controls (6.95,-1.4) and (3.31,-0.3) .. (0,0) .. controls (3.31,0.3) and (6.95,1.4) .. (10.93,3.29)   ;
\draw    (104.05,571.93) -- (103.86,614.93) ;
\draw [shift={(103.85,616.93)}, rotate = 270.25] [color={rgb, 255:red, 0; green, 0; blue, 0 }  ][line width=0.75]    (10.93,-3.29) .. controls (6.95,-1.4) and (3.31,-0.3) .. (0,0) .. controls (3.31,0.3) and (6.95,1.4) .. (10.93,3.29)   ;
\draw    (103.74,642.93) -- (103.55,685.93) ;
\draw [shift={(103.54,687.93)}, rotate = 270.25] [color={rgb, 255:red, 0; green, 0; blue, 0 }  ][line width=0.75]    (10.93,-3.29) .. controls (6.95,-1.4) and (3.31,-0.3) .. (0,0) .. controls (3.31,0.3) and (6.95,1.4) .. (10.93,3.29)   ;
\draw    (164.11,686.24) -- (164.96,644) ;
\draw [shift={(165,642)}, rotate = 91.16] [color={rgb, 255:red, 0; green, 0; blue, 0 }  ][line width=0.75]    (10.93,-3.29) .. controls (6.95,-1.4) and (3.31,-0.3) .. (0,0) .. controls (3.31,0.3) and (6.95,1.4) .. (10.93,3.29)   ;
\draw    (119,698) -- (148,698.94) ;
\draw [shift={(150,699)}, rotate = 181.85] [color={rgb, 255:red, 0; green, 0; blue, 0 }  ][line width=0.75]    (10.93,-3.29) .. controls (6.95,-1.4) and (3.31,-0.3) .. (0,0) .. controls (3.31,0.3) and (6.95,1.4) .. (10.93,3.29)   ;
\draw    (164.23,619.24) -- (164.01,573) ;
\draw [shift={(164,571)}, rotate = 89.73] [color={rgb, 255:red, 0; green, 0; blue, 0 }  ][line width=0.75]    (10.93,-3.29) .. controls (6.95,-1.4) and (3.31,-0.3) .. (0,0) .. controls (3.31,0.3) and (6.95,1.4) .. (10.93,3.29)   ;
\draw    (163.19,545.21) -- (163.01,503) ;
\draw [shift={(163,501)}, rotate = 89.75] [color={rgb, 255:red, 0; green, 0; blue, 0 }  ][line width=0.75]    (10.93,-3.29) .. controls (6.95,-1.4) and (3.31,-0.3) .. (0,0) .. controls (3.31,0.3) and (6.95,1.4) .. (10.93,3.29)   ;
\draw    (175,493) -- (204,493) ;
\draw [shift={(206,493)}, rotate = 180] [color={rgb, 255:red, 0; green, 0; blue, 0 }  ][line width=0.75]    (10.93,-3.29) .. controls (6.95,-1.4) and (3.31,-0.3) .. (0,0) .. controls (3.31,0.3) and (6.95,1.4) .. (10.93,3.29)   ;
\draw    (250,493) -- (280,493) ;
\draw [shift={(282,493)}, rotate = 180] [color={rgb, 255:red, 0; green, 0; blue, 0 }  ][line width=0.75]    (10.93,-3.29) .. controls (6.95,-1.4) and (3.31,-0.3) .. (0,0) .. controls (3.31,0.3) and (6.95,1.4) .. (10.93,3.29)   ;
\draw    (297.37,500.93) -- (297.18,543.93) ;
\draw [shift={(297.17,545.93)}, rotate = 270.25] [color={rgb, 255:red, 0; green, 0; blue, 0 }  ][line width=0.75]    (10.93,-3.29) .. controls (6.95,-1.4) and (3.31,-0.3) .. (0,0) .. controls (3.31,0.3) and (6.95,1.4) .. (10.93,3.29)   ;
\draw    (297.05,572.93) -- (296.86,615.93) ;
\draw [shift={(296.85,617.93)}, rotate = 270.25] [color={rgb, 255:red, 0; green, 0; blue, 0 }  ][line width=0.75]    (10.93,-3.29) .. controls (6.95,-1.4) and (3.31,-0.3) .. (0,0) .. controls (3.31,0.3) and (6.95,1.4) .. (10.93,3.29)   ;
\draw    (296.74,643.93) -- (296.55,686.93) ;
\draw [shift={(296.54,688.93)}, rotate = 270.25] [color={rgb, 255:red, 0; green, 0; blue, 0 }  ][line width=0.75]    (10.93,-3.29) .. controls (6.95,-1.4) and (3.31,-0.3) .. (0,0) .. controls (3.31,0.3) and (6.95,1.4) .. (10.93,3.29)   ;
\draw    (357.11,687.24) -- (357.96,645) ;
\draw [shift={(358,643)}, rotate = 91.16] [color={rgb, 255:red, 0; green, 0; blue, 0 }  ][line width=0.75]    (10.93,-3.29) .. controls (6.95,-1.4) and (3.31,-0.3) .. (0,0) .. controls (3.31,0.3) and (6.95,1.4) .. (10.93,3.29)   ;
\draw    (319,700) -- (338,700) ;
\draw [shift={(340,700)}, rotate = 180] [color={rgb, 255:red, 0; green, 0; blue, 0 }  ][line width=0.75]    (10.93,-3.29) .. controls (6.95,-1.4) and (3.31,-0.3) .. (0,0) .. controls (3.31,0.3) and (6.95,1.4) .. (10.93,3.29)   ;
\draw    (357.23,620.24) -- (357.01,574) ;
\draw [shift={(357,572)}, rotate = 89.73] [color={rgb, 255:red, 0; green, 0; blue, 0 }  ][line width=0.75]    (10.93,-3.29) .. controls (6.95,-1.4) and (3.31,-0.3) .. (0,0) .. controls (3.31,0.3) and (6.95,1.4) .. (10.93,3.29)   ;
\draw    (356.19,546.21) -- (356.01,504) ;
\draw [shift={(356,502)}, rotate = 89.75] [color={rgb, 255:red, 0; green, 0; blue, 0 }  ][line width=0.75]    (10.93,-3.29) .. controls (6.95,-1.4) and (3.31,-0.3) .. (0,0) .. controls (3.31,0.3) and (6.95,1.4) .. (10.93,3.29)   ;
\draw    (376,493) -- (405,493) ;
\draw [shift={(407,493)}, rotate = 180] [color={rgb, 255:red, 0; green, 0; blue, 0 }  ][line width=0.75]    (10.93,-3.29) .. controls (6.95,-1.4) and (3.31,-0.3) .. (0,0) .. controls (3.31,0.3) and (6.95,1.4) .. (10.93,3.29)   ;
\draw    (468,493) -- (498,493) ;
\draw [shift={(500,493)}, rotate = 180] [color={rgb, 255:red, 0; green, 0; blue, 0 }  ][line width=0.75]    (10.93,-3.29) .. controls (6.95,-1.4) and (3.31,-0.3) .. (0,0) .. controls (3.31,0.3) and (6.95,1.4) .. (10.93,3.29)   ;
\draw    (511.37,503.93) -- (511.18,546.93) ;
\draw [shift={(511.17,548.93)}, rotate = 270.25] [color={rgb, 255:red, 0; green, 0; blue, 0 }  ][line width=0.75]    (10.93,-3.29) .. controls (6.95,-1.4) and (3.31,-0.3) .. (0,0) .. controls (3.31,0.3) and (6.95,1.4) .. (10.93,3.29)   ;
\draw    (511.05,575.93) -- (510.86,618.93) ;
\draw [shift={(510.85,620.93)}, rotate = 270.25] [color={rgb, 255:red, 0; green, 0; blue, 0 }  ][line width=0.75]    (10.93,-3.29) .. controls (6.95,-1.4) and (3.31,-0.3) .. (0,0) .. controls (3.31,0.3) and (6.95,1.4) .. (10.93,3.29)   ;
\draw    (510.74,646.93) -- (510.55,689.93) ;
\draw [shift={(510.54,691.93)}, rotate = 270.25] [color={rgb, 255:red, 0; green, 0; blue, 0 }  ][line width=0.75]    (10.93,-3.29) .. controls (6.95,-1.4) and (3.31,-0.3) .. (0,0) .. controls (3.31,0.3) and (6.95,1.4) .. (10.93,3.29)   ;
\draw    (538,700) -- (567,700) ;
\draw [shift={(569,700)}, rotate = 180] [color={rgb, 255:red, 0; green, 0; blue, 0 }  ][line width=0.75]    (10.93,-3.29) .. controls (6.95,-1.4) and (3.31,-0.3) .. (0,0) .. controls (3.31,0.3) and (6.95,1.4) .. (10.93,3.29)   ;

\draw (95.8,488.56) node [anchor=north west][inner sep=0.75pt]  [rotate=0]  {$e_{1} \ $};
\draw (96.88,556.1) node [anchor=north west][inner sep=0.75pt]  [rotate=0]  {$e_{2} \ $};
\draw (106.8,621.79) node [anchor=north west][inner sep=0.75pt]  [rotate=-90]  {$...\ $};
\draw (94.87,692.63) node [anchor=north west][inner sep=0.75pt]  [rotate=0]  {$e_{m} \ $};
\draw (154.31,687.45) node [anchor=north west][inner sep=0.75pt]  [rotate=0]  {$f_{m} \ $};
\draw (153.43,485.16) node [anchor=north west][inner sep=0.75pt]  [rotate=0]  {$f_{1} \ $};
\draw (153.5,552.33) node [anchor=north west][inner sep=0.75pt]  [rotate=0]  {$f_{2} \ $};
\draw (166.28,621.91) node [anchor=north west][inner sep=0.75pt]  [rotate=-90]  {$...\ $};
\draw (211.84,484.03) node [anchor=north west][inner sep=0.75pt]  [rotate=0]  {$e',f'\ $};
\draw (288.8,488.56) node [anchor=north west][inner sep=0.75pt]  [rotate=0]  {$pe_{1} \ $};
\draw (289.88,556.1) node [anchor=north west][inner sep=0.75pt]  [rotate=0]  {$pe_{2} \ $};
\draw (299.8,622.79) node [anchor=north west][inner sep=0.75pt]  [rotate=-90]  {$...\ $};
\draw (287.87,692.63) node [anchor=north west][inner sep=0.75pt]  [rotate=0]  {$pe_{m} \ $};
\draw (347.31,688.45) node [anchor=north west][inner sep=0.75pt]  [rotate=0]  {$pf_{m} \ $};
\draw (346.43,485.16) node [anchor=north west][inner sep=0.75pt]  [rotate=0]  {$pf_{1} \ $};
\draw (346.5,553.33) node [anchor=north west][inner sep=0.75pt]  [rotate=0]  {$pf_{2} \ $};
\draw (358.28,622.91) node [anchor=north west][inner sep=0.75pt]  [rotate=-90]  {$...\ $};
\draw (412.84,484.03) node [anchor=north west][inner sep=0.75pt]  [rotate=0]  {$pe',pf'\ $};
\draw (502.8,483.56) node [anchor=north west][inner sep=0.75pt]  [rotate=0]  {$p^{2} e_{1} \ $};
\draw (503.88,554.1) node [anchor=north west][inner sep=0.75pt]  [rotate=0]  {$p^{2} e_{2} \ $};
\draw (513.8,625.79) node [anchor=north west][inner sep=0.75pt]  [rotate=-90]  {$...\ $};
\draw (501.87,691.63) node [anchor=north west][inner sep=0.75pt]  [rotate=0]  {$p^{2} e_{m} \ $};
\draw (574.87,698.63) node [anchor=north west][inner sep=0.75pt]  [rotate=0]  {$...\ $};

\end{tikzpicture}
   \end{center}
In this particular example, one can also explicitly compute the $n$-th decay rate $d_n(e)$ of a basis vector $e$. These decay rates themselves satisfy intriguing recursive relations, and are closely related to the intersection of $C$ with the Hodge bundle. 

In general situations, the decay may not strictly follow the above pattern. There may exist vectors that never decay, which complicates matters. However, the above pattern always happen for a sufficiently large quotient lattice (say, modulo the sublattice that resist decay). 
\end{example}

\subsubsection{Algebraization} The precise description of decay in Part 2 allows us to find very tight bounds on the local intersection $i_P(C. \cZ(m))$. For most curves $C$, this already yields the desired local-global comparison, hence the denseness of NL locus. The bound only fails when the curve is in an extreme position, where the $p$-divisible group $\mathcal{A}_P[p^\infty]$ admits a ``maximal amount'' of formal special endomorphisms that do not decay (i.e., they deform all the way to the formal neighborhood $C^{/P}$), so $i_P(C. \cZ(m))$ is as large as possible. In fact, when this happens, the main term of the local intersections at supersingular points already equal the main term of the global intersection, hence the local--global comparison breaks down and one cannot conclude that NL locus is dense. 

We note that this is not a byproduct of our decay results being insufficiently strong. Indeed, it is precisely this behaviour that gives rise to the counterexamples with small monodromy discussed in Section \ref{sec: resof scalars}. 

We overcome this structural difficulty by proving that in this extreme case, the connected-étale sequence $\mathbf{\Psi}$ of the enlarged formal Brauer group over the generic point of $C$ must split, cf. section \ref{sec: algebraization}. This is termed as algebraization in the following sense:  let $K(C)$ be the function field of $C$. The existence of this ``maximal amount of'' formal special endomorphisms of $\mathcal{A}_{C^{/P}}[p^\infty]$, which is a $p$-divisible group over a formal base, implies the splitting of $\mathbf{\Psi}$ over $K(C)$, which is a global and algebraic statement. It then follows that
$\cA|_C$ has strictly more endomorphisms than $\cA$, i.e., $\End(\cA|_C) \otimes \bQ\supsetneq \End(\cA) \otimes \bQ$. 

Therefore, if $\mathbf{\Psi}$ does not split over $K(C)$, or if $\End(\cA|_C) \otimes \bQ= \End(\cA) \otimes \bQ$, then we are not in the extreme situation, and the NL locus is dense as we have noted earlier. 


\subsection{Prior and related work}
Our paper fits into the broader context of studying the distribution of algebraic cycles in a family. Over $\bC$, this reduces to studying the \textit{Hodge locus} (at least when assuming the Hodge conjecture), and much recent progress has been made. In \cite{KO21}, \cite{BKU}, and \cite{KU25}, the authors thoroughly studied when the Hodge locus, or a distinguished part of it, is dense. From \cite{KU25}, we typically expect a theorem ensuring density to involve two parts in its assumptions.
\begin{enumerate}[label=\upshape{(\Roman*)}]
    \item The period dimension (i.e., the rank of the Kodaira-Spencer map at a generic point) is large enough for the Hodge locus to have non-negative expected dimension. 
    \item The monodromy group is not ``too degenerate''. 
\end{enumerate}

Note that for a variation of Hodge structures, having infinite monodromy is equivalent to having positive period dimension --- both are equivalent to non-isotriviality. In retrospect, it is somewhat a coincidence that when $h^{2, 0} = 1$, non-isotriviality already ensures a dense NL locus (cf. \cite[Prop.~6.4]{Moo17}). In general, one would expect (I) and (II) to be separate (though related) conditions. In fact, this already happens in positive characteristic even when $h^{2, 0} = 1$: the supersingular locus has positive period dimension and yet finite monodromy. 

In analogy to the complex analytic situation, the hypotheses of our \cref{mainShimura} consist of three parts: Non-isotriviality corresponds to (I). The non-exceptional-ness assumption excludes the obvious obstruction to a dense NL locus coming from Frobenius, while conditions (a), (b), (c) are analogues of (II). In particular, (a) implies the following: if the crystalline monodromy group, (cf. \cite{Crew}) on the open Newton stratum of $C$\footnote{i.e., the group $G(\bL_\cris[p^{-1}]|_U, \bar{\eta})$ if we use the notation of \cite[\S2.2]{DAparabolic}. Here $U$ is the open Newton stratum of $C$, and $\bL_\cris$ is the natural $F$-crystal on $\mathscr{S}$ recalled in \cref{def: special end} below.} does not split as a product over $W(\overline{\bF}_p)[1/p]$, then the NL locus is dense. So (a) is indeed a condition on monodromy. 

Another general theme which our results fit into is the intersection theory on Shimura varieties. During an AIM workshop, several participants (including the second named author) conjectured (\cite[Conj.~5.2]{STexpos}) that for an integral model $\shS$ of a Shimura variety of Hodge type and generically ordinary subvarieties $X, Y \subset \shS_{\bF_p}$ of complementary dimension, the intersection of $X$ with the union of prime-to-$p$ Hecke translates of $Y$ is Zariski dense in $X$. The ordinary Hecke orbit (HO) conjecture (\cite{vH}) is the special case $X = \shS_{\bF_p}$, $Y$ a point; \cite{MST} treats the case where $X$ is a curve in an orthogonal Shimura variety and $Y$ a special divisor; and \cite{AsvinQiaoAnanth} treats the case when $X,Y$ are ordinary curves inside a suitable Hilbert modular surface. These density results have some unexpected geometric applications, e.g., to a good reduction criterion for K3 surfaces (\cite{BY23}). There is more known about this conjecture in characteristic zero - for example, the work \cite{TT} address this question in many cases, including when $X$ is a curve and $Y$ is a divisor, and ongoing work of Baldi-Urbanik addresses this question over the complex numbers whenever $X,Y$ are Hodge-generic in the ambient Shimura variety. 

Beyond the ordinary case, these questions become significantly harder, and even the correct formulation becomes mysterious. For $X = \shS_{\bF_p}$ and $Y$ a point, the full Hecke orbit conjecture (recently proved by van Hoften and D'Addezio \cite{DvH}) roughly says that the obvious Frobenius constraint---that Hecke translates lie in the same central leaf---is the only one. However, as this paper illustrates, to extend \cite{MST} beyond the obvious Frobenius constraint, there are additional conditions whose failure may lead to counterexamples (e.g., \S\ref{sec: resof scalars}). On the other hand, these additional conditions have certain parallels in complex Hodge theory. In \cite{RJ23}, the first named author studied the almost ordinary locus when $b = 3$ (i.e., for abelian surfaces). In that paper, the question considered isn't a Picard-rank jumping question (which would be automatic as $b$ is odd), but the more fine question of showing that the almost-ordinary curve $C$ intersects $\cZ(m)$ where $m$ varies over a fixed square-class. In this paper, we now provide a full generalization for all $b$. In doing so, we hope to further illustrate how the Newton and Artin stratification governs the distribution of algebraic cycles in positive characteristic, and the parallel between the positive characteristic and the complex analytic situations. 

\subsection{Organization of the paper}
In Section \ref{sub:Basics_orthogobal_SV}, we define the various objects in play, including $\textrm{GSpin}$ Shimura varieties, special divisors and endomorphisms, and the Newton stratification. In Section \ref{sec:arithmetic}, we setup the intersection-theoretic framework required to prove our main result, and using results from \cite{SSTT}, \cite{MST} and \cite{Tay22}, we reduce our main theorem to a local statement at supersingular points. In Section \ref{sec: local structure at superspecial point}, we recall work of Kisin and Ogus to describe the universal K3 crystal at a formal neighbourhood of a supersingular point and use this to derive explicit equations for the Newton strata. Sections \ref{sec: decay superspecial} and \ref{sec: first step superspecial} are the technical heart of the paper where we bound local-intersection numbers at superspecial points under a certain local assumption. We generalize this to all supersingular points at \ref{sec: supersingular}. We finally prove the main theorem in Section \ref{sec: algebraization}, where we show that the local assumption always holds under the global hypotheses of our main result. We also have included a glossary with the notation that we use in our paper (especially for Sections \ref{sec: decay superspecial}, \ref{sec: first step superspecial}, and \ref{sec: supersingular}) for the reader's convenience. 

\subsection*{Acknowledgements}
We are very grateful to Davesh Maulik, Arul Shankar, Yunqing Tang, and Salim Tayou for several helpful discussions over the past several years. We are especially grateful to Davesh and Yunqing for invaluable discussions at the start of this project. We also thank David Urbanik for helpful comments. A.S. was partially supported by the NSF grant DMS-2338942, the Institute for Advanced Studies (via the NSF grant DMS-2424441), and a Sloan research fellowship. Z.Y. was supported by the CUHK start up grant, and an ECS grant from the Hong Kong RGC (Project No. 24308225).

\section{Preliminaries on orthogonal Shimura varieties}\label{sub:Basics_orthogobal_SV}

\subsection{Basic setup} \label{subsec: set up SV}
We will set up basic definitions and terminology for our Shimura varieties, special endomorphisms, and special divisors. At first, our exposition and setup will follow \cite[Section 2]{MST} and \cite[Section2]{SSTT}, and we will refer to \cite{AGHMP18} and \cite{MP16}. We will then work modulo $p$ and collect various results (mostly from \cite{Ogu79}, \cite{Ogu01}, and \cite{MST}) that together describe Newton strata. 

Let $L$ be a quadratic lattice with signature $(b,2)$ which is self-dual at $p$. Write $V$ for $L_\bQ$. We may also assume that $L \subset V$ is maximal among those lattices contained in $V$ on which which the bilinear form has integer values. Let $\Cl(L)$ denote the Clifford algebra associated to $L$. Note that we have a natural embedding of free $\bZ_{(p)}$-modules $L\otimes \bZ_{(p)} \hookrightarrow \Cl(L)\otimes \bZ_{(p)}$. Let $G:= \mathrm{GSpin}(L\otimes \bZ_{(p)})$ be the group of spinor similitudes. Note that $G$ is a reductive group over $\bZ_{(p)}$ and is naturally a subgroup of $\Cl(L \otimes \bZ_{(p)})^{\times}$.  Let $\cD $ be the Hermitian symmetric domain $\{ z \in V_\bC | \langle z, \overline{z} \rangle < 0 \}$. Then $(G, \cD)$ defines a spinor Shimura datum with reflex field $\bQ$. For every neat compact open subgroup $\bK \subseteq G(\bA_f)$, we obtain a Shimura variety $\Sh_{\bK}(G, \cD)$ over $\bQ$. 

Set $H = \mathrm{Cl}(L)$, viewed as a $\mathrm{Cl}(L)$-bimodule. Note that $H$ has a natural $\bZ/2 \bZ$-grading. Left multiplication of $\Cl(L)$ on $H$ the induces spin representation $G \to \mathrm{GL}(H)$ and an embedding $L \into \Cl(L) \into \End (H)$. Equip $\End (H)$ with the pairing $(\alpha, \beta) := 2^{-\mathrm{rank\,} L} \mathrm{tr}(\alpha \circ \beta)$. Then the composite embedding $L \into \End(H)$ is isometric and we can form the orthogonal projection $\psi : \End(H) \to L$. The spinor group $G$ can be conversely viewed as the stabilizer of the $\bZ / 2 \bZ$-grading, the right $\Cl(L\otimes \bZ_{(p)})$-action and the idempotent projector $\psi$.

Suppose now that $\bK$ is of the form $\bK_p \bK^p$ for $\bK^p \subseteq G(\bA^p_f)$ and $\bK_p = G(\bZ_p)$, i.e., is hyperspecial at $p$. 
By \cite{Kis10} (cf. \cite{MP16}) there is an canonical integral model $\shS_{\bK}(G)$ over $\bZ_{(p)}$. 
We may endow $H$ with a suitable symplectic form such that the left multiplication by $G$ on $H$ respects the form up to scaling, and indeed induces an embedding of Shimura data $(G, \cD) \into (\mathrm{GSp}, \mathcal{H}^\pm)$, where $\mathcal{H}^\pm$ is the associated Siegel half spaces. This Siegel embedding equips $\shS_{\bK}(G)$ with a universal abelian scheme $\cA$, called the \textit{Kuga-Satake abelian scheme}.\footnote{Technically, in \cite{Kis10} and \cite{MP16}, $\cA$ is only defined as a sheaf of abelian schemes up to prime-to-$p$ quasi-isogeny. 
However, for $\bK^p$ sufficiently small, we can take $\cA$ to be an actual abelian scheme (cf. \cite[(2.1.5)]{Kis10}).}  We use $\cA[p^\infty]$ to denote the $p$-divisible group associated to $\cA$. Throughout, we fix a choice of such level structure, and for brevity let $\Sh$ (resp. $\shS$) denote $\Sh_{\bK}(G,\cD)$ (resp. the integral model $\shS_{\bK}(G,\cD)$).

Define the sheaves $\bH_\mathrm{B}/\Sh_{\bC}$, $\bH_\ell/\shS $ ($\ell \neq p$), $\bH_p/\Sh$, $\bH_\dR/\shS$ and $\bH_\cris /\shS_{\bF_p}$ denote the (first relative) Betti cohomology, $\ell$-adic etale cohomology, $p$-adic etale cohomology, de Rham cohomology, and the crystalline cohomology of the universal abelian scheme. 
The abelian scheme $\cA$ is equipped with a ``CSpin-structure'': a $\bZ / 2 \bZ$-grading, $\mathrm{Cl}(L)$-action and an idempotent projector $\psi_? : \End(\bH_?) \to \End(\bH_?)$ for $?\in \{B, \ell, p, \dR, \cris\}$ on (various applicable fibers of) $\shS$. We use $\bL_?$ to denote the images of $\psi_?$. 

\begin{remark}\label{rmk:embwedge}
The embedding $L\subseteq \End_{\Cl(L)}(H)$ can be generalized to some extent. For every $r\geq 0$, one can construct a natural map $\wedge^r V \to \Cl(V)$ by sending $v_1 \wedge \cdots \wedge v_r$ to $\sum_\sigma \mathrm{sgn}(\sigma) v_{\sigma(1)} v_{\sigma(2)} \cdots v_{\sigma(r)}$.  
This induces an isomorphism between $\bQ$-spaces  $\wedge^{\bullet} V$ and $\mathrm{Cl}(V)$. The natural action of $\wedge^\bullet V$ on $H_{\bQ}$ results in a $G$-invariant embedding $\wedge^\bullet V\hookrightarrow \End_{\mathrm{Cl}(V)}(H_{\bQ})$. As a consequence, there is a natural embedding of rational cohomology sheaves 
 $\wedge^{\bullet} \mathbb{L}_?\otimes\bQ\hookrightarrow \mathrm{End}(\mathbb{H}_?\otimes \bQ)$.  
\end{remark}

\begin{definition}\label{def: special end} We now use the sheaves $\bL_?$ to define the notions of special endomorphisms of points of $\shS$ and special divisors in $\shS$.
    \begin{enumerate}
        \item Given any $\shS$-scheme $T$, an endomorphism $f \in \End(\cA_T)$ is called a \emph{special endomorphism} (\cite[Def.~5.2, see also Lem.~5.4, Cor.~5.22]{MP16}) if all cohomological realizations of $f$ lie in the image $\bL_? \rightarrow \End(\bH_?)$ where $? = \mathrm{B}, \dR, \cris, \ell \neq p, p$. For brevity, we simply drop those subscripts that do not make sense.  For example, if $p$ is invertible in $T$, then $? = \cris$ doesn't make sense and if $p$ is not invertible then $? = p,\mathrm{B}$ do not make sense. 
    
    We denote the submodule of $\End(\cA_T)$ of special endomorphisms by $L(\cA_T)$. By \cite[Lem. 5.2]{MP16}, for $v\in L(\cA_T)$, we have $v\circ v=[Q(v)]$ for
some $Q(v) \in \bZ_{\geq 0}$,  and $Q(v)$ is a positive definite quadratic form on the $\bZ$-lattice $L(\cA_T)$.
\item For $m\in \bZ_{>0}$, the \textit{special divisor} $\mathcal{Z}(m)$ is the Deligne–Mumford stack over $\mathscr{S}$ with
functor of points $\mathcal{Z}(m)(T) = \{v\in L(\cA_T)\, |\,Q(v) = m\}$ for any $\mathscr{S}$-scheme $T$. We use the
same notation for the image of $\mathcal{Z}(m)$ in $\mathscr{S}$. By  \cite[Prop. 4.5.8]{AGHMP18}, $\mathcal{Z}(m)$ is an
effective Cartier divisor flat over $\bZ_{(p)}$ and hence $\mathcal{Z}(m)_{\bF_p}$
is still an effective Cartier divisor of
$\mathscr{S}_{\mathbb{F}_p}$. We denote $\mathcal{Z}(m)_{\bF_p}$ by $Z(m)$.

    \item Suppose that $p$ is not invertible in $T$. We say that $f\in \End(\cA_T[p^{\infty}])$ is a \emph{formal special endomorphism} if its crystalline realization lies in $\bL_{\cris}$. We denote the $\bZ_p$-submodule of $\End(\cA_T[p^{\infty}])$ of formal special endomorphisms by $\cL(\cA_T)$.
\end{enumerate}
     \end{definition}

\begin{remark}
    Consider a connected scheme $T$ and a $T$-valued point of $\shS$. An endomorphism $f\in \End(\cA_T)$ or $\End(\cA_T[p^{\infty}])$ being special can be checked at any geometric point of $T$ (\cite[Proposition 4.3.4]{AGHMP18}). Further if $x\in T_{\bF_p}$ is a geometric point, then $f$ is special if and only if $f_{t,\cris} \in \bL_{\cris}$. 
\end{remark}

\subsection{Newton stratification}
 Let $k$ be any perfect field of characteristic $p \ge 5$ and let $W = W(k)$. We describe the Newton stratification on $\shS_k$ following the scheme-theoretic approach of \cite{Ogu01}. In order to extend the results there without referring to K3 surfaces, we can use the language of orthogonal $F$-zips. This extends the data of Hodge and conjugation filtrations, as well as the Frobenius action on the de Rham cohomology of K3 surfaces. 

For any closed subscheme $T$ of $\shS_k$, we write $\wt{T}$ for the ($p$-complete) PD-envelope of $T$ in $\shS_W$. When $\wt{T}$ is $p$-torsion-free, we say that $T$ is \textit{admissible}, and the restriction $(\bL_\cris|_T, p \Frob)$ is a K3 crystal in the terminology of Nygaard--Ogus (\cite[Def.~5.1]{NO}), ignoring the rank $22$ condition. Recall that a local complete intersection is admissible (\cite[\S2.3.3]{BBM82}). This condition allows us to view $\bL_\cris(\wt{T})$ as a submodule of $\bL_\cris(\wt{T})[1/p]$, on which $\Frob$ natually acts. 

The de Rham cohomology $\bH_\dR$ is equipped with a Hodge and conjugate filtration, these structures are inherited by the direct summand $\bL_\dR$ of $\End(\bH_\dR)$. Note that there is a canonical (reduction mod $p$) map $\bL_\cris(\tilde{T}) \rightarrow \bL_{\dR,T}$, which we will denote in this section by $\pi$. By Mazur--Ogus inequality \cite[Thm.~8.26]{BOBook}, one can recover these structures using the Frobenius structure on $\bL_{\cris}|_T$: 
\begin{align}
    \Fil^i \bL_{\dR, T} &= \pi\left(\Frob^{-1}(p^{i} \bL_{\cris}(\wt{T}))\right) \label{eqn: Mazur Ogus Hodge} \\
    \Fil_c^{i} \bL_{\dR,T} &= \pi\left(p^{-i} \left[\Frob(\bL_\cris(\wt{T})) \cap p^i  \bL_\cris(\wt{T})\right]\right)\label{eqn: Mazur Ogus conjugate}
\end{align}
The first equation breaks into two 
\begin{align}
    \pi^{-1} (\Fil^0 \bL_{\dR, T}) &= \{ \wt{x} \in \bL_\cris(\wt{T}) \mid \Frob(\wt{x}) \in \bL_\cris(\wt{T}) \} \\
    \pi^{-1}(\Fil^1 \bL_{\dR, T}) &\supsetneq \{ \wt{x} \in \bL_\cris(\wt{T}) \mid \Frob(\wt{x}) \in p \bL_\cris(\wt{T}) \}.
\end{align}
Note that the first is an equality because $\Frob(\bL_\cris) \in p^{-1} \bL_\cris$, but the second is not. Technically, \cite[Thm.~8.26]{BOBook} applies to a smooth base. So in order to obtain (\ref{eqn: Mazur Ogus Hodge}) and (\ref{eqn: Mazur Ogus conjugate}) for arbitrary $T$ (with $p$-torsion free $\wt{T}$), one can use that for any closed embedding $T \into T'$, there are commutative diagrams
\begin{equation}
\label{eqn: embed of PD envelopes}
    \begin{tikzcd}
	T & {T'} \\
	{\wt{T}} & {\wt{T}'}
	\arrow[hook, from=1-1, to=1-2]
	\arrow[hook, from=1-1, to=2-1]
	\arrow[hook, from=1-2, to=2-2]
	\arrow[from=2-1, to=2-2]
\end{tikzcd} \Rightarrow
\begin{tikzcd}
	{\bL_\cris(\wt{T}')} & {\bL_\cris(\wt{T})} \\
	{\bL_{\dR, T'}} & {\bL_{\dR, T}}
	\arrow["{\tensor_{\sO_{\wt{T}'}} \sO_{\wt{T}}}", from=1-1, to=1-2]
	\arrow["\pi"', two heads, from=1-1, to=2-1]
	\arrow["\pi"', two heads, from=1-2, to=2-2]
	\arrow["{\tensor_{\sO_{T'}} \sO_T}", from=2-1, to=2-2]
\end{tikzcd}
\end{equation}

We will often implicitly make use of such diagrams. Over $T = \shS_k$ (and hence any $T$), the $F$-zip $\bL_{\dR, T}$ is equipped with two $\sO_T$-linear isomorphisms 
\begin{align}
    \label{eqn: gamma -1} \gamma_{-1} : (F_T)^* (\gr^{-1} \bL_{\dR, T}) \stackrel{\sim}{\to} \Fil^{-1}_c \bL_{\dR, T} \text{ and }
    \gamma_0 : (F_T)^* (\gr^0 \bL_{\dR, T}) \stackrel{\sim}{\to} \gr^{0}_c \bL_{\dR, T}. 
\end{align}
Also, note that $\gr^{-1} \bL_{\dR, T} = \pi(p \Frob(\bL_\cris(\wt{T})))$.

Before we proceed, we briefly recall some basic facts abour K3 crystals over a point (\cite[p.326-328]{Ogu01}). 

\begin{definition}
\begin{enumerate}
    \item A K3 crystal over $k$ is a finite free $W$-module $H$ equipped with a pairing $\langle - , - \rangle$ and a $\sigma$-linear Frobenius automorphism $\Frob$ on $H[1/p]$ such that $p \Frob(H) \subseteq H$ has rank $1$ quotient and $\langle  \Frob(x), \Frob(y) \rangle = \sigma(\langle x, y \rangle)$. 
    \item Define Hodge and conjugation filtrations on $H_k$ using (\ref{eqn: Mazur Ogus Hodge}) and (\ref{eqn: Mazur Ogus conjugate}). 
    \item Define a sequence of subspaces $E_1 \subseteq E_2 \subseteq \cdots \subseteq \Fil^0_c H_k$ as follows. Define $E_1 = \Fil^{-1}_c H_k$ and for each $i \in \bN$, define $E_{i+1}$ recursively as follows: 
    \begin{enumerate}
        \item  if $E_i \subseteq \Fil^0 H_k$, define $E_{i + 1}:=\varpi^{-1}(\gamma_0(E_i))$, where $\varpi$ is the projection $\Fil^0_c H_k \to \gr^0_c H_k$. 
        \item  if $E_i \not\subseteq \Fil^0 H_k$, we stop (meaning that $E_{i+1}$ will not be defined). 
    \end{enumerate}
    \item The height of $(H, \Frob)$ is the maximal $h \in \mathbb{N} $ such that $E_{h }$ is well-defined. If $h=\infty$, then we say that $(H, \Frob)$ is \textit{supersingular}.
\end{enumerate}
\end{definition}
Our definition differs from that in \cite{Ogu01} up to a Tate twist. Namely, we divide the Frobenius operator in the original definition by $p$. Moreover, our conjugation filtration is ascending, which is compatible with the convention for F-zips but is opposite to the convention used in \cite{Ogu01}. In particular, if $\bL_{\dR, T}$ comes from (the primitive part of) the de Rham cohomology of a K3 surface $X$ over $T$ as in \cite{Ogu01}, then our $\Fil^{-1}_c \bL_{\dR, T}$ corresponds to $\Fil^2_c H^2_\dR(X/T)$.   

\begin{lemma}\label{lem: dim E_i} 
Let $(H, \Frob)$ be a K3 crystal over a point $k$. 
    \begin{enumerate}[label=\upshape{(\alph*)}]
        \item $(H, \Frob)$ has finite height $h$ if and only if its slopes contain $\{ - 1/h, 1/ h \}$, if and only if $h$ is the maximal number such that $p \Frob^i(H) \subseteq H$ for every $i \le h$. 
        \item $(H, \Frob)$ is supersingular if and only if its Newton polygon has constant slope $0$, if and only if $p \Frob^i(H) \subseteq H$ for every $i \in \mathbb{N}$.
        \item If $(H, \Frob)$ has finite height, then for every $i \le h$, $\dim E_i = i$ and $\Fil^1 H_k \not\subseteq E_i$. 
        \item If $(H, \Frob)$ is supersingular and has Artin invariant $\sigma_0$, then $\dim E_i = i$ for every $i \le \sigma_0$ and $E_i = E_{\sigma_0}$ for every $i \ge \sigma_0$; moreover, $E_{\sigma_0} = E_{\sigma_0 - 1} + \Fil^1 H_k$ and $\Fil^1 H_k \not\subseteq E_i$ for $i < \sigma_0$. 
    \end{enumerate}
\end{lemma}

Recall that the \textit{Artin invariant} $\sigma_0$ of the supersingular K3 crystal $(H, \Frob)$ is the integer such that the Tate module $H^{\Frob = 1}$ has discriminant isomorphic to $(\bZ/ p \bZ)^{2 \sigma_0}$. 


\begin{proof}
    Every statement is extracted directly from \cite[Prop.~3, Lem.~5]{Ogu01}, except the $\Fil^1 H_k \not\subseteq E_i$ statement in (c) and (d). The former can be easily deduced from Lem.~4 and Lem.~5(1) of \textit{loc. cit.}, and the latter holds for $i < \sigma_0$ because $E_{\sigma_0} \neq E_{\sigma_0 - 1}$ for dimension reasons. 
\end{proof}

Now, following \cite[p.~334]{Ogu01}, we define a sequence of closed subschemes $\sF_1 = \shS_k \supseteq \sF_2 \supseteq $ of $\shS_k$ as follows: 

\begin{definition}
\label{def: F_h and E_h}
    Define $E_1(T) := \Fil^{-1}_c \bL_{\dR, T}$. For any $\shS_k$-scheme $T$ that factors through $\sF_h$, $T$ factors through $\sF_{h + 1}$ if and only if $E_h(T) \subseteq \Fil^0 \bL_{\dR, T}$. If this happens, then we define a coherent subsheaf $E_{h + 1}(T)$ of $\Fil^0_c \bL_{\dR, T}$ by $\varpi^{-1}(\gamma_0(E_h(T)))$, where $\varpi$ denotes the projection $\Fil^0_c \bL_\dR \to \gr^0_c \bL_\dR$.
\end{definition}




Meanwhile, let us consider also the naive definition of Newton stratification which is customary in the literature on Shimura varieties. Namely, we simply define $\shS_h$ to be the closed subset consisting of points of height $\ge h$, and then endow it with the reduced subscheme structure (see e.g., \cite[\S2]{SZ22}). We write $\shS_{\infty}$ for the supersingular (or basic) locus on $\shS_k$. We recall the following dimension formula on $\shS_h$'s from Howard--Pappas \cite[Thm.~C]{HP17} (see also Shen--Zhang \cite[Cor.~7.3.4]{SZ22}).

\begin{proposition} For the lattice $L$ of rank $b + 2$, we have the following. 
    \label{prop: Shen-Zhang} 
    \begin{enumerate}
        \item If $b = 2m - 1$ is odd, then for $h = 1, \cdots, m + 1$, $\dim \shS_h = b + 1 - h$; the supersingular locus $\shS_\infty = \shS_{m + 1}$ has dimension $m - 1$. 
        \item If $b = 2m$ is even and $\det(V_{\bQ_p}) = (-1)^{m}$, then for $h = 1, \cdots, m + 1$, $\dim \shS_h = b + 1 - h$; the supersingular locus $\shS_\infty = \shS_{m + 1}$ has dimension $m$. 
        \item If $b = 2m$ is even and $\det(V_{\bQ_p}) \neq (-1)^{m}$, then for $h = 1, \cdots, m + 2$, $\dim \shS_h = b + 1 - h$; the supersingular locus $\shS_\infty = \shS_{m + 2}$ has dimension $m - 1$. 
    \end{enumerate}
    In all cases, a general point on a non-supersingular $\shS_h$ has height exactly $h$, i.e., $\shS_h \neq \shS_{h - 1}$. 
\end{proposition}

Within $\shS_\infty$, let us write $\shS_{\infty, \sigma_0}$ for the closed subscheme of points $s$ where $\sigma_0(s) := \sigma_0(\bL_{\cris, s}, p \Frob_s)$ has Artin-invariant $\le \sigma_0$. We recall that:
\begin{proposition}
\label{prop: dim artin strata}
    $\dim \shS_{\infty, \sigma_0} = \sigma_0 - 1$.
\end{proposition}
\begin{proof}
     This is essentially due to Ogus \cite[Prop.~4.6.1]{Ogu79} (see also \cite[Prop.~5.3.2]{HP17}, the relation to the supersingular locus of orthogonal Shimura varieties is explained in \S7.3 of \textit{loc. cit.}). 
\end{proof}


Most of the arguments and results in \cite{Ogu01} hold verbatim in the context of the general orthogonal Shimura varieties, except for Corollary 16 and Theorem~19, which only extend to the case (1). We shall need a statement for case (2). Therefore, below we rephrase and summarize the key inputs in \cite{Ogu01} which hold regardless of the three cases in our general context. 
\begin{lemma}
\label{lem: Ogus lem 9}
     Suppose that a $\shS_k$-scheme $T$ factors through $\sF_h$, and let $i \le h$ be a positive integer. 
    \begin{enumerate}[label=\upshape{(\alph*)}]
        \item Each $E_i(T)$ is a coherent totally isotropic subsheaf of $\bL_{\dR, T}$ locally generated by $i$ elements, and $E_0(T) \subseteq E_1(T) \subseteq \cdots \subseteq E_h(T)$. In fact, each $E_{i}(T) / E_{i - 1}(T)$ is locally monogenic. 
        \item If $g : T' \to T$ is a morphism of $\shS_k$-schemes, then the natural isomorphism $g^* \bL_{\dR, T} \to \bL_{\dR, T'}$ induces a surjection $g^* E_h(T) \twoheadrightarrow  E_h(T')$, and an isomorphism $g^* Q_i(T) \stackrel{\sim}{\to} Q_i(T')$, where $Q_i(T) := \bL_{\dR, T} / E_i(T)$.  
        \item Suppose that $T$ is of finite type over $k$ and $T \times_{\shS_k} \shS_{\infty, i - 1} = \emptyset$. Then $E_i(T)$ is a local direct summand of $\bL_{\dR, T}$ and its formation commutes with base change.
    \end{enumerate}
\end{lemma}

The reader may be concerned that in \cref{def: F_h and E_h}, $F_h$ is a priori only defined as a subfunctor for the functor of the points of $\shS_k$. That they are indeed representable by closed subschemes is justified by the surjectivity statement of (b) above (cf. \cite[Prop.~11]{Ogu01}). Indeed, suppose that this is true for $h$. Then $\sF_{h + 1}$ is the closed subscheme of $\sF_h$ defined by the ideal given by the image of the pairing 
\begin{equation}
\label{eqn: ideal of F h + 1}
    \Fil^1 \bL_{\dR, \sF_h} \tensor \left( E_h(\sF_h) / E_{h - 1}(\sF_h) \right) \to \sO_{\sF_h}. 
\end{equation}
Suppose that $\sF'_{h + 1}$ is defined by this ideal. Then a $\sF_h$-scheme $T$ factors through $\sF'_{h + 1}$ if and only if $E_h(\sF_h) |_T \subseteq \Fil^0 \bL_{\dR, T}$. But there is a commutative diagram 
\[\begin{tikzcd}
	{E_h(\sF_h)|_T} & {E(T)} \\
	{(\gr^{-1} \bL_{\dR, \sF_h})|_T} & {\gr^{-1} \bL_{\dR, T}}
	\arrow[two heads, from=1-1, to=1-2]
	\arrow[from=1-1, to=2-1]
	\arrow[from=1-2, to=2-2]
	\arrow["\sim", from=2-1, to=2-2], 
\end{tikzcd}\]
so the left vertical arrow vanishes if and only if the right one does. Hence $\sF'_{h + 1} = \sF_{h + 1}$. In fact, as $E_{h} / E_{h -1}$ is locally monogenic for $\sF_h$, (\ref{eqn: ideal of F h + 1}) implies that $\sF_{h + 1}$ is locally cut out by a single equation from $\sF_h$. Next, we recall Ogus' characterization of the tangent spaces of $\sF_h$.  
\begin{proposition}
\label{prop: dim of tangent space}
\emph{(\cite[Cor.~13]{Ogu01})}
The tangent space of $\sF_h$ at a closed point $s$ is naturally isomorphic to 
        \begin{equation}
        \gr^{-1} \bL_{\dR, s} \tensor \overline{E}_{h - 1}(s)^\perp \subseteq \gr^{-1} \bL_{\dR, s} \tensor \gr^0 \bL_{\dR, s} = T_s \shS_k, 
    \end{equation}
    where $\overline{E}_{i}(s)$ denotes the image of $E_{i}(s)$ in $\gr^0 \bL_{\dR, s}$. 
\end{proposition}

Now we can extend \cite[Cor.~16, Thm.~19]{Ogu01} as follows. 

\begin{corollary}
\label{cor:gen-nonteducdc} 
    In the three cases of \cref{prop: Shen-Zhang}, we have the following. 
    \begin{enumerate}
        \item When $b = 2m - 1$ is odd, for every $h \le m$, $\sF_h$ is reduced and hence is equal to $\shS_h$, $\sF_{m + 1}$ is genercially nonreduced and $\shS_\infty = (\sF_{m + 1})_{\mathrm{red}}$.
        \item When $b = 2m$ is even and $\det(V_{\bQ_p}) =  (-1)^{m}$, for every $h \le m + 1$, $\sF_h$ is reduced and hence is equal to $\shS_h$, and $\sF_{m + 1} = \shS_\infty$. Moreover, for every $h \ge m + 1$, $(\sF_h, E_h) = (\sF_{m + 1}, E_{m + 1})$.
        \item When $b = 2m$ is even and $\det(V_{\bQ_p}) \neq (-1)^{m}$, for every $h \le m + 1$, $\sF_h$ is reduced and hence is equal to $\shS_h$, $\sF_{m + 2}$ is generically non-reduced and $\shS_\infty = (\sF_{m + 2})_{\mathrm{red}}$.
    \end{enumerate}
\end{corollary}


\begin{proof}
    All three statements follow from some dimension computations. Define $m' = m + 1$ in cases (1) and (2) and $m + 2$ in case (3). 

    First, assume that $h < m'$. Take a point $s \in \sF_h$ of height exactly $h$, which is in particular non-supersingular. By \cref{lem: dim E_i}(c), $\Fil^1 \bL_{\dR, s} \not\subseteq E_{h}(s)$, and hence $\dim \overline{E}_{h - 1}(s) = \dim E_{h - 1}(s) = h - 1$. By \cref{prop: dim of tangent space}, $\dim T_s \sF_h = \dim \overline{E}_{h - 1}(s)^\perp = b + 1 - h$, which is equal to $\dim \sF_h$. This implies that $\sF_h$ is smooth at $s$. In particular, $\sF_h$ is generically reduced. However, since any local complete intersection is Cohen-Macaulay, generic reducedness implies reducedness. Hence $\sF_h$ is reduced and is equal to $\shS_h$.

    Now assume that $h = m'$. In case (1), $\dim \shS_\infty = m - 1$, so for a general $s \in \shS_\infty$, $\sigma_0(s) = m$. Since a finite height point on $\shS_k$ has height at most $m$, we already know $(\sF_{m + 1})_{\mathrm{red}} = \shS_\infty$. By \cref{lem: dim E_i}(d), $\Fil^1 \bL_{\dR, s} \subseteq E_{h}(s)$, and hence $\dim \overline{E}_{m}(s) = \dim E_{m}(s) - 1 = m - 1$. By \cref{prop: dim of tangent space}, $\dim T_s \sF_{m + 1} = \dim \overline{E}_{m}(s)^\perp = m > \dim \sF_{m + 1}$. Therefore, $\sF_{m + 1}$ is generically singular. The argument for cases (2) and (3) are entirely similar. What happens for case (2) is that, for a general point $s \in \shS_{\infty}$, $\sigma_0(s) = m + 1$. But from \cref{lem: dim E_i}(d) we infer that $\Fil^1 \bL_{\dR, s} \not\subseteq E_{m}(s)$, so the same argument as in case (1) implies that $\dim T_s \sF_{m + 1} = m = \dim \sF_{m + 1}$, so that $\sF_{m + 1}$ is reduced. We remark that case (3) is analogous to case (1), but $\dim T_s \sF_{m + 2} = m + 1$, which is greater than $\dim \sF_{m - 2}$ by $2$, so $\sF_{m'}$ in case (3) is more nonreduced than that in case (1).   
    
    Finally, we show that in case (2), $(\sF_h, E_h)$ stablizes after $h \ge m + 1$. Let us note first that $\sF_{h} = \sF_{m + 1}$ because $\sF_{h}$ is also supported on $\shS_\infty$, but $\sF_{m + 1}$ is already reduced. Consider $\sF'_{m + 1} := \sF_{m + 1} \smallsetminus \shS_{\infty, m}$. By \cref{lem: Ogus lem 9}(c), the formation of $E_{m + 1}$ over $\sF'_{m + 1}$ commutes with base change, and $E_{m + 1}(\sF'_{m + 1})$ is a direct summand of $\bL_{\dR, \sF'_{m + 1}}$. This implies that its orthogonal complement $E_{m + 1}^\perp := \ker(\bL_{\dR} \to E^\vee_{m + 1})$ is also a direct summand, and its formation also commutes with base change. Now, for every closed point $s \in \sF'_{m + 1}$, $E_{m + 1, s}$ is an isotropic subspace of $\bL_{\dR, s}$ of exactly half the dimension, so that $E_{m + 1, s} = E^\perp_{m + 1, s}$. As $E_{m + 1}$ itself is totally isotropic, $E_{m + 1} \subseteq E_{m + 1}^\perp$ and hence they are equal. Since $E_{m + 2}$ on $\sF_{m + 2} = \sF_{m + 1}$ is a totally isotropic subsheaf and contains $E_{m + 1}$, it is sandwiched by $E_{m + 1}$ and $ E_{m + 1}^\perp$, so the claims follows by induction.
\end{proof}

In the next lemma, we extend \cite[Lem.~2, Prop.~3(2)]{Ogu01} to a general base.\footnote{Technically, the lemma below uses that each $\sF_h$ is a local complete intersection. One easily checks that this holds for all cases in \cref{prop: Shen-Zhang}, though later we only make use of case (2). } Recall that $\wt{T}$ denotes the PD-envelope in $\shS_W$ for a closed $\shS_k$-scheme $T$. 
\begin{lemma}\label{lem: Frob and E} For every $h \in \bN$, 
\begin{enumerate}[label=\upshape{(\alph*)}]
    \item $p \Frob^j$ is integral on $\bL_\cris(\wt{\sF}_h)$ for every $j \le h$; 
    \item $E_h = \pi(p \Frob^h(\bL_\cris(\wt{\sF}_h))) + E_{h - 1}$ over $\sF_h$; and
    \item each $\sF_{h + 1}$ is cut out from $\sF_h$ by a section of the line bundle $\omega^{p^h - 1} |_{\sF_h}$, where $\omega$ is the Hodge bundle $\Fil^1 \bL_\dR$. 
\end{enumerate} 
\end{lemma}
\begin{proof}
    Both parts (a) and (b) hold for $h = 1$. For (c), we may assume that $T = \shS_{\bF_p}$. Note that the non-ordinary locus $\sF_2$ is cut out by the condition that $p \Frob$ kills $\gr^{-1} \bL_\dR$, which has already been observed in \cite[Lem.~4.9]{MST} (it is unnecessary to assume that the point $P$ in \textit{loc. cit.} is supersingular). Let $F_T$ denote the absolute Frobenius morphism on $T$. Then $p\Frob$ induces a linear map $F_T^*(\gr^{-1} \bL_\dR) \to \gr^{-1} \bL_\dR$. But $\gr^{-1} \bL_\dR$ is dual to $\omega$ via the pairing on $\bL_\dR$, so the conclusion follows.

    Now we proceed by induction. Since over $\sF_h$ we already have $E_{h - 1} \subseteq \Fil^0 \bL_\dR$, using that (b) is true for $i = h$, we have that by (\ref{eqn: Mazur Ogus Hodge}) the condition $E_{h} \subseteq \Fil^0 \bL_\dR$ over $\sF_{h + 1}$ is equivalent to 
\begin{equation}
    \label{eqn: apply Frob to F h + 1}
    \Frob(p \Frob^h(\bL_\cris(\wt{\sF}_{h + 1}))) = p \Frob^{h + 1}(\bL_\cris(\wt{\sF}_{h + 1})) \subseteq \bL_\cris(\wt{\sF}_{h + 1}).
\end{equation}
This proves part (a) for $h + 1$. This also proves (c). Indeed, it is easy to see that (\ref{eqn: apply Frob to F h + 1}) is equivalent to $p \Frob^{h + 1}$ kills $\gr^{-1} \bL_\dR$, so we conclude by the same argument as in the preceeding paragraph.   

Now let us prove part (b). Let us set $$E^+_{h + 1} := \pi(p \Frob^{h + 1}(\bL_\cris(\wt{\sF}_{h + 1})))$$ and $E_{h + 1}' = E^+_{h + 1} + E_{h}$. Our goal is to prove that $E_{h + 1} = E'_{h + 1}$. First, we note that $E_{h + 1}^+ $ (and hence also $E'_{h + 1}$) lies in $ \Fil^0_c$ over $\sF_{h + 1}$. By (\ref{eqn: Mazur Ogus conjugate}), it suffices to show that $$p \Frob^{h + 1}(\bL_\cris(\wt{\sF}_{h + 1})) \subseteq \Frob(\bL_\cris(\wt{\sF}_{h + 1})) \cap \bL_\cris(\wt{\sF}_{h + 1}). $$ 
This is true by (\ref{eqn: apply Frob to F h + 1}) and the fact that $p \Frob^h(\bL_\cris(\wt{\sF}_{h + 1})) \subseteq \bL_\cris(\wt{\sF}_{h + 1})$. 
Now we have that over $\sF_{h + 1}$, $E_{h + 1}'$ and $E_{h + 1}$ are both submodules of $\Fil^0_c \bL_\dR$ which contains $E_1 = \Fil^1 \bL_\dR$. Therefore, to check that they are equal it suffices to check that they are equal modulo $E_1$. 

By the inductive hypothesis on $E_{h}$, for every $x \in E_h$, there exist $\wt{x}_1, \wt{x}_2, \cdots, \wt{x}_{h} \in \bL_\cris(\wt{\sF}_{h + 1})$ and $x = \pi(\sum_{\ell=1}^h p \Frob^{\ell}(\wt{x}_\ell))$.
By the definition of $\gamma_0$, its image in $\gr^0_c$ is given by 
$$ \Frob(\sum_{\ell=1}^h p \Frob^{\ell}(\wt{x}_\ell)) = \sum_{\ell=1}^h p \Frob^{\ell + 1}(\wt{x}_\ell). $$
modulo $\Fil^1_c$. The conclusion easily follows.
\end{proof}

\begin{remark}
    Ogus' theory in \cite{Ogu01} has been generalized by Koskivirta, La Porta and Reppen in \cite{KLR25}. Much of the statements about even dimensional orthogonal Shimura varieties (corresponding to the $D_n$-types) above should also follow from the general theory of \textit{loc. cit}. However, for our purposes it is easier to simply extend Ogus' original approach. 
\end{remark}

\section{Arithmetic intersection theory}\label{sec:arithmetic}

In this section, we will set up the intersection-theoretic framework required to prove Theorem \ref{mainShimura}. We will also reduce the proof of Theorem \ref{mainShimura} to Theorem \ref{thm: main local version}, which is a purely local statement at supersingular points. 

\subsection{Global intersection}  We let $L$ as in Section \ref{sub:Basics_orthogobal_SV}.

\subsubsection{The global Eisenstein series}\label{subsub:gEs} Let $L$ be as in Section \ref{sub:Basics_orthogobal_SV}. Set $U = \bC[L^\vee / L]$. Let $\mathrm{Mp}_2(\bZ)$ denote the metaplectic double cover of $\mathrm{SL}_2(\bZ)$, which is equipped with a Weil representation $\rho_L : \mathrm{Mp}_2(\bZ) \to \mathrm{GL}(U)$. Let $\kappa = 1 + \frac{b}{2}$ and $M_{\kappa}$ be the space of vector-valued modular forms of $\mathrm{Mp}_2(\bZ)$ of weight $\kappa$ with respect to $\rho_L$. Let $\{ \mathfrak{e}_\mu : \mu \in L^\vee / L \}$ be the standard basis for $U$. Let $E_0(\tau)\in M_{\kappa}(\rho_L)$ be the vector valued Eisenstein series with constant term $\mathfrak{e}_0$ as in \cite[\S 2.1]{Bru17}.  Denote the Fourier expansion of $E_0(\tau)$ as
\begin{equation}
\label{eqn: Fourier expansion}
    E_0(\tau)= {\sum_{\mu \in L^\vee / L}} \sum_{m\geq 0,\mu\in \mathbb{Z}+Q(\mu)}q_L(m,\mu) q^m\mathfrak{e}_\mu
,\text{ where }q=e^{2\pi i\tau}.
\end{equation}
Let $q_L(m) = q_L(m, 0)$. The theorem \cite[Theorem 11]{BK01} provides an explicit expression of $q_L(m)$. Before stating it, we first recall the notion of local density, which will be used throughout the paper:
\begin{definition}
\label{def: local density}
For a quadratic $\mathbb{Z}$-lattice $(L,Q)$, a prime $l$ and an integer $m$, the local density of $L$ representing $m$ over $\mathbb{Z}_l$ is defined as $$\delta(l,L,m)= \lim_{a\rightarrow \infty} l^{a(1-\mathrm{rk}\,L)}\#\{v\in L/l^aL|Q(v)\equiv m\mod l^a\}.$$
\end{definition}
By \cite[Lem.~5]{BK01}, the limit stablizes when $a \ge 1 + 2 v_\ell(2m)$ (take $d_\mu = 1$ as $\mu = 0$). 

Given $0\neq D\in\mathbb{Z}$ such that $D\equiv 0,1\pmod{4}$, we use $\chi_D$ to denote the Dirichlet character $\chi_D(a)=\left(\frac{D}{a}\right)$, 
where $\left(\frac{\cdot}{\cdot}\right)$ is the Kronecker symbol. For a Dirichlet character $\chi$ we set $\sigma_s(m,\chi)=\sum_{d\mid m}\chi(d)d^{s}$. 
\begin{theorem}[{\cite[Thm.~11]{BK01}}] \label{thm: Eisenstein Coeffs}
In the above setting, we have: 
\begin{enumerate}[label=\upshape{(\alph*)}]
\item For $b$ even, the Fourier coefficient $q_L(m)$ is 
\[-\frac{2^{\kappa}\pi^{\kappa}m^{\frac{b}{2}}\sigma_{1-\kappa}(m,\chi_{(-1)^{\kappa}4\det L})}{\sqrt{|L^\vee/L|}\Gamma(\kappa)L(\kappa,\chi_{(-1)^{\kappa}4\det L})}\prod_{\ell \mid 2\det(L)}\delta(\ell, L, m).\]

\item For $b$ odd, write $m=m_0f^2$, where $\gcd(f,2\det L)=1$ and $v_\ell(m_0)\in \{0,1\}$ for all $\ell\nmid 2\det L$. Then the Fourier coefficient $q_L(m)$ is
\[-\frac{2^{\kappa}\pi^{\kappa}m^{\frac{b}{2}}L(\kappa+1/2,\chi_{\mathcal{D}})}{\Gamma(\kappa)\sqrt{|L^\vee/L|}\zeta(2\kappa)}\left(\sum_{d\mid f}\mu(d)\chi_{\mathcal{D}}(d)d^{-(\kappa+1/2)}\sigma_{1-2\kappa}(f/d)\right)\prod_{\ell\mid 2\det L}\Big( \frac{\delta(\ell,L,m)}{1-\ell^{-2\kappa}}\Big),\]
where $\mu$ is the Mobius function and $\mathcal{D}=(-1)^{\kappa-1/2}2m_0\det L$.
\end{enumerate}
In particular, $|q_L(m)|\asymp m^{\frac{b}{2}}$ for all $m$ representable by $(L,Q)$.
\end{theorem}

Here the last assertion is a direct consequence of the above explicit formulae and the fact $\delta(\ell,L,m)\asymp 1$ (see also \cite[\S 4.3.1]{MST}).
It plays an important role in arithmetic intersection theory.

\subsubsection{Arithmetic intersection theory}\label{subsub: arithmetic intersection theory} Let $k=\overline{\bF}_p$. In the following we review the arithmetic intersection theory on $\mathscr{S}^\tor_k$. Let the setup be as in \S\ref{subsub:gEs}. We 
will use the notation $\overline{\mathcal{Z}(\mu, m)}$ to denote the closure of the special divisor $\mathcal{Z}(\mu, m)$ in $\mathscr{S}^{\text{tor}}$. To every $\overline{\mathcal{Z}(\mu, m)}$, there is a carefully chosen associated divisor $$\mathcal{Z}^{\text{tor}}(\mu, m)= \overline{\mathcal{Z}(\mu, m)}+ \mathcal{B}\in \mathrm{Pic}(\mathscr{S}^\tor)_{\bR},$$
where $\mathcal{B}$ is a certain $\bR$-linear combinations of boundary divisors; see \cite[(3.1.1)]{Tay22}. Write $Z^{\tor}(\mu,m)$ for the mod $p$ reduction of $\mathcal{Z}^{\text{tor}}(\mu, m)$, and let $Z^\tor(m)$ be $Z^{\tor}(0,m)$. Let $\overline{Z(m)}$ be the Zariski closure of $Z(m)$ in $\mathscr{S}_k^\tor$. Consider the following generating series 
$$\Phi^{\text{tor}}_L:=\mathcal{L}^{-1}\mathfrak{e}_0+\sum_{\mu\in L^\vee/L}\sum_{m\in Q(\mu)+\mathbb{Z}}{Z}^{\text{tor}}(\mu, m) q^m\mathfrak{e}_\mu
,\text{ where }q=e^{2\pi i\tau}.$$
It is known that $\Phi^{\tor}_L$ is modular, i.e.,  $\Phi^{\tor}_L\in M_{\kappa}(\rho_L)\otimes \mathrm{Pic}(\mathscr{S}_k^\tor)_\mathbb{Q}$; cf.  \cite[Theorem 3.1]{Tay22}.  

Let $C$ be a proper smooth connected  curve with a non-constant map $C\rightarrow \mathscr{S}_{k}^\tor$ whose image lies generically in $\mathscr{S}_{k}$. The modularity of $\Phi^{\tor}_L$ enables us to relate the intersection number $(\overline{Z(m)}\cdot C)$ and the coefficients $q_L(m)$ of the Eisenstein series $E_0(\tau)$: 
\begin{proposition}[\cite{Tay22}, Proposition 4.10]\label{prop: global bound}
 We have $\overline{Z(m)}\cdot C =|q_L(m)|(C\cdot \omega) + o(m^{\frac{b}{2}}).$
\end{proposition}

\begin{definition}
\label{def: intersect at P}
We call a point $P\in C$ a \emph{boundary point} if it does not map to $\mathscr{S}_k$, and an \emph{interior point} otherwise. We call an interior point $P\in C$ a \emph{supersingular point} if it maps to the supersingular locus. We write $i_{P}(C \cdot \overline{Z(m)})$ for the intersection multiplicity of $C$ and $\overline{Z(m)}$ at $P$. 
\end{definition}

We note that $q_L(m) \asymp m^{\frac{b+2}{2}}$. \cite{MST} and \cite{Tay22} control the local intersection numbers at non-supersingular points. For $X$ some integer, let $S_X$ denote the set of positive integers up to $X$ which are relatively prime to $p$. 
\begin{theorem}[\cite{MST},\cite{Tay22}]\label{thm: bounding nonss contribution}
~\begin{enumerate}[label=\upshape{(\alph*)}] \item ({\cite[Proposition 7.11]{MST}}) Let $P\in C$ be an interior point which doesn't map to the superinsginular locus. Then, we have $$\sum_{m\in S_X} i_P(C\cdot \overline{Z(m)}) = O(X^{\frac{b}{2}} \log X).$$
        \item (\cite[Proposition 4.12]{Tay22}) Let $P\in C$ be a boundary point. Then we have $$\sum_{m\in S_X} i_P(C\cdot \overline{Z(m)}) = O(X^{\frac{b}{2}} \log X).$$
    \end{enumerate}
\end{theorem}

In order to prove the main theorem, we will need the following result that treats supersingular points. 

\begin{theorem}\label{thm: main version 2}
    Let the setup be as above. Suppose $C$ is as in Theorem \ref{mainShimura}. There exists a constant $\alpha < 1$ independent of $X$ such that we have $$\sum_{m\in S_X} \sum_{P\in C_{\mathrm{ss}}} i_P(C\cdot Z(m)) < \alpha (C\cdot\omega) \sum_{m\in S_X} |q_L(m)|+O(X^{\frac{b+1}{2}}).$$
\end{theorem}

We now prove Theorem \ref{mainShimura} assuming Theorem \ref{thm: main version 2}. As the deduction of Theorem \ref{mainShimura} from the above stated local and global results follows the same contours as the proofs in \cite{COhypersymmetric}, \cite{Ch18},\cite{ST}, \cite{MST1}, \cite{SSTT}, and \cite{MST}, we will be content with a brief account.
\begin{proof}[Proof of Theorem \ref{mainShimura} assuming Theorem \ref{thm: main version 2}]

We will obtain a contradiction by assuming that there are only finitely many points $P_1\hdots P_n$ of $C$ contained in $\bigcup_{m \in \bZ_{\geq 1}}\overline{Z(m)}$. We may assume that the first $r$ of these points are supersingular. Then, we have that $C\cdot \overline{Z(m)} = \sum_{i=1}^ni_{P_i}(C\cdot \overline{Z(m)})$ for every positive $m$. For $X$ large enough, summing both sides of the equality for $m\in S_X$ and applying Proposition \ref{prop: global bound}, we obtain $$\sum_{m\in S_X} \left[|q_L(m)|(C\cdot \omega) + o(m^{\frac{b}{2}})\right] = \sum_{m\in S_X} \sum_{i=1}^ni_{P_i}(C\cdot \overline{Z(m)}).$$ By Theorem \ref{thm: main version 2}, we have $$(1-\alpha)(C\cdot \omega)\sum_{m\in S_X} |q_L(m)| + O(X^{\frac{b+1}{2}}) < \sum_{m\in S_X} \sum_{i=r+1}^n i_{P_i}(C\cdot \overline{Z(m)}).$$ As $P_i$ is non-supersingular for $i >r$, we apply Theorem \ref{thm: bounding nonss contribution} to see that the RHS has order of magnitude $O(X^{\frac{b}{2}}\log X)$. The LHS has order of magnitude ${X^{\frac{b+2}{2}}}$, which yields the required contradiction.  
\end{proof}

In the next subsection, we will reduce Theorem \ref{thm: main version 2} (using \cite{MST}) to purely local statement for formal curves passing through supersingular points.


\subsection{The local and global intersection numbers at a supersingular point}\label{sub: intersection numbers}
Recall that $a$ is the largest integer such that $C$ maps to the locus $\cF_{a+1} \subseteq \shS$. Set $h_P = i_P(C\cdot\cF_{a+2})$, the local intersection multiplicity of $C$ and $\cF_{a + 2}$ at $P$. 

\begin{definition}
\label{def: global int num}
    For an integer $m>0$,  define \textbf{global intersection number of $(C\cdot Z(m))$ at $P$} to be
$$g_P(m)=  \frac{h_P}{p^{a+1}-1}|q_L(m)|.$$
\end{definition} 

Since $\cF_{a+2}$ is the vanishing locus of a section of $\omega^{p^{a+1}-1}|_{\cF_{a+1}}$; cf. Lemma~\ref{lem: Frob and E}, we immediately get
\begin{equation}\label{eq:g_Pandglobal}
(C\cdot\omega)|q_L(m)| = \sum_{P\in C\cap \cF_{a+2}}g_P(m).
\end{equation}

 Now, let $P\in C_{\mathrm{ss}}$ be a supersingular point, and let $t$ be a local parameter at $P$. Let $\cA_r := \cA/\Spf k[\![t]\!]/(t^r)$ denote the pullback of $\cA$ to the $r$th infinitesimal neighbourhood of $C$ at $P$. 
\begin{definition}
\label{def: L_P, r}
We let $L_{P,r} = L(\cA_r)$. Let $L_{P,1} \subset L'_{P,1} \subset L_{P,1}\otimes \bZ_{(p)}$ be a lattice that is maximal at all primes $\ell \neq p$, and let $L'_{P,r} = L_{P,r} \cap L'_{P,1}$. 
\end{definition}

By construction, we have that $L_{P,r} \subset L_{P,r-1}$. The equation below follows directly from the moduli interpretation of $Z(m)$. 
\begin{equation}\label{eq:local intersection defn}
    i_P(C\cdot Z(m)) = \sum_{r=1}^{\infty} \#\{v\in L_{P,r}: Q(v) = m \}. 
\end{equation}

For the rest of this section, we will reduce Theorem \ref{thm: main version 2} to the purely local result stated just below: 
\begin{theorem}\label{thm: main local version} Notation as above. Let $q_{L'_{P,r}}(m)$ denote the $m$th Fourier  coefficient of the Eisenstein series associated to the lattice $L'_{P,r}$, and define the series $$S'(m):=\sum_{r=1}^{\infty} \frac{q_{L'_{P,r}}(m)}{g_P(m)}.$$
   There is a constant $\alpha <1$ such that for any $m$ not a multiple of $p$, we have $S'(m) < \alpha$. 
\end{theorem}
\begin{remark}\label{rmk:explicit q_L'}
    By the first displayed equation in the proof of Lemma 7.16 of \cite{MST}, we have that $$q_{L'_{P,r}}(m) = \frac{q_L(m)\delta(p,L'_{P,r},m)}{\sqrt{[L^{\prime,\vee}_{P,r}\otimes \bZ_p\ :\ L^{\prime}_{P,r}\otimes \bZ_p]} \left(1-\chi_{(-1)^{1+b/2}4\det L}(p)p^{-1-b/2}\right)}. $$
    So each term of the sum $S'(m)$ can be explicitly written as 
    $$ \frac{q_{L'_{P,r}}(m)}{g_P(m)}=\frac{p^{a+1}-1}{h_P}\left(1-\chi_{(-1)^{1+b/2}4\det L}(p)p^{-1-b/2}\right)^{-1} \frac{\delta(p,L'_{P,r},m)}{\sqrt{[L^{\prime,\vee}_{P,r}\otimes \bZ_p\ :\ L^{\prime}_{P,r}\otimes \bZ_p]} }.$$
\end{remark}

\begin{proof}[Proof of Theorem \ref{thm: main version 2} assuming Theorem \ref{thm: main local version}]
    The implication follows directly from \cite[Section 7]{MST} so we will be brief. Clearly, we have that $i_P(C\cdot Z(m)) \leq \sum_{r=1}^{\infty} \#\{v\in L'_{P,r}: q(v) = m\}$. An analysis identical to the one in \cite[\S 7.12, Proposition 7.13, \S 7.14, Theorem 7.18]{MST} allows us to replace $i_P(C\cdot Z(m))$ in Theorem \ref{thm: main version 2} with $ \sum_{r=1}^{\infty} q_{L'_{P,r}}(m)$ up to an error term.
    
     Assuming  Theorem~\ref{thm: main local version}, we have $ \sum_{r=1}^{\infty} q_{L'_{P,r}}(m)<\alpha g_P(m)$ for some $\alpha<1$. Using (\ref{eq:g_Pandglobal}), we have $$\sum_{P\in C_{\mathrm{ss}}} g_{P}(m) \leq (C\cdot \omega)|q_L(m)|. $$
    Summing up over $P\in C_{\mathrm{ss}}$ and $m\in S_X$, and possibly enlarging $\alpha$, we immediately find that $$\sum_{m\in S_X} \sum_{r=1}^{\infty} q_{L'_{P,r}}(m)<\alpha \sum_{m\in S_X} \sum_{P\in C_{\mathrm{ss}}} g_P(m)\leq \alpha (C\cdot \omega)\sum_{m\in S_X}|q_L(m)|.$$
This implies Theorem \ref{thm: main version 2} as required. 
\end{proof}

\section{Local structure of Shimura varieties at a superspecial point}\label{sec: local structure at superspecial point}
\subsection{Local structure in characteristic $p$}\label{subsec: loc supersingular}

Let $k$ be a perfect field of characterisic $p$. We recall the group-theoretic description of \cite{Kis10} of the formal neighborhood of a closed point $P$ in $\shS$, which we shall denote by $\widehat{\shS}_P$. 

Let us write the K3 crystal $\bL_{\cris, P}$ simply as $\bL$. Then recall that $\bL \tensor_W k = \bL_{\dR, P}$ is equipped with a Hodge filtration. Let $\mu : \mathbb{G}_{m,W} \to \text{SO}(\bL)$ be a co-character (referred to as the ``Hodge co-character") whose reduction modulo $p$ splits the Hodge filtration. Define $U$ to be the opposite unipotent subgroup in $\mathrm{SO}(\bL)$ relative to $\mu$, and let $\Spf(R)$ denote the completion of $U$ at the identity section. Let $\sigma$ be a lift of the Frobenius endomorphism on $R \tensor_W k$, and let $u \in U(R)$ be the tautological point. Then, according to \cite[\S\S 1.4, 1.5]{Kis10} (see also \cite[Proposition 4.7]{KLSS}), there exists an isomorphism $\shS^{/P}_W \simeq \Spf(R)$ through which $\bL_\cris(R)$ can be identified with $\bL \tensor_W R$ such that the Frobenius action is given by $\Frob = u \circ (\varphi \tensor \sigma)$. Let us fix these identifications.

Using the above description, together with Ogus' results on supersingular K3 crystals (\cite{Ogu79}), we can express F-crystal $\bL_\cris$ over $\shS^{/P}_W$ in terms of local coordinates. This has been done in \cite{MST} so we briefly recap the main points for readers' convenience and set up some notations. First, the Tate module $\cL := \bL^{\varphi = 1} = $ orthogonally splits into two parts $\cL_0 \oplus \cL_1$ such that , if we denote by $\<-, -\>_i$ the symmetric pairing on $\cL_i$, then $\<-, -\>_1$ is self-dual over $\bZ_p$ whereas $\<-, -\>_0$ is divisible by $p$ and $p^{-1} \<-, -\>_0$ is self-dual. Define $\bL_1 = \cL_1 \tensor W$. Then $\bL_1$ is a direct summand of $\bL$, which splits into $\bL_0 \oplus \bL_1$. The $W$-lattice $\bL_0$ is self-dual such that $\bL_0^{\varphi = 1} = \cL_0$ and $\cL_0 \tensor W \subseteq \bL_0 \subseteq p^{-1}(\cL_0 \tensor W)$. Set $t_P := \mathrm{rk}_{\bZ_p} \cL_0$. Then $2 \mid t_P$ and $a_P := t_P / 2$ is called the \textit{Artin-invariant} of $\bL$. The point $P$ is defined to be \emph{superspecial} if $a_P = 1$. In this section, we address the case when $P$ is a superspecial point, and will deal with more general supersingular points in a later section. 

As in \cite[Section 4.7]{MST}, we have the following trichotomy. 
\begin{enumerate}\label{trichotomy}
    \item $\textrm{rank } \cL_1$ is odd. We call this the \emph{odd case}, and it corresponds to $b$ odd.
    
    \item $\textrm{rank } \cL_1 = 2m$ and $\cL_1$ is a split quadratic space. We call this the \emph{split case}, and it corresponds to $b$ even and $L\otimes \bZ_p$ \emph{not} being split. 

    \item $\textrm{rank } \cL_1 = 2m$ and $\cL_1$ is \emph{not} a split quadratic space. We call this the \emph{inert case}, and it corresponds to $b$ even and $L\otimes \bZ_p$ \emph{being split}.
\end{enumerate}

\begin{remark}
    The cases (1), (2), (3) above match up with the cases (1), (2), (3) in \cref{prop: Shen-Zhang}. This is essentially a consequence of the Hasse–Minkowski theorem, and the reader may refer to the proof of \cite[Thm.~7.2.4]{HP17} for this argument. 
\end{remark}

\begin{remark}
   We note that we are overloading the symbol $m$.
 Earlier, $m$ indexed special divisors, whereas we are now using $m$ to denote half the rank of $\cL_1$ . We will explicitly alert the reader before reverting to the former usage.

\end{remark}

\begin{remark}\label{rmk:onlyinert}
    As in \cite{MST}, the main theorem in the setting of the inert case follows from the main theorem in the setting of the split case by embedding inside a larger dimensional Shimura variety. Indeed, in the inert setting (under the assumptions of Theorem \ref{mainShimura}), the local calculations needed to establish the inequality in Theorem \ref{thm: main local version} follow by embedding the global lattice $L$ inside a higher-rank maximal lattice $L'$ which is \emph{self-dual but not} split at $p$ (which therefore corresponds to the split case), calculating the quantities $q_{L'}(m)$, and using this to calculate the quantities $q_L(m)$ to obtain the main theorem. Henceforth, we will constrain ourselves to the split case. In this case, the term with the character $\chi$ equals $\frac{1}{(1+p^{-(m+1)})}$.
\end{remark}

We will briefly need the odd case (despite the statement of Theorem \ref{mainShimura} being trivial) to analyze singularities of Newton strata. \textit{For the rest of this section, and in Sections \ref{sec: decay superspecial}, \ref{sec: first step superspecial}, and \ref{sec: supersingular}, we will work in the split case unless explicitly stated otherwise.}

\subsection{Equations for Newton strata}
\subsubsection{}
\label{subsub:explicitcoordinatheonL}
In order to compute with coordinates, we need to introduce bases to $\bL$. For $\bL_1$, we just take any basis $\{ e_i, f_i \}$ for $\cL_1$ such that the intersection matrix is given by $\left[
\begin{array}{c|c}
0 & I \\
\hline
I & 0
\end{array}
\right]$. Using work of Ogus as in \cite{MST1} and \cite{MST}, $\cL_0$ has a basis $e',f'$ such that $\< e', e' \>_0 = 2p, \< f', f' \>_0 = -2 \lambda^2 p$ and $\< e', f' \>_0 = 0$, where $\lambda \in W(\bF_{p^2})^{\times}$ satisfies $\sigma(\lambda) = -\lambda$ (cf. \cite[Lem.~4.3]{MST}). Then, a basis of $\bL_0$ is given by $w' = \frac{1}{2\lambda p}(\lambda e' + f'), v' = \frac{1}{2\lambda}(\lambda e' - f')$. 
From \cite[\S4.8]{MST}, we can extract that for some choice of coordinates 
\begin{equation}
    \label{eqn: coords of R} 
    \shS^{/P}_W = R \simeq W[\![x_1, \cdots, x_m, y_1, \cdots, y_m]\!]
\end{equation}
the operator $\Frob$ on $\bL \tensor R$ (identified with $\bL_\cris(R)$) is given by the matrix $(u B) \circ \sigma$
where 
\begin{equation}
\label{eqn: u and B}
    u = \left[
\begin{array}{cc|cc}
1 & Q & \mathbf{x} & \mathbf{y} \\
 & 1 & & \\
\hline
& - \mathbf{y}^t & I_m &  \\
& - \mathbf{x}^t & & I_m
\end{array}
\right] \text{ and } 
B = \left[
\begin{array}{cc|cc}
 & p &  &  \\
 p^{-1} &  & & \\
\hline
&  & I_m &  \\
&  & & I_m
\end{array}
\right] \Rightarrow 
uB = \left[
\begin{array}{cc|cc}
p^{-1} Q & p & \mathbf{x} & \mathbf{y}  \\
p^{-1} &  & & \\
\hline
- p^{-1} \mathbf{y}^t & & I_m &  \\
- p^{-1} \mathbf{x}^t &  & & I_m
\end{array}
\right]
\end{equation}
with respect to the basis $\{ v', w', e_1, \cdots, e_m, f_1, \cdots, f_m \}$. 
Note that compared to \textit{loc. cit.}, we switched the role played by the letters $e_i, f_i, x_i, y_i$'s with $e'_i, f'_i, x'_i, y'_i$, so that now basis vectors or coordinates related to $\bL_0$ are labelled with a ($'$). In the above, $\mathbf{x}$ and $\mathbf{y}$ are the row vectors $[x_1, \cdots, x_m]$ and $[y_1, \cdots, y_m]$ respectively. Here, $Q=-\sum_{k=1}^m x_ky_k = - \bfx \bfy^t$. 

For a matrix $M$ we denote $M^{(i)}$ the matrix whose entries are twisted by $\sigma^{i}$, set
\begin{equation}
\label{eqn: define F infty and Q}
    Q_{j}=-\sum_{k=1}^m\left( x_ky^{(j)}_k+x^{(j)}_ky_k \right) = - \left( \bfx^{(j)} \bfy^t + \bfx (\bfy^{(j)})^t \right) \in W[\![x_i,y_i]\!].
\end{equation}

In particular, we have that $Q =\frac{1}{2} Q_0$.



\begin{remark}\label{rmk:abuseofnotation_Qi}By abuse of notation, we will also use $Q_i$ to denote the image of $Q_i$ in the ring $\mathbb{F}[\![x_i,y_i]\!]$, just as what we do in the next proposition. 
\end{remark}

\begin{proposition}\label{NPstrata} In case (2) of \cref{prop: Shen-Zhang}, we have the following. 
\begin{enumerate}[label=\upshape{(\alph*)}]
    \item for $0\leq h< m-1$, the NP strata $\sF_{h + 2}$ of height $(h+2)$ is cut out by $Q_0, Q_{1},...,Q_{h}$;
    \item the supersingular strata $\shS_\infty = \sF_{m + 1}$ is cut out by $Q_{0}, Q_{1},...,Q_{m-1}$
    \item  for $h \geq m$, $Q_{h}$ is contained in the ideal $(Q_{0}, Q_{1},...,Q_{m-1})$
\end{enumerate}
\end{proposition}

\begin{proof}
Let us first show that $\sF_{h + 1}$ is cut out from $\sF_h$ by $Q_{h - 1} = 0$. By \cref{lem: Frob and E}, $\pi (p \Frob^h)$ defines $\sigma$-semilinear endormorphism on $\bL_\dR |_{\sF_h}$. Then $E_h = \pi(p \Frob^h)(\bL_\dR) + E_{h - 1}$. By definition, $\sF_{h + 1}$ is cut out from $\sF_h$ by the condition that $E_h \subseteq \Fil^0 \bL_\dR$, which is equivalent to 
\begin{equation}
\label{eqn: cut h + 1 from h}
     \pi(p \Frob^h) = 0 \textrm{ on } \gr^{-1} \bL_\dR,
\end{equation}
since $E_{h - 1}$ is already in $\Fil^0 \bL_\dR$ over $\sF_h$. 

Let $A_i$ be the $R[1/p]$-matrix such that $p \Frob^i = A_i \circ \sigma^i$. Via the base change $R \to \sO_{\wt{\sF}_i}$, $A_i$ defines a matrix with $\sO_{\wt{\sF}_i}[1/p]$-coefficients, which by the integrality condition in fact has coefficients in $\sO_{\sF_i}$. Reducing mod $p$, we obtain a matrix $\overline{A}_i$ such that $\pi(p \Frob^i) = \overline{A}_i \circ \sigma^i$ on $\bL_\dR|_{\sF_i}$. By (\ref{eqn: cut h + 1 from h}), we are only concerned with the first column of $\overline{A}_i$, i.e., $\pi(p \Frob^i)([1, 0, \cdots, 0]^t)$.

We make the following definitions for all $i, n \in \mathbb{N}$.  
\begin{itemize}
    \item Let $W_i(n)$ denote the length $n$ words in $\{ Q_j^{(\beta)} | 0 \le j \le i, \beta \in \mathbb{N} \}$. 
    \item Let $\Delta_i(d)$ denote the set of sums of the elements of the form $p^{-r} w$ for $w \in W_i(n)$ and $n - r = d$. 
\end{itemize}
For book-keeping set them to be empty when $i = -1$. One can easily prove by induction that for every $i \in \mathbb{N}$, $\Frob^i(e_1) \in R[1/p]^{2 + 2m}$ is of the following form: 
\begin{equation}
    \begin{bmatrix} a + p^{-1} Q_{i - 1} \\ b \\ - \left(\sum_{\ell = 0}^{i - 2} c_\ell (\bfy^t)^{(\ell)} \right) - p^{-1} (\bfy^t)^{(i - 1)} \\ - \left(\sum_{\ell = 0}^{i - 2} c_\ell (\bfx^t)^{(\ell)} \right) - p^{-1} (\bfx^t)^{(i - 1)} \end{bmatrix} \text{ such that } a \in \Delta_{i - 2}(0) \text{ and } b, c_0, \cdots, c_{i-2} \in \Delta_{i-2}(-1). 
\end{equation}
For example, the first entry of $\Frob^{i + 1}(e_1) = uB \cdot \Frob^i(e_1)^{(1)}$ is given by 
\begin{equation}
    [p^{-1} Q a^{(1)} + p^{-2} Q Q_{i - 1}^{(1)} + p b^{(1)} - \sum_{\ell = 0}^{i - 2} c_\ell^{(1)} Q_{\ell + 1}] + p^{-1} Q_i, 
\end{equation}
and the term in the bracket is the new $a$ for $\Frob^{i + 1}(e_1)$, which indeed lies in $\Delta_{i - 1}(0)$. We leave the checking of the other entries to the reader. 

Note that for every $i \le h$, since $\sF_i = R_0 / (Q_0, \cdots, Q_{i - 2})$, we have $p \mid Q_j$ inside $\wt{\sF}_i$ for every $j \le i - 2$, and hence $p^{-d} \Delta_i (d)$ defines an element in $\sO_{\wt{\sF}_i}$ for all $d \in \mathbb{Z}$. Using this observation, one checks that over $\wt{\sF}_h$, $p \Frob^h(e_1)$ indeed has coefficients in $\sO_{\wt{\sF}_h}$ and modulo $p$, its first entry is just $Q_{h - 1}$. Therefore, the condition (\ref{eqn: cut h + 1 from h}) amounts to $Q_{h -1} = 0$, or equivalently, $\sF_{h + 1}$ is cut out from $\sF_h$ by $Q_{h - 1}$, as desired. 

Note that the above argument has no restriction on $h$, including those $h \ge m + 1$. Parts (2) follows from the above when $h = m + 1$, and part (3) follow from the fact that $(\sF_h, E_h) = (\sF_{m + 1}, E_{m + 1})$ for all $h \ge m + 1$. 
\end{proof}

\begin{remark}
\label{rmk: summarize loc eqn}
We now summarize the procedure above, since we need to use it again but do not wish to repeat the details. First, we select a basis for \(\mathbf{L}_{\mathrm{cris}, P}(W)\) such that, upon reduction modulo \(p\), all but the first basis vector lie in \(\mathrm{Fil}^0 \mathbf{L}_{\mathrm{dR}}\). We then express the Frobenius operator as \(b\sigma\) for an explicit matrix \(b\). Next, we write the tautological point \(u\) in the local coordinates of the opposite unipotent subgroup for a chosen cocharacter $\mu$ that splits the Hodge filtration mod $p$. The remainder of the argument is essentially formalizing the following: the condition that the \((1,1)\)-entries of the matrices \(p(ub), \dots, p(ub)(ub)^{(h)}\) vanish modulo \(p\) is equivalent to the vanishing of the polynomials \(Q_0, \dots, Q_h\).
\end{remark}

\section{Decay of special endomorphisms: superspecial points.}
\label{sec: decay superspecial}

Let us keep using the notations in the preceeding section. In particular, $P$ is a superspecial point on $\shS_\bF$ and assume that we are in the split case. Suppose that $v \in \cL = \cL(\cA_P)$ is a formal special endomorphism and $C$ is a formal curve at $P$ on $\shS_\bF$ which is not generically ordinary. In this section, we study the decay rate of $v$, which measures how far $v$ extends along $C$.  

To formulate this notion, let us make a preliminary definition. 
\begin{definition}\label{eq:v_defspace}
    Given a formal special endomorphism $v \in \cL$, we let $x_v\in \mathbb{F}[\![x_i,y_i]\!]$ denote the equation whose vanishing is the locus to which $v$ extends. 
\end{definition}
For the fact that the locus to which $v$ extends is indeed cut out by a single equation, the reader may refer to \cite[Cor.~5.17]{MP16}. Technically, $x_v$ is only defined up to $\bF^\times$. For convenience we fix one by asking its leading term to be $1$ with respect to the lexicographic order of $x_1, \cdots, x_m, y_1, \cdots, y_m$. 
Note that then the coordinates we have chosen for $\shS^{/P}_W$ can be re-interpreted by $x_i = x_{e_i}, y_i = x_{f_i}$.

For the purpose of this section, we shall express the Frobenius operator $\Frob$ on $\bL[1/p] \tensor R$ with respect to the Frobenius-invariant basis $\{e', f', e_i, f_i \}_{i = 1, \cdots, m}$ of $\bL[1/p]$, as opposed to $\{ v', w', e_i, f_i \}_{i = 1, \cdots, m}$ used in the preceeding section.\footnote{We again remark that the letter $m$ is used in relation to the rank of $\bL$, unlike in Section \ref{sec:arithmetic}. We will alert the reader if we resume using $m$ in the context of special divisors.} Now, with respect to $\{e', f', e_i, f_i \}_{i = 1, \cdots, m}$, $\Frob=(1+F)\comp \sigma$, where 
\begin{equation}
\label{eq:Fmatrix}
\renewcommand{\arraystretch}{1.5}
    F = \left[
    \begin{array}{cc|cc}
         \frac{Q}{2p} & \frac{-\lambda Q}{2p} & \frac{\bfx}{2p} & \frac{\bfy}{2p} \\
         \frac{Q}{2p\lambda} & \frac{-Q}{2p} & \frac{\bfx}{2p \lambda} & \frac{\bfy}{2p \lambda} \\ \hline
         -\bfy^t & \lambda \bfy^t & & \\
         -\bfx^t & \lambda \bfx^t & & 
    \end{array}\right] 
\end{equation}
Set $F_\infty=\prod_{i=0}^\infty (I+F^{(i)})$. Recall that $\mathbf{x}=[x_1,...,x_m]$, $\mathbf{y}=[y_1,...,y_m]$ and we defined $Q_i$'s in (\ref{eqn: define F infty and Q}). 

\subsection{Explicit computation of $F_\infty$}\label{subsec: F infty} In this subsection, we explicitly compute $F_\infty$ over the non-ordinary Newton strata. The assumption implies that $Q=0$. So the upper left block of $F$ in (\ref{eq:Fmatrix}) is 0. Then the upper right, and lower left blocks of $F$ can be written as \begin{equation}\label{eq:blocks}
    \frac{1}{2p}\begin{bmatrix}
    1\\
   \lambda^{-1}
\end{bmatrix}\begin{bmatrix}
    \mathbf{x} & \mathbf{y}
\end{bmatrix},\text{ and   }\begin{bmatrix}
    -\mathbf{y}^t\\
    -\mathbf{x}^t
\end{bmatrix}\begin{bmatrix}
    1 & -\lambda
\end{bmatrix},
\end{equation} respectively. The following are some basic observations: 
\begin{lemma}\label{lm:easythinglemma}Let $0\leq i<j$, then \begin{enumerate}[label=\upshape{(\alph*)}]
    \item $\begin{bmatrix}
    1 & -\lambda
\end{bmatrix}^{(i)}\begin{bmatrix}
    1\\
   \lambda^{-1}
\end{bmatrix}^{(j)}=1-(-1)^{j-i}{= 2 \delta_{i < j}}$.
\item $\begin{bmatrix}
    \mathbf{x} & \mathbf{y}
\end{bmatrix}^{(i)}\begin{bmatrix}
    -\mathbf{y}^t\\
    -\mathbf{x}^t
\end{bmatrix}^{(j)} = Q_{j-i}^{(i)}$. 
\end{enumerate}
\end{lemma}
\begin{proof}
    Follows by a direct computation. 
\end{proof}
\begin{notation}
We make several definitions, whose importance will be clear in a short while.
For $n\geq 0$, let $l_1<l_2<...<l_{2n}$ be an increasing chain of $2n$ nonnegative integers. For $n>0$, define \begin{align}
    &Q_{l_1<l_2<...<l_{2n}}=\prod_{k=1}^n Q_{l_{2k}-l_{2k-1}}^{(l_{2k-1})}. \label{eqn: define Q <} \\
&\delta_{l_1<l_2<...<l_{2n}}= \prod_{k=1}^n\frac{1}{2}[{1-(-1)^{l_{2k}-l_{2k-1}}}] \label{eqn: define delta <}.
\end{align}
When $n=0$, we also define by default $$Q_{\varnothing}=1, \delta_{\varnothing}=1.$$

It is immediate from the definition that $\delta_{l_1<l_2<...<l_{2n}}=\prod_{k=1}^n \delta_{l_{2k-1}<l_{2k}}$ can only take value in $\{0,1\}$. It is 1 only when $l_{2k}+l_{2k-1}\equiv 1(\mod 2)$ for all $k=1,2,...,n$. Otherwise it is 0.  
\end{notation}

We now give a thorough description of $F_\infty$:\begin{lemma}\label{lm:Finftyexplicit} Notation as above. We have
 \begin{align*}
    F_\infty= I+\sum_{n\geq 1}p^{-n}F_n,
\end{align*}
where $F_n=\begin{bmatrix}
X_n & Y_n\\
Z_n & W_n
\end{bmatrix}$, and $X_n$ is the upper left $2\times 2$ block, $Y_n$ is the upper right $2\times 2m$ block, $Z_n$ is the lower left $2m\times 2$ block, and $W_n$ is the lower left $2m\times 2m$ block, with explicit description as follows:
\begin{align*}
    X_n&=\sum_{0\leq i_1<i_2<...<i_{2n}} \frac{1}{2} \delta_{i_2<i_3<...<i_{2n-2}<i_{2n-1}}Q_{i_1<i_2<...<i_{2n}}\begin{bmatrix}
    1\\
   (-1)^{i_1}\lambda^{-1}
\end{bmatrix}\begin{bmatrix}
    1 & (-1)^{1+i_{2n}}\lambda
\end{bmatrix},\\
   Y_n&= \sum_{0\leq  i_1<i_2<...<i_{2n-1}} \frac{1}{2} \delta_{i_2<i_3<...<i_{2n-2}<i_{2n-1}}Q_{i_1<i_2<...<i_{2n-2}}\begin{bmatrix}
    1\\
   (-1)^{i_1}\lambda^{-1}
\end{bmatrix}\begin{bmatrix}
   \mathbf{x} & \mathbf{y}
\end{bmatrix}^{(i_{2n-1})},\\
  Z_n&= \sum_{0\leq  i_0<i_1<i_2<...<i_{2n}} \delta_{i_0<i_1<...<i_{2n-2}<i_{2n-1}}Q_{i_1<i_2<...<i_{2n}}\begin{bmatrix}
    -\mathbf{y}^t\\
    -\mathbf{x}^t
\end{bmatrix}^{(i_0)}\begin{bmatrix}
   1 & (-1)^{1+i_{2n}}\lambda
\end{bmatrix},\\
  W_n&= \sum_{0\leq  i_0<i_1<i_2<...<i_{2n-1}} \delta_{i_0<i_1<...<i_{2n-2}<i_{2n-1}}Q_{i_1<i_2<...<i_{2n-2}}\begin{bmatrix}
    -\mathbf{y}^t\\
    -\mathbf{x}^t
\end{bmatrix}^{(i_0)}\begin{bmatrix}
   \mathbf{x} & \mathbf{y}
\end{bmatrix}^{(i_{2n-1})}.
\end{align*}
\end{lemma}


\begin{proof}
This follows easily from a direct computation.    
\end{proof}
\begin{lemma}[Recursion formula]\label{lm:Finftyrecursive}
     Notation as above. Let $X_n^+$ (resp. $Y_n^+$) be the sub-sum of $X_n$ (resp. $Y_n$) with $i_1$ even. For $n\geq 1$, define $\mathbf{U}_n$ to be the first row of $\begin{bmatrix}
     X_n^+& Y_n^+
     \end{bmatrix}$. 
 Then for $r>0$, we have
\begin{equation}
\label{eq: split U}
    \mathbf{U}_{n+r}= \sum_{\substack{0\leq i_1<i_2<...<i_{2r}\\i_1 \text{ even}}}\delta_{i_2<i_3<...<i_{2r-1}}Q_{i_1<i_2<...<i_{2r}}\mathbf{U}^{(1+i_{2r})}_n.
\end{equation}
Moreover, the matrix $F_{n}$ can be recovered from $\mathbf{U}_n$ as follows: 
\begin{align}
    \begin{bmatrix}
    X_{n}&Y_{n}
\end{bmatrix}&= \begin{bmatrix}
    1\\
    \lambda^{-1}
\end{bmatrix}\mathbf{U}_n+ \begin{bmatrix}
    1\\
    -\lambda^{-1}
\end{bmatrix}\mathbf{U}_n^{(1)}, \label{eq: X+ Y+} \\
\begin{bmatrix}
    Z_{n}&W_{n}
\end{bmatrix}&=\sum_{0\leq i_0}2 
\begin{bmatrix}
    -\mathbf{y}^t\\
    -\mathbf{x}^t
\end{bmatrix}^{(i_0)}
    \mathbf{U}^{(1+i_{0})}_n. \label{eq: Z+ W+}
\end{align}

\end{lemma}

\begin{proof}
This follows from Lemma~\ref{lm:Finftyexplicit} and a direct computation. 
\end{proof}

We shall now use $F_\infty$ to study the deformation theory of a special endomorphism. 
Recall that $\{e', f', e_i, f_i\}$ is a $\bZ_p$-basis for $\cL$. 
\begin{notation} \label{not: v tilde +}
    For an element $v\in \cL$, we define 
\begin{equation}
    \label{eqn: def tilde v}
    \tilde{v}:=F_\infty v\in \mathbb{L}(W)\otimes K[\![x_i,y_i]\!].
\end{equation}
There exists a rank $2m+2$ column vector $\mathbf{c}$ (resp. $\tilde{\mathbf{c}}$) valued in $\bZ_p$ (resp. $K[\![x_i,y_i]\!]$) such that
\begin{equation}
\label{eqn: define c and c tilde}
\begin{aligned}
     v&=(e',f';e_i,f_i)\mathbf{c},\\\tilde{v}&=(e',f';e_i,f_i)\tilde{\mathbf{c}}.
\end{aligned}  
\end{equation} 
Express $\mathbf{c}$ as $\sum_{k\geq 0} \mathbf{c}_kp^k$, where each entry of $\mathbf{c}_k$ is a Teichmüller lift of some element in $\mathbb{F}_p$. Let $\mathbf{U}_n$ be as in Lemma~\ref{lm:Finftyrecursive}. Define
\begin{equation}
\label{eqn: defined D_n's}
\begin{aligned}
     D_n(v)&:= 2\sum_{k\geq 0}\mathbf{U}_{n+k}\mathbf{c}_k \in W[\![x_i,y_i]\!].\\
    \overline{D}_n(v)&:={D}_n(v)\bmod p\in \mathbb{F}[\![x_i,y_i]\!].
\end{aligned}  
\end{equation} 
\end{notation}
\begin{lemma}\label{lm;recursive2}
    The following are true: 
    \begin{enumerate}[label=\upshape{(\alph*)}]
    \item For $l\geq 0$, $D_n(v)^{(l)}-D_n(v)^{p^l}\in p^lW[\![x_i,y_i]\!]$. 
        \item For $r>0$, we have: $$D_{n+r}(v)=\sum_{\substack{0\leq i_1<i_2<...<i_{2r}\\i_1 \text{ even}}}\delta_{i_2<i_3<...<i_{2r-1}}Q_{i_1<i_2<...<i_{2r}}D_n(v)^{(1+i_{2r})}.$$ 
\item         The first two entries of $\tilde{\mathbf{c}}$ are 
$$\frac{1}{2}\sum_{n\geq 1}p^{-n}\left(\begin{bmatrix}
    1\\
    \lambda^{-1}
\end{bmatrix}D_{n}(v)+ \begin{bmatrix}
    1\\
    -\lambda^{-1}
\end{bmatrix}D_n(v)^{(1)}\right)+ O(1).$$
Here by $O(1)$ we mean that the entries have coefficients in $W$. The last $2m$ entries of $\tilde{\mathbf{c}}$ are 
$$\sum_{n\geq 1}p^{-n}\left(\sum_{0\leq i_0} 
\begin{bmatrix}
    -\mathbf{y}^t\\
    -\mathbf{x}^t
\end{bmatrix}^{(i_0)}
   D_n(v)^{(1+i_{0})}\right)+O(1).$$
   \item If we write $\tilde{v}=(v',w'; e_i,f_i)\tilde{\mathbf{c}}'$, then the first two entries of $\tilde{\mathbf{c}}'$ are $ \sum_{n\geq 1}p^{-n} D_n(v)^{(1)}+O(1)$ and $ \sum_{n\geq 1}p^{-n} D_{n+1}(v)+O(1)$, respectively. The last $2m$ entries of $\tilde{\mathbf{c}}'$ coincide with that of $\tilde{\mathbf{c}}$. 
    \end{enumerate}
\end{lemma}

\begin{proof}
Note that $a^{(l)}=a^{p^l}$ for a Teichmüller lift. For part (a),  we can reduce to showing that if $u$ is an entry of $2\mathbf{U}_{n}$, then $u^{(l)}\equiv u^{p^l}(\bmod p^l)$. This follows from the explicit formulae in Lemma~\ref{lm:Finftyexplicit}. Part (b) and part (c) are direct consequences of Lemma~\ref{lm:Finftyrecursive}. Part (d) follows from (c) and the change of basis $w' = \frac{1}{2\lambda p}(\lambda e' + f'), v' = \frac{1}{2\lambda}(\lambda e' - f')$.
\end{proof}
Next we will show that $D_n(v)$ controls the deformation of $v$. 
\begin{lemma}\label{T:PDlemma} Consider a map $\iota: \Spf\mathbb{F}[\![t]\!]\rightarrow \Spf\mathbb{F}[\![x_i,y_i]\!]$. Lift $\iota$ to a map $\Spf W[\![t]\!]\rightarrow \Spf W[\![x_i,y_i]\!]$ and by abuse of notation, still call it $\iota$. Let $I=(t^s), 0\leq s\leq \infty$ be an ideal of $\mathbb{F}[\![t]\!]$, and $\tilde{I}=(p,t^s)$ be the preimage of $I$ in $W[\![t]\!]$. A special endomorphism $v$ deforms to $\mathrm{Spf}(\mathbb{F}[\![t]\!]/I)$ if and only if ${\iota}^*\tilde{v}\in\mathbb{L}(W)\otimes \hat{D}_{I}$, where $\hat{D}_I=\hat{D}_{W[\![t]\!]}(\tilde{I})\subseteq K[\![t]\!]$ is the $p$-adic completion of the PD-envelope of $\tilde{I}$.
\end{lemma}
\proof We use the standard terminology ``MIC'' for module with integrable connection. Recall that $\widehat{\bL}:=\bL_{\mathscr{S}^{/P}_{\mathbb{F}}}$ can be identified with a filtered Frobenius MIC over  $W[\![x_i,y_i]\!]$, i.e., a quadruple
$(\mathbb{L}(W)\otimes W[\![x_i,y_i]\!],  \nabla,\Fil^\bullet\mathbb{L}(W)\otimes W[\![x_i,y_i]\!], F)$. Let $\widehat{\mathbb{L}}_I$ be the base change of $\bL_{\mathscr{S}^{/P}_{\mathbb{F}}}$ to $\hat{D}_I$ along the map $W[\![x_i,y_i]\!]\rightarrow W[\![t]\!]\hookrightarrow \hat{D}_I$, which is again a quadruple
$(\mathbb{L}(W)\otimes \hat{D}_I, \nabla\otimes \hat{D}_I, \Fil^\bullet\mathbb{L}(W)\otimes \hat{D}_I,  F)$.

By \cite[\S 2.2, \S 2.3]{DJ95}, there is a natural functor from the category of $p$-divisible groups over $\Spf(\mathbb{F}[\![t]\!]/I)$ to Dieudonné modules over $\hat{D}_I$. Since $\mathbb{F}[\![t]\!]/I$ is a complete intersection scheme, \cite[Theorem 4.6]{DJ99}
 further implies that the aforementioned natural functor is fully faithful. As a result, there exists a formal special endomorphism of $\mathcal{A}_{\Spf(\mathbb{F}[\![t]\!]/I)}[p^\infty]$ deforming $v$, if and only if there exists an endomorphism of the corresponding Dieudonné modules over $\hat{D}_I$ deforming $v$, if and only if there is a horizontal section of $\widehat{\mathbb{L}}_I$ extending $v$. 

Recall that $\tilde{v}$ is the unique horizontal section of the MIC $(\mathbb{L}(W)\otimes K[\![x_i,y_i]\!],\nabla\otimes K[\![x_i,y_i]\!])$ extending $v$, and ${\iota}^*\tilde{v}$ is the unique horizontal section of the MIC $(\mathbb{L}(W)\otimes K[\![t]\!],\nabla\otimes K[\![t]\!])$ extending $v$. If $v$ deforms to $\Spf(\mathbb{F}[\![t]\!]/I)$, then there is a horizontal section $v_I$ of $\widehat{\mathbb{L}}_I$  extending $v$. Then the base change of $v_I$ along the natural embedding $\hat{D}_I\hookrightarrow K[\![t]\!]$ identifies with  ${\iota}^*\tilde{v}$. In other words, ${\iota}^*\tilde{v}=v_I\in \mathbb{L}(W)\otimes \hat{D}_I$. Conversely, if 
${\iota}^*\tilde{v}\in \mathbb{L}(W)\otimes \hat{D}_I$, the same argument shows that it must be the  horizontal section of $\widehat{\mathbb{L}}_I$  extending $v$.  $\hfill\square$
\begin{corollary}\label{cor:defloci}Consider a map $\iota: \Spf \mathbb{F}[\![t]\!]\rightarrow \Spf \mathbb{F}[\![x_i,y_i]\!]$ that lies in the non-ordinary stratum. Let $\alpha$ be the degree of the series $\iota^*\overline{D}_1(v)\in \mathbb{F}[\![t]\!]$. Then a formal special endomorphism $v$ deforms to  $\Spf(\mathbb{F}[\![t]\!]/t^{\alpha})$ but not $\Spf(\mathbb{F}[\![t]\!]/t^{\alpha+1})$.
\end{corollary}
\begin{proof}
Let $I=(t^s)$ be an ideal of $\mathbb{F}[\![t]\!]$. It suffices to show that $v$ lifts to $\Spf(\mathbb{F}[\![t]\!]/I)$ if and only if $\Spf (\mathbb{F}[\![t]\!]/I)$ factors through $\Spf(\mathbb{F}[\![x_i,y_i]\!]/\overline{D}_1(v))$.

Let $n\geq 2$. Combining Lemma~\ref{lm;recursive2} (a) and (b), and note that $i_{2n-2}\geq 2n-3$, we have: 
\begin{equation}\label{eq:3434}
D_{n}(v)=\sum_{\substack{0\leq i_1<i_2<...<i_{2n-2}\\i_1 \text{ even}}}\delta_{i_2<i_3<...<i_{2n-3}}Q_{i_1<i_2<...<i_{2n-2}}[D_1(v)]^{p^{1+i_{2n-2}}}+O(p^{2n-2}).   
\end{equation}
Lift $\iota$ to a map $\Spf W[\![t]\!]\rightarrow \Spf W[\![x_i,y_i]\!]$ and still call it $\iota$. If $\overline{D}_1(v)\in I$, then $\iota^*D_1(v)\in (p,I)=(p,t^s)$. Take $\iota^*$ on both sides of (\ref{eq:3434}), use the fact that $i_{2n-2}\geq 2n-3$, then plug in Lemma~\ref{lm;recursive2}(d). We find that all entries of $\iota^*\tilde{\mathbf{c}}'$ lie in the PD envelope $\hat{D}_{I}$. This shows that $\iota^*\tilde{v}\in\bL(W)\otimes \hat{D}_{I}$. Apply Lemma~\ref{T:PDlemma} to conclude that $v$ lifts to $\Spf (\mathbb{F}[\![t]\!]/I)$. Conversely, suppose that $\overline{D}_1(v)\notin I$. A similar argument shows that the pullback under $\iota$ of the first entry of $\tilde{\mathbf{c}}$ in Lemma~\ref{lm;recursive2}(d) already forces $\iota^*\tilde{v}$ to lie outside of $\bL(W)\otimes\hat{D}_{I}$. So $v$ does not lift to $\Spf (\mathbb{F}[\![t]\!]/I)$.
\end{proof}
\begin{remark}\label{rmk:exlicitx_v}
    The corollary shows, in particular, that we can take $x_v=\overline{D}_1(v)$, where $x_v$ is the series in Remark~\ref{eq:v_defspace}. This provides us with an explicit recipe of computing $x_v$.
\end{remark}

\subsection{Setting up the local curve}\label{sub:setuplocalcurve} 

Throughout, our curve $C$ will be contained in the non-ordinary Newton strata. Therefore, the running assumption will be that $Q_0 = 0$. For the decay calculation, we will choose a uniformizer so that $C=\Spf \bF[\![t]\!]$. The equation of the local curve is give by \begin{equation}\label{eq:localequation1}
    C:t\rightarrow (x_1(t),...,x_m(t);y_1(t),...,y_m(t)),\;\;x_i(t),y_i(t)\in \mathbb{F}[\![t]\!].
\end{equation}

As in \cite{MST1}, we will also use  $x_i(t), y_i(t)$ to denote the power series in  $W[\![t]\!]$ where all the coefficients are Teichmuller lifts of the coefficients of $x_i(t), y_i(t)$ in (\ref{eq:localequation1}). This abuse of notation won't cause confusion if one keeps track of which base ring they work in. To ease notation, we often just write $x_i,y_i$ for the corresponding power series in $t$. 

\begin{notation}
\label{not: define h_i and a_i}
    We will assume that $Q_0(t), \hdots Q_{a-1}(t) = 0$ and that $Q_a(t)$ is the first non-zero function of $t$. Let $h_i := v_t(Q_i(t))$. Then $h_a$ matches up with $h_P$ as per \S\ref{sub: intersection numbers}. Note that we have $h_i = \infty$ for $i< a$. Recursively define a sequence of integers $a_1<a_2<...$, with $a_1=a$, and $a_{i+1}$ is the smallest integer bigger than $a_{i}$ such that $h_{a_{i+1}}\leq h_{a_{i}}$. Let $a_k$ be the biggest number in the sequence that is smaller than $m$ (in fact, the sequence terminates at $a_k$). 
    Throughout this section, the notation $a_i$ will exclusively refer to the numbers listed above.
\end{notation}

As the $a_i$'s play a pivotal rule in our analysis, we draw a picture for readers' convenience. If we draw a bar chart for $h_{a}, h_{a + 1}, \cdots$, then the $a_i$'s look like the following. The dashed lines will be ``thrown away'' in our analysis. 

\tikzset{every picture/.style={line width=0.75pt}} 

\begin{figure}[H]
    \centering
    \begin{tikzpicture}[x=0.75pt,y=0.75pt,yscale=-1,xscale=1]

\draw    (165,202.38) -- (471,203.39) ;
\draw [shift={(473,203.4)}, rotate = 180.19] [color={rgb, 255:red, 0; green, 0; blue, 0 }  ][line width=0.75]    (10.93,-3.29) .. controls (6.95,-1.4) and (3.31,-0.3) .. (0,0) .. controls (3.31,0.3) and (6.95,1.4) .. (10.93,3.29)   ;
\draw  [line width=1.5]  (185,99.6) -- (185,201.35) ;
\draw  [line width=1.5]  (227,120.65) -- (227,201.35) ;
\draw  [line width=1.5, dash pattern={on 4.5pt off 4.5pt}]  (264.49,107.36) -- (263.77,202.9) ;
\draw   [line width=1.5] (308,144.14) -- (308,202.21) ;
\draw  [line width=1.5, dash pattern={on 4.5pt off 4.5pt}]  (348.98,117.58) -- (348.98,202.38) ;
\draw   [line width=1.5] (390,171.73) -- (390,203.4) ;
\draw  [line width=1.5, dash pattern={on 4.5pt off 4.5pt}]  (430.75,133.93) -- (432.12,202.38) ;

\draw (174.75,203) node [anchor=north west][inner sep=0.75pt]   [align=left] {$a_1$};
\draw (216.17,203) node [anchor=north west][inner sep=0.75pt]   [align=left] {$a_2$};
\draw (295.82,203) node [anchor=north west][inner sep=0.75pt]   [align=left] {$a_3$};
\draw (377.6,203) node [anchor=north west][inner sep=0.75pt]   [align=left] {$a_4$};

\end{tikzpicture}
    \caption{A picture for the $a_i$'s}
    \label{fig: the a sequence}
\end{figure}
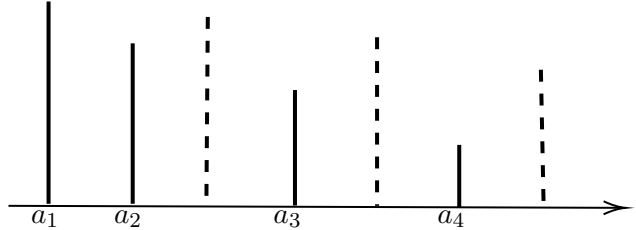

\subsubsection{Setup with special endomorphisms}\label{subsub:setupspecialendo}
In this section, we always assume that $v\in\cL$ is a primitive formal special endomorphism, i.e., $p^{-1} v\notin \cL$. Write $\tilde{v}=F_\infty v$. 

\begin{definition}
\label{def: nth order rate of decay}
For every $n\geq 1$, let $d_n(v)$ be the smallest integer such that the coefficients of $1=t^0,...,t^{d_n(v)}$ in the power series $p^{n-1}\tilde{v}$ with respect to the basis $\{e',f';e_i,f_i\}$ do not all lie in $W$ (if such an integer does not exist, we define $d_n(v)=\infty$).  We call $d_n(v)$ the \textit{$n$-th order rate of decay} for $v$.   
\end{definition}

The concrete meaning of $d_n(v)$ is given in the following lemma. The quantity turns out to be equal to $v_t(x_{p^{n-1}v})$ and $v_t(\overline{D}_n(v))$. However, as compared to $v_t(x_{p^{n-1}v})$, its description is more concrete; as compared to $v_t(\overline{D}_n(v))$, its definition is more conceptual. This makes it easier to use in many situations. 

\begin{lemma}\label{lm:specialendoprepair} \begin{enumerate}[label=\upshape{(\alph*)}]
 \item $p^{n-1}v$ lifts to $\Spf \mathbb{F}[\![t]\!]/(t^{d_n(v)})$ but not $\Spf \mathbb{F}[\![t]\!]/(t^{d_n(v)+ 1})$. 
 \item $d_n(v)=v_t(x_{p^{n-1}v})=v_t(\overline{D}_n(v))$. 
\item  $d_{n+1}(v)\geq 
\min_{i\geq a}\{h_i+p^{i+1}d_n(v)\}$.
    \end{enumerate}
\end{lemma}
\begin{proof} By definition, we have $D_1(p^{n-1}v)=D_n(v)$. So we can reduce part (a) and (b) to the case where $n=1$. For this, we can assume that $d_1(v)<\infty$. From the explicit formulae in Lemma~\ref{lm;recursive2}(c), the $t^{d_1(v)}$ term in the first entry of $\tilde{\mathbf{c}}$ has coefficient $p^{-1}u,u\in W^*$. While for any $d<d_1(v)$, the term $t^d$ has integral coefficient. This implies that $v_t(\overline{D}_1(v))=d_1(v)$. Given this, part (a) follows from  Corollary~\ref{cor:defloci}, and part (b) follows from Remark~\ref{rmk:exlicitx_v}. Part (c) follows from the $r=1$ case in Lemma~\ref{lm;recursive2} (b). 
\end{proof}
Suppose that $n\geq 2$.  It is helpful to get a feeling of what $\overline{D}_n(v)$ should be, and how large $d_n(v)$ can get. By recursion formula Lemma~\ref{lm;recursive2}(b) and Remark~\ref{rmk:exlicitx_v}, $\overline{D}_n(v)$ is a huge linear combination of terms of the following form
\begin{equation}\label{eq:term2}
  w = Q_{i_1<i_2<...<i_{2r}}\,x_v^{(1+i_{2r})}, \,\,\,\,\,\,r\geq n-1 \text{ and }2|i_1
\end{equation}
We can compute the $v_t$-valuation of a term in (\ref{eq:term2}) as \begin{equation}\label{eq:termdeg3}
     v_t(w(t)) = \sum_{k=1}^{r} p^{i_{2k-1}}h_{i_{2k}-i_{2k-1}}+p^{1+i_{2r}}v_t(x_v).
\end{equation} 
Let $H$ be the minimal $v_t$-valuation among all the terms in (\ref{eq:term2}). To estimate how large $d_n(v)$ is, the first step is to look at all possible terms in (\ref{eq:term2}) with $v_t$-valuation $H$. 

 \begin{lemma}
 \label{lm:wordtruncate}
Every term $w$ as in (\ref{eq:term2}) that achieves $v_t(w) = H$ satisfies the following conditions.
    \begin{enumerate}[label=\upshape{(\alph*)}]
        \item  $r=n-1$, $i_1=0$.
         \item  for $k\geq 1$, 
                $i_{2k+1}-i_{2k}=1$, 
        \item for $k\geq 1$, $i_{2k}-i_{2k-1} \in \{ a_1, \cdots, a_{k} \}$.
    \end{enumerate}
 \end{lemma}
 
 \begin{proof}
Consider a term  $w=Q_{i_1<i_2<...<i_{2r}}\,x_v^{(1+i_{2r})}$ in (\ref{eq:term2}). If $r>n-1$, then by the explicit computation (\ref{eq:termdeg3}), $w'=Q_{i_1<i_2<...<i_{2(r-1)}}\,x_v^{(1+i_{2(r-1)})}$ has strictly smaller $v_t$-valuation than $w$. So $w$ can not have $v_t$-valuation $H$. If $r=n-1$ but $i_1>0$, then the term $w''=Q_{0<i_2-i_1<...<i_{2r-i_1)}}\,x_v^{(1+i_{2r}-i_1)}$ has strictly smaller $v_t$-valuation than $w$. Again, $w$ can not have $v_t$-valuation $H$. This proves (a). Part (b) can be proved similarly. 

For part (c), we first show that $a \le i_{2k}-i_{2k-1}$. Since we assumed that $Q_i(t)=0$ for $i<a$, if $w$ is such that $i_{2k}-i_{2k-1}<a$ for some $k$, then $v_t$-valuation of $w$ is $\infty$. To see that $i_{2k}-i_{2k-1} \in \{ a_1, \cdots, a_{k} \}$, let us assume that this is false for some $k$. Then from \cref{fig: the a sequence} one easily sees that for some $a' \in \{ a_1, \cdots, a_{k} \}$, $a' < i_{2k}-i_{2k-1}$ and $h_{a'} \le h_{i_{2k}-i_{2k-1}}$. Set $\delta := (i_{2k}-i_{2k-1}) - a' > 0$. Now we set a new sequence $i'_1 < i'_2 < \cdots < i'_{2r}$ by 
$$ i_j' = \begin{cases}
    i_j &\text{ for } j \le 2k - 1 \\
    i_j - \delta &\text{ for } j \ge 2k. 
\end{cases}$$
Then one readily sees that the new $w'$ defined by $i'_1 < i'_2 < \cdots < i'_{2r}$ achieves smaller $v_t$. 
 \end{proof}

\begin{remark}\label{rmk:H}
Clearly, $d_n(v)\geq H$. In generic cases, we even have $d_n(v)=H$. However, since $\overline{D}_n(v)$ is a linear combination of the terms in (\ref{eq:term2}), it can happen that multiple terms with minimal $v_t$-valuation cancel out, leaving us $d_n(v)> H$. This makes the estimation of $d_n(v)$ much more subtle. Nevertheless, this essentially a complicated combinatorial problem which we solve in the next section.
\end{remark}

\subsection{Precise definition of rapid decay}\label{Sec: word properties decay defn}
In this section, we study the terms in (\ref{eq:term2}) that achieve the minimal $v_t$-valuation $H$; cf. Lemma~\ref{lm:wordtruncate}. In order to be able to work with the combinatorics in a flexible way, we introduce the following notions: 

\begin{definition}
\label{def: bI_r's}
    \begin{enumerate}[label=\upshape{(\roman*)}]
\item Let $\bI_r$ be the set of all words in $Q_a, Q_{a+1},..., Q_{m-1}$ of length $r$.

\item Let $v\in \cL$ be a primitive formal special endomorphism. 
\begin{equation}
    \bI_r(v) = \begin{cases}
        \{ \textrm{words of the form } \mathcal{W}x_v : \mathcal{W}\in \bI_{r-1} \} &\textrm{ if } v \bmod p \in \cL_1/p \\
        \qquad \qquad \qquad \qquad  \bI_r &\textrm{ if } v \bmod p \in \cL_0 /p
    \end{cases}
\end{equation}

 \item Given $\cW\in \bI_r$, $\cW[r']$ is the sub-word consisting of the last $r'$ letters of $\cW$. $[r']\cW$ is the sub-word consisting of the first $r'$ letters of $\cW$. $\cW_i$ is the $i$th letter (from the left) of $\cW$.

\item Define a function $\nu : \bI_r \to \bN$ recursively by setting 
\begin{equation}
    \nu(Q_{i}) = h_i \textrm{ and }  \nu(Q_i\cW)= h_{i}+p^{i+1}\nu(\cW).
\end{equation}
 For $\bI_r(v)$ defined in (ii), define $\nu : \bI_r(v) \to \bN$ by adding the rule $\nu(x_v) = v_t(x_v)$. 
    \item Let $\bI^{\min}_r$ (resp. $\bI^{\min}(v)$) be the set of all words in $\bI_r$ (resp. $\bI_r(v)$ in (ii)) having the minimal valuation. Write $\nu^{\min}_r$ (resp. $\nu_r^{\min}(v)$) for this minimal valuation.  
\end{enumerate}
\end{definition}

\begin{lemma}
\label{lem: w and W}
    There is a bijective correspondence between the terms in $\mathcal{I}^{\min}(v)$ and the words in $\bI^{\min}(v)$ given by 
    \begin{equation}
        \label{eqn: w and W}
        w = Q_{i_1<i_2<...<i_{2r}}\,x_v^{(1+i_{2r})} \Leftrightarrow \cW = Q_{u_1}Q_{u_2}...Q_{u_{n-1}}x_v \textrm{ where } u_i = i_{2k} - i_{2k - 1}.
    \end{equation}
\end{lemma}
\begin{proof}
    One can recover $w \in \mathcal{I}^{\min}(v)$ from $\cW$ by setting 
    \begin{equation}
    \label{eqn: W -> w}
        i_1 = 0, i_2 = u_1, i_{2k-1} = \big(\sum_{j = 1}^{k - 1} u_j \big) + (k - 1), i_{2k} = \big(\sum_{j = 1}^k u_j \big) + (k - 1) \text{ for } k \ge 2. 
    \end{equation}
    by \cref{lm:wordtruncate}. 
\end{proof}

\begin{notation}
    For every $\cW \in \bI_r$ or $\bI_r(v)$, we define $\cW(t) := w(t) \in K[\![t]\!]$ for $w$ defined by (\ref{eqn: W -> w}). Then $\nu(\cW) = v_t(\cW(t))$. The $H$ in \cref{lm:wordtruncate} is simply $\nu^{\min}(v)$. For every $\cW \in \bI_r$ or $\bI_{r + 1}(v)$ and $j \le r$, we write $a(\cW_j)$ for the index such that $\cW_j = Q_{a(\cW_j)}$. 

    \end{notation}

\begin{proposition}
\label{prop: lower bound for d_r}
    For any primitive formal special endomorphism $v$, we have $d_r(v)\geq \nu_r^{\min}(v)$.
\end{proposition}

\begin{proof}
    For $v \bmod p \in \cL_1/p$, this is just a rephrase of the first sentence in Remark~\ref{rmk:H}. For $v \bmod p \in \cL_0/p$, we additionally need to check that $d_1(v) \ge h_{a_k}$. Recall that $h_{a_k} = \min \{ v_t(Q_j) : j \in \bN \}$. The conclusion follows by inspecting the formula for $D_1(v)$ in \cref{not: v tilde +} when only the first two entries of $\mathbf{c}$ are possibly nonzero. 
\end{proof}

The argument above also explains why in the above definition we treated the cases $v \bmod p$ lie in $\cL_0 / p$ or $\cL_1/ p$ separately.

\begin{definition}
\label{def: rapid decay}
    A primitive special endomorphism $v\in \cL$ is said to \textit{decay rapidly to $r$th order} provided that 
    \begin{equation*}
        \begin{cases}
         d_r(v) = \nu_r^{\min}(v) \text{ and } v_t(x_v) < h_{a_k} & \text{ when } v \bmod p \in \cL_1 / p \\
         d_r(v) = \nu_r^{\min} (\text{in particular, } v_t(x_v) = h_{a_k}.) & \text{ when } v \bmod p \in \cL_0 / p 
    \end{cases}
    \end{equation*}
     We say that a saturated submodule $\Lambda\subseteq \cL$ \textit{decays rapidly to $r$th order} if every primitive special endomorphism $v\in \Lambda$ decays  rapidly to $r$th order.
\end{definition}

The lemma below is a basic reality check for the definition above. 
\begin{lemma}
\label{lem: order s < r}
    If a primitive formal endomorphism $v \in \cL$ decays rapidly to $r$th order, then it decays rapidly to $s$th order for any $s \le r$. 
\end{lemma}
\begin{proof}
    It suffices to show that if $v$ does not decay rapidly to $r$th order, then it does not decay rapidly to $(r+1)$th order. Suppose that $v \bmod p \in \cL_0 / p$ and $v$ does not decay rapidly to $r$th order, so $d_r(v)> {\nu_r^{\min}}$. Then by Lemma~\ref{lm:specialendoprepair}, there exists $i$ such that  
    $$d_{r+1}(v)\geq 
    h_i+p^{i+1}d_r(v)> h_i+p^{i+1}\nu_r^{\min}\geq {\nu_{r+1}^{\min}}.$$
    So $v$ does not decay rapidly to $(r+1)$th order. The proof for $v \bmod p \in \cL_1 / p$ is similar. 
\end{proof}

\begin{remark}\label{rem: ak is a}
    Suppose that $a_k = a$. In this setting, $\bI_r^{\min}(v)$ is a singleton, and therefore $d_r(v) = \nu_r^{\min}(v)$. It follows that in this special case, every vector $v$ that satisfies $v \bmod p \in \cL_1 / p $ that decays rapidly to first order must decay rapidly to every order. 
\end{remark}

Recall that we have used the symbol ``\,$v'$\,'' for an element in the basis $\{v',w';e_i,f_i\}$ of $\bL(W)$. As we won't use this basis again, in the rest of this section, we will abuse the symbol ``\,$v'$\,'' to mean a special endomorphism in $\cL$.

\begin{lemma}
\label{lem: compare words}
    Let $v,v'$ be primitive formal special endomorphisms in $\cL$. Let $\cW \in \bI^{\min}_{r+1}(v)$ and $\cW' \in \bI_{r+1}^{\min}(v')$. 
    \begin{enumerate}[label=\upshape{(\alph*)}]
        \item\label{lem: compare words 1} If $v_t(x_v) > v_t(x_{v'})$, then $\nu(\cW) > \nu(\cW')$. 
        \item\label{lem: compare words 2} If both $v$ and $v'$ decay rapidly to first order, then for every $i < j \le r$, $a(\cW_i) \le a(\cW'_j)$. 
        \item\label{lem: compare words 3} If both $v$ and $v'$ decay rapidly to first order and $v_t(x_v) > v_t(x_{v'})$, then for every $i \le r$, $a(\cW_i) \le a(\cW'_i)$. 
    \end{enumerate}
\end{lemma}
\begin{proof}
    Part~\ref{lem: compare words 1} is clear. The proofs for part \ref{lem: compare words 2} and \ref{lem: compare words 3} are entirely similar so we only explain the former. As $\cW[r] \in \bI^{\min}_r(v)$, $\cW'[r] \in \bI^{\min}_r(v')$ and we allow $\cW = \cW'$, by induction we may always assume that $i = 1$ and $j = 2$. The base case $r = 2$ follows from the assumption that $\nu(x_{v'}) \le h_{a_{k}} \le h_{a(\cW_2)}$. Define $\alpha = a(\cW_1)$, $\alpha' = a(\cW'_1)$, $\nu = \nu(\cW[r])$ and $\nu' = \nu(\cW'[r - 1])$. Note that by induction we have $\nu' \le \nu$. By the minimality of $\cW = Q_{\alpha} \cdot (\cW[r])$ and $\cW'[r] = Q_{\alpha'} \cdot (\cW'[r -1])$, if we swap $Q_{\alpha}$ with $Q_{\alpha'}$, the $\nu$-value may only increase. This leads to inequalities:
    \begin{align*}
        h_{\alpha'} + p^{\alpha'} \nu' \le h_{\alpha} + p^{\alpha} \nu' \le h_\alpha + p^\alpha \nu \le h_{\alpha'} + p^{\alpha'} \nu 
    \end{align*}
    Therefore, $p^{\alpha'}(\nu - \nu') \ge p^\alpha (\nu - \nu')$. As $\nu - \nu' \ge 0$, we obtain $\alpha' \ge \alpha$ as desired. 
\end{proof}

We have the following result.

\begin{proposition}
\label{prop: prelim decay results}
    Let $v,v' \in \cL$ be two primitive special endomorphisms. 
    \begin{enumerate}[label=\upshape{(\alph*)}]
        \item\label{prop: prelim decay results 1} If $v$ decays rapidly to $r$th order and $v \equiv v' \bmod p$, then $d_r(v) = d_r(v')$. In particular, $v_t(x_v) = v_t(x_{v'})$ and $v'$ also decays rapidly to $r$th order. 
        \item\label{prop: prelim decay results 2} If both $v$ and $v'$ decays rapidly to $r$th order, and $v_t(x_v) \neq v_t(x_{v'})$, then so does any primitive $v'' \in \cL$ that is a linear combination of $v, v'$. 
    \end{enumerate}
\end{proposition}
\begin{proof}
    \ref{prop: prelim decay results 1} Suppose that $v = v' + p u$ for $u \in \cL$. It suffices to show that the coefficients of $t^0,...,t^{\nu_r^{\min}}$ in the power series $p^r\tilde{u}$ all lie in $W$. Let $pu=p^{\gamma}u'$, where $\gamma\geq 1$ and $u'$ is primitive. It suffices to show that  $\nu_r^{\min}<w_{r+\gamma}^{\min}(u')$. Let $\cW\in\bI_{r+\gamma}^{\min}(u')$, then $[r]\cW\in \bI_r$. We have $$\nu_r^{\min}\leq \nu([r]\cW)< \nu(\cW) = \nu_{r+\gamma}^{\min}(u')$$ 
    as desired. In particular, when $r = 1$, we have $v_t(x_v) = v_t(x_{v'})$. The statement that $v'$ also decays rapidly to $r$th order is clear from definition. 

    \ref{prop: prelim decay results 2} Let $v'' =\alpha v+ \beta v'$ be a primitive vector. Then at least one of $\alpha,\beta$ lies in $\bZ_p^*$. If one of $\alpha$ and $\beta$ lies in $p\bZ_p$, we reduce to \ref{prop: prelim decay results 1}. So we can assume that $\alpha,\beta \in \bZ_p^*$. Without loss of generality, assume that $v_t(x_v) = d_1(v) < v_t(x_{v'}) = d_1(v')$. By \cref{def: rapid decay}, this can happen only when $v \bmod p \in \cL_1 / p$. No matter what $v'$ is, we have $d_r(v)=\nu^{\min}_r(v)$ and $d_r(v')=\nu^{\min}_r(v')$ by definition. It follows from \cref{lem: compare words} that 
    $$d_r(v)=\nu^{\min}_r(v)<\nu^{\min}_r(v')=d_r(v').$$
    So the coefficients of $t^0,...,t^{d_r(v)}$ in the power series $p^{r-1}\widetilde{v}$ do not all lie in $W$, while the coefficients of $t^0,...,t^{d_r(v)}$ in the power series $p^{r-1}\widetilde{v'}$ all lie in $W$. It follows that the  coefficients of $t^0,...,t^{d_r(v)}$ in the power series $p^{r-1}\widetilde{v''}$ do not all lie in $W$. Therefore, $d_r(v'') = d_r(v)$. So the claim holds.
\end{proof}
 
Part~\ref{prop: prelim decay results 2} of the above has the following consequence. 
\begin{corollary}\label{cor: decay linear combination L1 L1}
    If $v_0\in \cL_0$ and $v_1 \in \cL_1$ are primitive vectors that decay rapidly to $r$th order, then so does every linear combination of $v_0$ and $v_1$ that is a primitive vector in $\cL$. 
\end{corollary}

\begin{definition}
\label{def: intervals I_r(v)}
    Let $v$ be a primitive formal special endomorphism that decays rapidly up to first order. We define $I_r(v) \subseteq \bN$ to be the interval 
    $$ I_r(v) = \Big[\min_{\cW \in \bI_r^{\min}(v)} a(\cW_1)  , \max_{\cW \in \bI_r^{\min}(v)} a(\cW_1) \Big], $$
    and write $|I_r(v)|$ for the length of this interval. 
\end{definition}
Note that for any $r \ge s$, 
$$ I_s(v) = \Big[\min_{\cW \in \bI_r^{\min}(v)} a(\cW_{r - s + 1})  , \max_{\cW \in \bI_r^{\min}(v)} a(\cW_{r - s + 1}) \Big]. $$

We have the following results about intervals.

\begin{proposition}\label{prop: intervals}
    Let $v$ and $v'$ be primitive formal endomorphisms decay rapidly up to first order. 
    \begin{enumerate}[label=\upshape{(\alph*)}]
        \item\label{prop: intervals 1} If $v_t(x_v) \neq v_t(x_{v'})$, $I_r(v) \cap I_s(v')$ is either empty or a singleton for any $r,s \in \bN$. 

        \item\label{prop: intervals 2} If $v_t(x_v) = v_t(x_{v'})$, $I_r(v) \cap I_s(v')$ is either empty or a singleton for any distinct $r\neq s \in \bN$.
    \end{enumerate}
\end{proposition}
\begin{proof}
    For \ref{prop: intervals 1}, suppose that $v_t(x_v) > v_t(x_{v'})$. Let $\cW \in \bI_{r}^{\min}(v)$, $\cW' \in \bI_s^{\min}(v')$. Suppose that $r=s$. Then, by \cref{lem: compare words}\ref{lem: compare words 3}, we have that $a(\cW_1) \le a(\cW'_1)$. That is, every integer in $I_r(v)$ is bounded above by every integer in $I_r(v')$. Hence the intervals are either disjoint, or intersect only at a singleton. For \ref{prop: intervals 2} identical argument goes through, except that one replaces \ref{lem: compare words 3} of \cref{lem: compare words} by \ref{lem: compare words 2}. 
\end{proof}

\subsection{Isotropic subspaces}\label{sec: superspecial isotropics}

We have the following result. 
\begin{lemma}\label{lemma: height valuation}
    The number of distinct $i$ such that $v_t(x_i) \geq h_{a_k}$ is at most $ m-1-a_k$. 
\end{lemma}
\begin{proof}
    The key observation is that we can find a suitable sublattice $L' \subseteq L$ defining an orthogonal Shimura subvariety $\mathscr{S}'$ of $\mathcal{S}$ that is cut out by the ideal $(x_{a_k+1} \dots x_m, \, y_{a_k+1} \dots y_m)$. Then, by applying \cref{NPstrata} to $\mathscr{S}'$, we deduce that $Q_{a_k} \in (Q_0, Q_1, \hdots Q_{a_k-1}, x_{a_k+1}, \cdots x_m)$. It follows that at least one of the numbers $v_t(x_{a_k+1}) $,..., $v_t(x_m)$ must be smaller or equal to $h_{a_k}$. We may reshuffle indicies to deduce that the number of distinct $i$ such that $v_t(x_i) > h_{a_k}$ is at most $m-1-a_k$.

    We have therefore proved that the number of distinct $i$ such that $v_t(x_i) > h_{a_k}$ is at most $m-1-a_k$. In order to prove the claim with $>$ replaced with $\geq$, write $Q_{a_k} = \sum f_i Q_i + \sum g_i x_i$ in the ring $\bF[\![x_1 \hdots x_m, y_1 \hdots y_m]\!]$. It suffices to prove that elements $g_i$ have non-zero constant term. But this is true by just comparing coefficients of each monomial on the LHS and RHS. 
    \end{proof}

Lemma \ref{lemma: height valuation} yields the following corollary. 

\begin{corollary}\label{cor: isotropic subspace first order decay}
    \begin{enumerate}[label=\upshape{(\alph*)}]
        \item Let $\cL'\subset \cL $ be a saturated sub-lattice having rank $m-a_k$ such that $\cL'/p\cL \subseteq \cL_1 / p\cL$, and such that $\cL' /p\cL$ is isotropic. Then there is at least one primitive vector $v$ that decays rapidly to first order, i.e., $v_t(x_v) < h_{a_k}$.


        \item There exists a saturated sublattice $\cL_1' \subset \cL_1$ having rank $2a_k$ such that every primitive vector decays rapidly to first order. 
    \end{enumerate}
     
\end{corollary}
\begin{proof}
    (a) As observed earlier, we may replace $\cL'$ by a sublattice that has the same mod $p$ reduction. Therefore, we will assume throughout that $\cL' \subset \cL_1$ and that $\cL'$ is already isotropic. Recall that in \S\ref{subsub:explicitcoordinatheonL}, the basis $\{e_i, f_i \}_{1 \le i \le m}$ on $\bL_1$ can be chosen to be any basis for which the Gram matrix is $\begin{bmatrix}
        & I_m \\ I_m & 
    \end{bmatrix}$. 
    Therefore, by Witt cancellation, we are allowed to assume that the basis was chosen such that $\cL'$ is the span of $e_1 \hdots e
    _{m-a_k}$. Recall that $x_i = x_{e_i}$. Then we conclude by \cref{lemma: height valuation}. 


    
    (b) Let $\cL_1'' \subseteq \cL_1$ be the saturated lattice spanned by all primitive vectors which resist decay up to first order and let $N := \dim \cL''_1$.  
    By (a), $\cL_1''$ cannot contain an isotropic subspace of dimension $\geq m-a_k$. By the classification of quadratic forms over a field in odd characteristic, $N \leq 2(m-a_k)$. Therefore, we may take $\cL_1'$ to be any saturated sublattice of $\cL_1$ that is complementary to $\cL_1''$. 
\end{proof}

\subsection{Proof of the decay Lemma}
We are now ready to state and prove the decay Lemma.
\begin{theorem}\label{thm: superspecial rapid decay}
    There is a saturated sublattice $\Lambda \subset \cL$ having rank $\geq a_1+a_k + 2$ that decays rapidly.
\end{theorem}

\begin{proof}
    Let $\cL'_1 \subset \cL_1$ be the saturated sub-lattice produced in part (b) of Corollary \ref{cor: isotropic subspace first order decay}. We will show the existence of a saturated rank-$(a_1+a_k + 2)$ sublattice $\Lambda = \Lambda_0 \oplus \Lambda_1 \subset \cL_0 \oplus \cL'_1 =: \cL'$ that decays rapidly. We first define a filtration on $\cL'/p\cL$. For $i < h_{a_k}$, let $$\overline{\Fil^i} = \{\bar{v}\in (\cL'/p\cL)\setminus \{\mathbf{0}\}: v_t(x_v) \geq i \} \cup \{ \mathbf{0} \}.$$ Define $\overline{\Fil^{h_{a_k}}} = \cL_0$, and define higher filtered pieces to be just the zero vector. Intersecting this filtration with $\cL'_1/p\cL$, we get an exhaustive filtration with the $h_{a_k}$ piece equal to zero by the hypothesis that every primitive vector in $\cL'_1$ decays to first order. We pick a splitting of this filtration and lift it to $\cL'$ such that the splitting respects the direct sum decomposition $\cL' = \cL_0 \oplus \cL'_1$. Therefore, we have written $\cL' = \oplus_{0<i \leq h_{a_k}} \cL^i$ where for $i< h_{a_k}$, every primitive vector $v\in \cL^i$ satisfies $v_t(x_v) = i$, and $\cL^{h_{a_k}} = \cL_0$.

    Fix some $\cL^i$, some integer $r$, and consider $T_r(i) = \sum_{1 < s \leq r} |I_s(v_i)|$, where $v_i$ is any primitive vector in $\cL^i$ (by definition the intervals do not depend on the choice of $v_i$). By Proposition \ref{prop: intervals}, we have that $$T_r(1) + T_r(2)+ \hdots +T_r(h_{a_k}) \leq a_k - a_1.$$ It follows that $$\sum_{i = 1}^{h_{a_k}} \max \{0, \dim \cL^i - T_r(i) \} \geq 2 a_k + 2 - a_k + a_1 = a_k + a_1 + 2.$$ Together with Proposition \ref{prop: prelim decay results}, the two claims below implies the existence of a saturated rank-$a_k + a_1 + 2$ sublattice that decays rapidly to $r$th order. 
    \begin{claim}\label{claim 1}
        Let $i< h_{a_k}$, and suppose that $\cL^i$ has rank $d_i$ with $T_r(i)< d_i$. Then, there is a rank $d_i-T_r(i)$ sublattice of $\cL^i$ that decays rapidly to $r$th order. 
    \end{claim}
    \begin{proof}
      It suffices to prove the claim with $d_i-T_r(i) = 1$ -- indeed, proving that any saturated rank-$T_r(i)+1$ subspace contains a primitive vector that decays rapidly up to $r$th order would imply that any saturated rank $T_r(i)+2$ subspace must contain a saturated rank-2 subspace that decays rapidly up to $r$-th order, etc. 
      
      We may assume that $\cL^i$ is spanned by $v_1, ..., v_{d_i}$. Let $\alpha_1, \hdots, \alpha_r \in \bF_q$ be the coefficients of the leading order terms of $x_{v_j}(t)$ for $j = 1\hdots d_i$.  We now consider the power-series expansion for the vector $\tilde{v_j}$ - specifically, the first two entries of this vector. Consider the sum of the $\cW(t)$ where $\cW$ ranges over all words in $\bI_r^{\min}(v_j)$. It suffices to prove that this sum has $t$-adic valuation equal to $\nu_r^{\min}(v_j)$ for some $1\leq j \leq d_i$. Note that as $v_t(x_{v_j})$ is the same for all $j$, $\bI_r^{\min}(v_j)$ does not depend on $j$. Let $\cW \in \bI_r^{\min}(v_j)$, and let $M_\cW$ be the number such that the function $\cW(t)$ equals $([r-1]\cW(t)) (x_{v_j})^{(M_{\cW})}(t)$. In other words, under the equivalence of \cref{lem: w and W}, $M_\cW = 1 + i_{2r} = r + \sum_{j = 1}^r a(\cW_j)$. There is a unique word $\cW_{\min} \in \bI_r^{\min}(v_j)$ (resp. $\cW_{\max} \in \bI_r^{\min}(v_j)$) which minimizes (resp. maximizes) the integer $M_{\cW}$. Indeed, they are constructed by minimizing or maximizing each of $a(\cW_s)$, i.e., 
      $$  \bI_{r - s + 1} = \Big[\min_{\cW \in \bI_r^{\min}(v)} a(\cW_{s})  , \max_{\cW \in \bI_r^{\min}(v)} a(\cW_{s}) \Big] = \big[ a(\cW^{\min}_s) , a(\cW^{\max}_s)\big] $$
      for each $1 < s \le r$. Note therefore that $T_r(i) = M_{\cW^{\max}} - M_{\cW^{\min}}$. 


Now, for every $M$ in the range $M_{\cW^{\min}} \leq M \leq M_{\cW^{\max}}$ consider all words with the property $M_\cW = M$. The sum of these words equals $(\sum [r-1]\cW)(t) (x_{v_j})^{(M)}(t)$. We write the sum $(\sum [r-1]\cW)(t)$ as $\beta_Mt^{N_M} + t^{N_M+1}(...)$, where $\beta_M$ might equal zero. Note that $N_M$ is just the valuation of $[r-1]\cW$. We see that the vector $v_j$ decays rapidly to $r$th order if and only if $\sum_{M=M_{\cW^{\min}}}^{M_{\cW}^{\max}} \beta_M \alpha_j^{(M)} \neq 0$. 

We now use the fact that the rank of $\cL'$ (i.e. $d_i$) is one greater than $T_r(i) = M_{\cW^{\max}} - M_{\cW^{\min}}$. Suppose that every $v_j$ fails to decays rapidly to $r$th order. Then, it follows that $\sum_{M=M_{\cW^{\min}}}^{M_{\cW}^{\max}} \beta_M \alpha_j^{(M)} = 0$ for every $j$. Now, consider the matrix whose $jth$ row is $\alpha_j^{(M_{\cW^{\min}})}, \alpha_j^{(M_{\cW^{\min}}+1)}, \hdots \alpha_j^{(M_{\cW^{\max}})}$. The matrix has $d_{i}$ rows and $M_{\cW^{\max}} - M_{\cW^{\min}} +1 = T_r(i)+1= d_i$ columns. This is (the Frobenius twist of) the so-called Moore matrix, whose determinant is non-zero precisely when the $\alpha_j$ are $\bF_p$-linearly independent. As every vector in $\cL'$ decays rapidly up to first order, we have that the $\alpha_i$ are $\bF_p$-linearly independent! Therefore, the determinant of the matrix is non-zero, and hence the only way we have $\sum_{M=M_{\cW^{\min}}}^{M_{\cW}^{\max}} \beta_M \alpha_j^{(M)} = 0$ is for every $\beta_M = 0$. But $\beta_{\cW^{\min}} \neq 0, \beta_{\cW^{\max}} \neq 0$! We have therefore proved that for every integer $r$, there has to be at least one vector that decays rapidly to $r$th order. As we have already reduced the claim to the case $T_r(i)+1 = d_i$, the proof is finished. \end{proof}

    \begin{claim}\label{claim 2}
        Suppose that $T_{r}(h_{a_k}) < \textrm{rank} \cL_0$. Then, there is a rank $(\textrm{rank} \cL_0-T_r(h_{a_k}))$ sublattice of $\cL_0$ that decays rapidly.
    \end{claim}
\begin{proof}
    The setting of this claim is somewhat degenerate, as $\cL_0$ is only a rank 2 lattice. There are only two cases -- $T_{r}(h_{a_k})$ is 0 or 1. The case of $T_r(h_{a_k}) = 0$ is easy -- indeed, in the sum-expansion of the matrix $X_r(t)$ in Lemma \ref{lm:Finftyexplicit} has a unique term with minimal $t$-adic valuation and the claim holds. In the case where $T_{r}(h_{a_k}) = 1$, there are exactly two terms  in the sum-expansion of $X_r$ with minimal $t$-adic valuation. The sum of these two matrices (without counting powers of $p$) will be of the form $\cW_1(t)A + \cW_2(t)B$, where the first row of $A$ is $[1 \phantom{a}\lambda]$ and of $B$ is $[1 \phantom{a} -\lambda]$. Either $\cW_1(t) +\cW_2(t)$ or $\cW_1(t) - \cW_2(t)$ has the form $\alpha t^{\nu_r^{\min}} + O(t^{\nu_r^{\min}+1})$ where $\alpha$ is a $p$-adic unit, in which case either $e + f$ or $e-f$ decays rapidly to $r$th order. The claim follows.
\end{proof}

By \cref{prop: lower bound for d_r} and \cref{def: rapid decay}, we have proved that there exists a saturated sublattice $\Lambda_r \subset\cL$ of rank at least $a_k + a_1 + 2$ that decays rapidly to $r$th order for any given $r$. As whether a saturated lattice $\Lambda_r$ decays rapidly or not (to any fixed order) depends only on its mod $p$ reduction, and as there are only finitely many sublattices of any rank of a finite-dimensional $\bF_p$ vector space, we may pass to a subsequence $r_j$ of the set of all positive integers such that $\Lambda_{r_j}/ p\cL \subseteq \cL / p\cL$ doesn't depend on $r_j$. Then by \cref{lem: order s < r} and \cref{prop: prelim decay results}(a), the sublattice $\Lambda_{r_j}$ for any value of $r_j$ decays rapidly and the theorem is proved. 
\end{proof}

\section{First order decay and Eisentein series: superspecial points} 
\label{sec: first step superspecial}
In this section we prove the following main result on the Eisenstein series, which is a major middle step towards Theorem~\ref{thm: main local version}:
\begin{theorem}\label{lm:localglobalEisensteinSSP}
 Let $C$ be a curve as in \S\ref{sub:setuplocalcurve}, and let $P$ be a superspecial point. If $\mathcal{A}_{C^{/P}}$ does not admit a saturated lattice of formal special endomorphisms of rank $2(m-a)$, then there exists a constant $\alpha< 1$, depending only on $C$ and $P$, such that for $M\in \bN$ coprime to $p$, we have  $$S'(M)=\sum_{r=1}^\infty \frac{{q_{L_{r,P}'}(M)}}{g_P(M)}<\alpha .$$
\end{theorem}
The theorem will be proved at the end of \S\ref{sub:easysitu}. There are two main inputs: the first input is the decay lemma established in   \S\ref{sec: decay superspecial}. The second is a thorough understanding of the first order decay, which is the main contributor to the local Eisenstein series. The study of the first order decay is carried out in \S\ref{sub:lineconfig}$\sim$\S\ref{sub:patternfirstorderdecay}, culminating in Proposition~\ref{prop:thincase!!!} and \ref{prop:L+thin}, which state that a decaying sequence of lattices must follow a strict pattern. 

\begin{remark}
    In this section, $m$ will be as in Sections \ref{sec: local structure at superspecial point} and \ref{sec: decay superspecial}. The symbol $M$ will play the role that $m$ plays in Section \ref{sec:arithmetic}. 
\end{remark}
\subsection{An example}\label{subsub:toyexamplem=2}Let's assume $m=2$. We will use the setup and notation in \S\ref{sub:setuplocalcurve}. Then the only case we care is that $C$ lies generically on the almost supersingular stratum. Suppose that the local coordinates for $C^{/P}$ is given by series  $x_1,y_1,x_2,y_2\in \bF[\![t]\!]$. Then 
\begin{align}
    \label{eq:m=2q1}
  &  Q_0 = x_1y_1+x_2y_2= 0, \\
      \label{eq:m=2q2}
   &Q_1= x_1^py_1+ x_1y_1^p+x_2^py_2+ x_2y_2^p\neq 0.
\end{align}
 In particular, we must have $a=a_k=1$ in this case. Let $c_i=v_t(x_i)$ and $d_i=v_t(y_i)$. Note that (\ref{eq:m=2q1}) implies that $c_1+d_1=c_2+d_2$. We also suppose that $\alpha_i,\beta_i$ are the leading coefficients of the $x_i,y_i$ respectively. Note that replacing the vectors $e_i,f_i \in \cL_1$ by an $\textrm{SO}(\bZ_p)$ coordinates where each matrix entry is the teichmüller-lift of its mod $p$ reduction induces changing the variables $(x_i,y_i)$ by an element of $\textrm{SO}(\bF_p)$. Note that doing this leaves the equations of the height strata unchanged. We shall refer to transformations of this form as orthogonal change of coordinates. Clearly, there are only finitely orthogonal change of coordinates. Note that swapping $x_i$ with $y_i$ is such change of coordinates. 
\begin{lemma}\label{lem:exampleassm=2}
    Consider the set $\Sigma = \{\alpha_1 t^{c_1}, \alpha_2 t^{c_2}, \beta_1 t^{d_1}, \beta_2 t^{d_2}\}$. After an orthogonal change of coordinates, we may always assume that $c_1\leq c_2\leq d_2\leq d_1$. Further, in addition to this holding, we belong to one of the following cases:
    \begin{enumerate}
        \item The set $\Sigma$ is linearly independent over $\bF_p$.

        \item There is exactly one $\bF_p$ linear relation in $\Sigma$, and which is of the form $\frac{\alpha_2 t^{c_2}}{\beta_2t^{d_2}}  = \gamma_2 \in \bF_p^{*} $.

        \item There are exactly two $\bF_p$ linear relations in $\Sigma$, and these have the form $\frac{\alpha_i t^{c_i}}{\beta_it^{d_i}} = \gamma_i \in \bF_p^{*}$. Further, the quadratic form $\gamma_1 X^2 + \gamma_2 Y^2$ is anisotropic over $\bF_p$.
    \end{enumerate}
    
\end{lemma}

The proof consists of a fairly straightforward process of iteratively applying orthogonal change of coordinates. We prove (with full detail) a far more general statement in a far more general setting later in this section (Lemmas \ref{lm:hypermaxiaml} and \ref{lm:quadmaximal}).

We will explain how to deduce Theorem \ref{lm:localglobalEisensteinSSP} for $m=2$. This amounts to upper-bounding the series $S'(M)$. Following Remark~\ref{rmk:explicit q_L'}, we can explicate each term of $S'(M)$ as 
$$\frac{q_{L_{P,r}'}(M)}{g_P(M)}=\frac{p^{2}-1}{h_P}\frac{1}{p(1+p^{-3})}\frac{\delta(p,L'_{P,r},M)}{[L'_{P,1}\otimes\bZ_p:L'_{P,r}\otimes\bZ_p]}. $$
Note that in the setting of Lemma \ref{lem:exampleassm=2}, we have that \begin{equation}\label{eq:estimateh_Pm=2}
    h_P = v_t(Q_1) \geq pc_1 + d_1,  
\end{equation}
so the term $\frac{p^2-1}{h_P}$ can be bonded from above accordingly. Therefore, in theory we can bound $S'(M)$ by bounding the local density $\delta(p,L'_{P,r},M)$ and covolume ${[L'_{P,1}\otimes\bZ_p:L'_{P,r}\otimes\bZ_p]}$. However, in practice, the shape of $L'_{P,r}$ is somewhat evasive. Instead of directly bounding the local density and covolume for $L'_{P,r}$, we will construct a sequence of auxiliary lattices $L''_{P,r} $ which we can describe exactly. For each lattice $L''_{P,r}$ and each integer $M$ not a multiple of $p$, we can also precisely compute $\delta(p,L''_{P,r},M)$ and $[L''_{P,1}\otimes\bZ_p:L''_{P,r}\otimes\bZ_p]$. This allows us to explicitly upper-bound the auxiliary series $$S''(M):=\sum_{r=1}^{\infty}\frac{q_{L_{P,r}''}(M)}{g_P(M)}=\frac{p^{2}-1}{h_P}\frac{1}{p(1+p^{-3})}\sum_{r=1}^\infty\frac{\delta(p,L''_{P,r},M)}{[L''_{P,1}\otimes\bZ_p:L''_{P,r}\otimes\bZ_p]}.$$ We will then show that $S'(M)\leq S''(M)$. 
Upon computing $S''(M)$ explicitly, we obtain that $S'(M)<1$ for all $M$ coprime to $p$ \emph{unless} $\cA_{C^{/P}}$ admits a rank-2 saturated lattice of formal special endomorphisms. Since $\delta(p,L'_{P,r},M)$ depends only on $\left(\frac{M}{p}\right)$  and ${[L'_{P,1}\otimes\bZ_p:L'_{P,r}\otimes\bZ_p]}$ is independent of $M$, there exists a constant $\alpha<1$ independent of $M$, such that $S'(M)<\alpha$.  This yields the theorem. 

We will now carry out this procedure in case 3 of Lemma \ref{lem:exampleassm=2} (the hardest case). In other words, we will describe $L''_{P,r}$ in case 3 of Lemma \ref{lem:exampleassm=2}, and will explicitly compute $S''(M)$. Note that in case 3 we must have $c_1=c_2=d_2=d_1$, and we will use $d$ to denote this number. In the following, we construct $L''_{P,r}$. First, let $\mathcal{L}_{\mathrm{brd}}:=\mathrm{Span}_{\bZ_p}\{e_1-\gamma_1f_1, e_2-\gamma_2f_2\}$, where by abuse of notation $\gamma_i\in \bZ_p^*$ denotes the Teichmüller lift of $\gamma_i$ in Lemma~\ref{lem:exampleassm=2}(3). Second, for each $j\geq 1$, define $H_{j}:=h_P(1+p^2+...+p^{2(j-1)})$ and $H_{0}:=0$. Finally, define $ L''_{P,r}$ to be the $\bZ$-lattice such that $L''_{P,r}\otimes \bZ_l=L'_{P,r}\otimes \bZ_l$ for $l\neq p$, and 
\begin{equation*}
 L''_{P,r}\otimes\bZ_p= \left\{\begin{aligned}
    & \mathcal{L}_{\mathrm{brd}}+p^j\cL ,\,\,\,\, &r\in (H_{j},H_{j}+p^{2j}d],\\
     & \mathcal{L}_{\mathrm{brd}}+ p^{j}\mathcal{L}_0 +p^{j+1}\mathcal{L}_1,\,\,\,\, &r\in (H_{j}+p^{2j}d,H_{j+1}].
 \end{aligned}\right.    
\end{equation*}
The lattice $L''_{P,r}$ can be thought of as $L'_{P,r}$ in the special case where all vectors in $\cL_{\mathrm{brd}}$ never decay. For $M$ coprime to $p$, one computes that 
$$\delta(p,L''_{P,r},M)=\left\{\begin{aligned}
    & 1-p^{-2},\,\,\,\, &r\in (0,a],\\
     & 1+p^{-1},\,\,\,\, &r\in (a,\infty).
 \end{aligned}\right.    $$
 $$[L''_{P,1}\otimes\bZ_p:L''_{P,r}\otimes\bZ_p]=\left\{\begin{aligned}
    & p^{4j},\,\,\,\, &r\in (H_{j},H_{j}+p^{2j}d],\\
     & p^{4j+2},\,\,\,\, &r\in (H_{j}+p^{2j}d,H_{j+1}].
 \end{aligned}\right.   $$
It follows that \begin{align*}
     S''(M)&\leq  \frac{p^2-1}{d(p+1)}\frac{1}{p(1+ p^{-3})}\left[d(1-p^{-2})+ (1+p^{-1})\sum_{j=1}^\infty \frac{dp^{2j}}{p^{4j}} +(1+p^{-1})\sum_{j=0}^\infty\frac{(h_P-d)p^{2j}}{p^{4j+2}} \right]\\
   & \leq \frac{p^2-1}{d(p+1)}\frac{d}{p(1+ p^{-3})}\left[(1-p^{-2})+\frac{(1+p)^2}{p(p^2-1)}\right] =1.
\end{align*} 
Each inequality in the above formula comes from (\ref{eq:estimateh_Pm=2}).
\begin{lemma}\label{lm:isometric embedding in the m=2 case}
    Notation as above.  For all $M$ coprime to $p$ we have $S'(M)\leq S''(M)$. If $\cA_{C^{/P}}$ does not admit a rank-2 saturated lattice of formal special endomorphisms, then $S'(M)< S''(M)$. 
\end{lemma}\begin{proof}
For $\Lambda\subseteq\cL$ a sublattice, we let $\overline{\Lambda}$ denote its mod $p\cL$ reduction. We immediately see that for $d<r\leq h_P$, we have $\overline{\cL_0}\subseteq \overline{L'_{P,r}\otimes\bZ_p}\subseteq\overline{\cL_{\mathrm{brd}}}\oplus\overline{\cL_0}=\overline{L''_{P,r}\otimes\bZ_p}$. For $d>h_P$, every primitive vector in the lattice $\cL_0$ decays. It follows that $\overline{L'_{P,r}\otimes\bZ_p}$ is a sublattice of $\overline{\cL_{\mathrm{brd}}}\oplus\overline{\cL_0}$ whose projection to $\overline{\cL_{\mathrm{brd}}}$ is an embedding, which is moreover an isometrical embedding since $\overline{\cL_0}$ carries the trivial quadratic form. In the following, we show that for all $r\geq 1$,
\begin{equation}\label{eq:localdensitycomparison}
\frac{\delta(p,L'_{P,r},M)}{[L'_{P,1}\otimes\bZ_p:L'_{P,r}\otimes\bZ_p]}\leq \frac{\delta(p,L''_{P,r},M)}{[L''_{P,1}\otimes\bZ_p:L''_{P,r}\otimes\bZ_p]}.   
\end{equation}Since $L''_{P,r}=L'_{P,r}$ when $r\leq d$, it suffices to consider the case where $r>d$. By the Decay Lemma \ref{thm: superspecial rapid decay} (also see Remark \ref{rem: ak is a}), the lattice $\cL_{\mathrm{rv}}:=\mathrm{Span}_{\bZ_p}\{e',f',f_1,f_2\}$ decays rapidly to arbitrary order. So \begin{equation}\label{eq:Lgeq L''}
[L'_{P,1}\otimes\bZ_p:L'_{P,r}\otimes\bZ_p]\geq [L''_{P,1}\otimes\bZ_p:L''_{P,r}\otimes\bZ_p].\end{equation} Suppose that (\ref{eq:Lgeq L''}) is a strict inequality, then (\ref{eq:localdensitycomparison}) holds and is a strict inequality. Suppose that (\ref{eq:Lgeq L''}) is an equality, then $\overline{L'_{P,r}\otimes\bZ_p}$ and $\overline{L''_{P,r}\otimes\bZ_p}$ need to have the same dimension. If $d<r<h_P$, this means
$\overline{L'_{P,r}\otimes\bZ_p}=\overline{\cL_{\mathrm{brd}}}\oplus\overline{\cL_0}$; if $r>h_P$, then the projection of $\overline{L'_{P,r}\otimes\bZ_p}$ to  $\overline{\cL_{\mathrm{brd}}}$ is an isomorphism. It follows that $\delta(p,L'_{P,r},M)=\delta(p,L''_{P,r},M)$. So (\ref{eq:localdensitycomparison}) holds and is an equality. Therefore (\ref{eq:localdensitycomparison}) holds for all $r\geq 1$, and $S'(M)\leq S''(M)$. 

To show $S'(M)<S''(M)$ under the extra condition, we proceed by contradiction. If $S'(M)=S''(M)$, then (\ref{eq:localdensitycomparison}) need to be an equality for all $r\geq 1$, hence  (\ref{eq:Lgeq L''}) must be an equality for all $r\geq 1$. In the following let $r>h_P$. The projection of $\overline{L'_{P,r}\otimes\bZ_p}$ to  $\overline{\cL_{\mathrm{brd}}}$ is an isomorphism. In particular, $\overline{L'_{P,r}\otimes\bZ_p}$ does not depend on $r$. Pick $\alpha_{1},\beta_{1},\alpha_{2},\beta_{2} \in \bZ_p$, such that $E_1:=e_1-\gamma_1f_1+\alpha_{1} e'+\beta_{1}f'$ and $E_2:=e_2-\gamma_2f_2+\alpha_{2} e'+\beta_{2}f'$ reduce mod $p$ to a basis of $\overline{L'_{P,r}\otimes\bZ_p}$. Then for each $r$ there exit $R_{r,1}, R_{r,2} \in p\cL$ with the property that $E_1+R_{r,1}, E_2+R_{r,2} \in {L'_{P,r}\otimes\bZ_p}$. We can modify $R_{r,i}$ by $E_i$ as follows: 
there exits $u_{r,i}\in p\bZ_p$ such that $R'_{r,i}:=R_{r,i}-u_{r,i}E_i\in p\cL_{\mathrm{rv}}$. It follows that $E_i+(1+u_{r_i})^{-1}R'_{r,i}\in {L'_{P,r}\otimes\bZ_p}$. Therefore, we can assume at the beginning that $R_{r,i}\in p\cL_{\mathrm{rv}}$.  
Since $\cL_{\mathrm{rv}}$ decays rapidly to arbitrary order, the sequence $\{E_i+R_{r,i}\}_{r=h_P+1}^\infty$ converges to a vector that is congruent to $E_i$ mod $p$ and lies in every $L_{P,r}'\otimes\bZ_p$. As a result, $\cA_{C^{/P}}$ admits a rank-2 saturated lattice of formal special endomorphisms.
\end{proof}

Since $\cA_{C^{/P}}$ does not admit a rank-2 saturated lattice of formal special endomorphisms, follows from the lemma that $S'(M) < S'(M)\leq 1$. We have seen that this implies Theorem \ref{lm:localglobalEisensteinSSP}. 


\subsection{$Q$-degeneracy and line configurations}\label{sub:lineconfig}
We introduce the notions of $Q$-degeneracy and line configuration that effectively captures the combinatorics of the Frobenius-twisted quadratic forms $\{Q_i(t)\}_{i\geq 0}$ associated to a collection of power series $\{\underline{x}(t),\underline{y}(t)\}$ that arise in the context of \S\ref{sub:setuplocalcurve}. This can be thought of as an analogue of the Newton polygon. For maximal flexibility, we will build our tools separately from the setup in \S\ref{sub:setuplocalcurve}. 

The following notion will be extremely important through out the section: 
\begin{definition}
  Let $s_1,s_2,...,s_n\in \mathbb{F}[\![t]\!]$. We say that $s_1,s_2,...,s_n$ are \textit{$\mathbb{F}_p$-hyper-independent} if for any $\alpha_1,\alpha_2,...,\alpha_n\in \mathbb{F}_p$, we have $v_t(\sum_{i=1}^n\alpha_i s_i)= \min_{1\leq i\leq n}\{v_t(\alpha_is_i)\}$. Equivalently, $s_1,s_2,...,s_n$ are $\mathbb{F}_p$-hyper-independent if for any $d\in \mathbb{N}$, the series in the subset $\{s_i: v_t(s_i)=d\}$ have $\mathbb{F}_p$-linear independent leading coefficients.  
\end{definition}
\begin{example}
    If $s_1,...,s_n$ are $\mathbb{F}_p$-hyper-independent, then so is $s_1,...,s_n, 0$. 
\end{example}

\subsubsection{Basic definitions}
\begin{definition}\label{def:Q-degen}
Let $\{x_1,x_2,...,x_n,y_1,y_2,...,y_n\}$ be a  collection of power series in $\mathbb{F}[\![t]\!]$. For $r\in \mathbb{Z}_{\geq 0}$, the $r$th Frobenius-twisted quadratic form $Q_r$ is defined by:  $$Q_r(\underline{x}, \underline{y})= \sum_{i=1}^n (x_i^{p^r}y_i+ x_iy_i^{p^r}).$$
If $s$ is a non-negative integer, we say that the collection $\{\underline{x}, \underline{y}\}$ is \textit{$Q_s$-degenerate}, if for all $0\leq r\leq s$, we have 
    \begin{equation}\label{eq:degenerate}
    v_t\left(Q_r(\underline{x}, \underline{y})\right)> \min_{1\leq i\leq n}\{v_t(x_i^{p^r}y_i), \;v_t(x_iy_i^{p^r})\}.
    \end{equation}
If both sides of (\ref{eq:degenerate}) are $\infty$, the inequality holds by default.
\end{definition}

\begin{definition}\label{def:lineconfig}
 The set $\mathbf{E}$ of \textit{admissible lines} (in short, \textit{lines}) is the set of symbols of the form ``$y=ax+b$'' where $a,b\in \bR_{\geq 0}\cup\{\infty\}$. When $a,b\in \bR_{\geq 0}$, we say the line is \textit{geometric}: it is a genuine line in $\mathbb{R}^2$. When $a$ or $b$ equals $\infty$, we say the line is \textit{degenerate}. Let $n$ be a non-negative integer. A \textit{line configuration} of $n$ lines is an element $\{\Pi_1,...,\Pi_n\}\in\mathbf{E}^n$. The following are some notions related to line configurations: 
\begin{enumerate}
\item   Let $\{x_1,x_2,...,x_n,y_1,y_2,...,y_n\}$ be a collection of power series in $\mathbb{F}[\![t]\!]$. To each $1\leq i\leq n$, we attach a line $\Pi_i\subseteq \mathbb{R}^2$, with equation $$\Pi_i:y=\min\{v_t(x_i),v_t(y_i)\} x  + \max\{v_t(x_i),v_t(y_i)\} .$$ 
This gives rise to a line configuration $\{\Pi_i\}_{1\leq i\leq n}$ as per Definition~\ref{def:lineconfig}, which will be called the \textit{line configuration associated to $\{\underline{x},\underline{y}\}$}. In this paper,
we will mainly be concerned with line configurations that arise this way. 
\item\label{def:multiset} When there is no need to keep track of indices, we will write a line configuration as a multiset: \begin{equation}
    \{\Pi_i\}_{1\leq i\leq n}= \{m(\ell_j)\cdot\ell_j\}_{1\leq j\leq w},
\end{equation}
where $m(\ell_j)$ is the multiplicity of the line $\ell_j$. \item  Given a line configuration $\{\Pi_{i}:y=a_ix+b_i\}_{{1\leq i\leq n}}$, let $\Omega_i=\{(x,y)\in \bR^2| y\leq a_ix+b_i\}$ (if the line is degenerate, then $\Omega_i$ is defined to be $\mathbb{R}^2$). The region $$\Omega:= \bigcap_{i=1}^n\Omega_i$$ 
will be called the \textit{$Q$-polygon} attached to the line configuration. The \textit{boundary} $\partial\Omega$ makes sense (for example, if all lines are degenerate, then $\Omega=\mathbb{R}^2$ and $\partial\Omega=\varnothing$). 
    \item\label{def:critpoints} A point on $\partial\Omega$ with $x$-coordinates $p^r, r\in \mathbb{Z}_{\geq  0}$ will be called a \textit{critical point} of $\Omega$. A critical point that lies on two lines of different slopes will be called \textit{vertex-like}, otherwise it will be called \textit{edge-like}. By our definition, a degenerate line can not contain any critical point. 
\item\label{def:redundant} A line in $\{\Pi_i\}_{1\leq i\leq n}$ is said to be \textit{redundant} over an interval $I\subseteq [1,\infty)$, if it does not contain any critical point in the region $I\times \mathbb{R}$. It is called \textit{totally redundant}, if it is redundant over $[1,\infty)$. For example, degenerate lines are always totally redundant. 
\item\label{def:quasi-redundant} A line in $\{\Pi_i\}_{1\leq i\leq n}$ is said to be \textit{quasi-redundant} over an interval $I\subseteq [1,\infty)$, if it is redundant, or its intersection with $\Omega\cap (I\times \bR)$ consists of only one point. For example, a quasi-redundant line which is not redundant over $I$ intersects $\Omega$ at exactly one point, which is necessarily a vertex-like critical point.  \end{enumerate}
\end{definition}

\begin{definition}\label{def:lineconfig2}Notation as in Definition~\ref{def:lineconfig}.
    \begin{enumerate}    
\item Let $\Pi_1:y=a_1x+b_1$, $\Pi_2:y=a_2x+b_2$ be two lines in $\mathbf{E}$, we say that $\Pi_1=\Pi_2$, if $a_1=a_2,b_1=b_2$. We say that $\Pi_1$ \textit{lies above} $\Pi_2$, denoted by $\Pi_1\succeq \Pi_2$, if one of the following holds: \begin{enumerate}
    \item $\Pi_1,\Pi_2$ are geometric, and $a_1\geq a_2$, $a_1+b_1\geq a_2+b_2$. Geometrically, this means $\Pi_1$ lies above $\Pi_2$ in the interval $[1,\infty)$.
    \item $\Pi_1$ is degenerate and $\Pi_2$ is geometric. 
    \item $a_1=a_2=\infty$ and $b_1\geq b_2$.
    \item $b_1=b_2=\infty$ and $a_1\geq a_2$.
\end{enumerate}
This defines a partial order on $\mathbf{E}$. 
\item\label{def:linconfig it 7} Two line configurations $\{\Pi_i\}_{1\leq i\leq n}$ and $\{\Pi_i'\}_{1\leq i\leq n}$ are said to be \textit{equal}, if $\Pi_i=\Pi_i'$ for every $i$. A line configuration $\{\Pi_i\}_{1\leq i\leq n}$ is said to \textit{lie above} another line configuration $\{\Pi_i'\}_{1\leq i\leq n}$, if for every $i$, $\Pi_i\succeq \Pi_i'$. We denote this by $\{\Pi_i\}_{1\leq i\leq n}\succeq \{\Pi_i'\}_{1\leq i\leq n}$. This is a partial order on the set of line configurations of $n$ lines. 
\end{enumerate}
\end{definition}

\begin{lemma}\label{lm:lineconfig}Suppose that $\{x_1,x_2,...,x_n,y_1,y_2,...,y_n\}$ is a collection of power series in $\mathbb{F}[\![t]\!]$ and let $\{\Pi_i\}$ be the associated {line configuration}.
If $\{\underline{x},\underline{y}\}$ is $Q_s$-degenerate, then any critical point 
       of the $Q$-{polygon} within $[1,p^s]\times \mathbb{R}$ must lie on either at least two lines (counting multiplicity), or a line of form $y=dx+d$. 
\end{lemma}
\begin{proof}
If (\ref{eq:degenerate}) holds, then at least two elements in $\{x_i^{p^r}y_i, x_iy_i^{p^r}\}_{i=1}^n$ achieve the smallest valuation.
\end{proof}
\begin{remark}[Redundant line principle]
    Let $\{x_1,x_2,...,x_n,y_1,y_2,...,y_n\}$ be a  $Q_s$-degenerate collection of power series in $\mathbb{F}[\![t]\!]$. Let $\{\Pi_i\}$ be the associated line configuration. If $\{\Pi_{i}\}_{i\in \Delta\subseteq\{1,2,...,n\}}$ are redundant over $[1,p^s]$, then the collection $\{x_i,y_i\}_{i\in \{1,2,...,n\}\setminus \Delta}$ is again $Q_s$-degenerate.
\end{remark}
\subsubsection{$Q$-isometry and maximal elements} 
\begin{definition}\label{def:isometricshuffle}
Let $\{x_1,...,x_n,y_1,...,y_n\}$ be a collection of power series in $\mathbb{F}[\![t]\!]$. The \textit{$Q$-isometry class} of $\{\underline{x},\underline{y}\}$ is the set of all collections $\{x_1',...,x_n',y_1',...,y_n'\}$ such that $$(\underline{x}',\underline{y}')=(\underline{x},\underline{y})\cdot R,$$
where $R\in \mathrm{SO}(Q_0)(\mathbb{F}_p)$, and $Q_0= \begin{bmatrix}
    O & I_n\\
    I_n & O
\end{bmatrix}$.  For example, the following three operations produce elements in the $Q$-isometry class of $\{\underline{x},\underline{y}\}$:  
\begin{enumerate}
    \item re-indexing: $(x_1',...,x_n', y_1',...,y_n')=(x_{\sigma(1)},...,x_{\sigma(n)}, y_{\sigma(1)},...,y_{\sigma(n)})$ for some $\sigma\in S_n$.
    \item swapping: $x'_i,y'_i \leftrightarrow y_i,x_i$.
    \item quantum re-indexing: $(x_1',...,x_n', y_1',...,y_n')=((x_1,...,x_n)B,(y_1,...,y_n)B^{-1,T})$, where $B\in \mathrm{GL}_n(\mathbb{F}_p)$.  
\end{enumerate}    
It is clear that the $Q$-isometry class of a given collection is a finite set.
\end{definition}

The $Q$-isometric class can be best understood in the following context: Fix $n$. We can form a \textit{frame space} $$V_{\mathrm{frame}}=\mathrm{Span}_{\mathbb{F}_p}\{X_1,...,X_n,Y_1,...,Y_n\},$$ where $X_i,Y_i$ are merely symbols. This is a quadratic $\mathbb{F}_p$-space with pairing $$Q\left(\sum_{i=1}^na_iX_i+b_iY_i\right)= \sum_{i=1}^n a_ib_i.$$
We can then say that a collection $\{\underline{x},\underline{y}\}$ of power series  in $\mathbb{F}[\![t]\!]$ is an  \textit{$\mathbb{F}[\![t]\!]$-realization} of $\{\underline{X},\underline{Y}\}$, written as $\{\underline{X},\underline{Y}\}\rightarrowtail \{\underline{x},\underline{y}\}$. Then any element in the $Q$-isometry class of $\{\underline{x},\underline{y}\}$ is induced from an isometry of $V_{\mathrm{frame}}$: 
$$\{\underline{X},\underline{Y}\} \xrightarrow{R} \{\underline{X},\underline{Y}\}\rightarrowtail \{\underline{x},\underline{y}\}.$$
 \begin{definition}
 Let $\{x_1,...,x_n,y_1,...,y_n\}$ be a collection of power series in $\mathbb{F}[\![t]\!]$ and let $\{\Pi_i\}_{1\leq i\leq n}$ be the associated line configuration. The \textit{$Q$-isometry class} of $\{\Pi_i\}_{1\leq i\leq n}$ is the finite set of line configurations associated to the elements in the $Q$-isometry class of $\{\underline{x},\underline{y}\}$. A \textit{maximal configuration} in the $Q$-isometry class of $\{\Pi_i\}_{1\leq i\leq n}$ is a maximal element in the partial order $\succeq$, see Definition~\ref{def:lineconfig2}(\ref{def:linconfig it 7}). A collection $\{\underline{x},\underline{y}\}$ is said to be \textit{maximal} in its $Q$-isometry class, if the associated line configuration is maximal in its $Q$-isometry class.
 \end{definition}
\begin{remark}\label{rmk:maximal_sub}
  The following observations are clear: 
\begin{enumerate}
    \item If $\{x_1,...,x_n,y_1,...,y_n\}$ be a collection of power series in $\mathbb{F}[\![t]\!]$, maximal in its $Q$-isometry class. Then for any $\Sigma \subseteq \{1,2,...,n\}$, the subcollection $\{x_i,y_i\}_{i\in \Sigma}$ is also maximal in its own $Q$-isometry class. 
    \item The property of being maximal is preserved under re-indexing and swapping, but not under a general quantum re-indexing.
     \item The multiset of lines in the line configuration is also preserved under re-indexing and swapping, but not under a general quantum re-indexing.
\end{enumerate}
\end{remark}
  Maximal collections have nice properties. The following lemmas will be extremely useful:
\begin{lemma}\label{lm:hypermaxiaml}
 Suppose that $\{x_1,...,x_n,y_1,...,y_n\}$ is a collection of power series in $\mathbb{F}[\![t]\!]$, which is maximal in its $Q$-isometry class. Then $x_1,...,x_n$ are $\mathbb{F}_p$-hyper-independent. The same holds for $y_1,...,y_n$.
\end{lemma}
\begin{proof} Let $v\in \mathbb{N}$ and $D_v=\{i\in\{1,2,...,n\}: v_t(x_i)=v\}$. Let $\alpha_i, i\in D_v$ be the leading coefficient of $x_i$. It suffices to show that  
    $\{\alpha_i\}_{i\in D_v}$ are $\mathbb{F}_p$-independent. Suppose this is not the case. Up to re-indexing, we can take $D_v=\{1,2,...,l\}$ and there is an $r\leq l$ and $c_1,c_2,...,c_r\in \mathbb{F}_p^*$, such that $c_1\alpha_1+ ...+c_r\alpha_r=0$. Possibly re-indexing again, let  $v_t(y_1)=\max_{1\leq i\leq r}\{v_t(y_i)\}$. Without loss of generality, let $c_1=1$. We can do a quantum re-indexing (Definition \ref{def:isometricshuffle}) by $x_1'=c_1x_1+ ...+c_rx_r$, and $x_i'=x_i$ if $i\geq 2$. The effect of this quantum re-indexing on $y_i$ is  $y_1'=y_1$, $y_i'=y_i-c_iy_1$ when $2\leq i\leq r$ and $y_i'=y_i$ when $i>r$. Now $v_t(x_1')> v_t(x_1)$ and $v_t(y_1')= v_t(y_1)$, and when $2\leq i\leq r$ we have $v_t(y_i')\geq v_t(y_i)$. As a result, while the new collection $\{\underline{x}',\underline{y}'\}$ lies in the same $Q$-isometry class, its line configuration lies strictly above the original one, contradicting the maximality of $\{\underline{x},\underline{y}\}$. 
\end{proof}

\begin{lemma}\label{lm:quadmaximal}
  Suppose that $\{x_1,...,x_n,y_1,...,y_n\}$ is a collection of power series in $\mathbb{F}[\![t]\!]$, which is maximal in its $Q$-isometry class. Let $d\in \mathbb{N}$ and let $C_d=\{i\in \{1,2,...,n\}:v_t(x_i)=v_t(y_i)=d\}$. Then $$V^+_d:=\left\{\sum_{i\in C_d}c_iX_i+d_iY_i: v_t\left(\sum_{i\in C_d}c_ix_i+d_iy_i\right)>d\right\}\subseteq V_{\mathrm{frame}}$$
   has dimension at most 2. In addition, as a quadratic $\mathbb{F}_p$-space, $V^+_d$ is nondegenerate of inert type when $\dim V^+_d =2$, and is nondegenerate when $\dim  V^+_d=1$.
\end{lemma}
\begin{proof}
Let $V_d=\left\{\sum_{i\in C_d}c_iX_i+d_iY_i\right\}\subseteq V_{\mathrm{frame}}$. Up to re-indexing, we can assume that $C_d=\{1,2,...,r\}$. We first show that the radical of $V^+_d$ is $\mathbf{0}$. Suppose this is not the case, then up to an isometry of $V_d$, we have $X_1\in V^+_d, Y_1\notin V_d^+$. This means that there is a collection $\{\underline{x}',\underline{y}'\}$ in the $Q$-isometric class of $\{\underline{x},\underline{y}\}$, 
whose line configuration lies strictly above\footnote{We note that any isometry applied to the set $\{x_1\hdots x_r, y_1\hdots y_r\}$ must necessarily produce lines $\succeq$ the original lines, as we have the conditions $v_t(x_i) = v_t(y_i) = d$.} that of $\{\underline{x},\underline{y}\}$. A contradiction to the maximality. 

Now we show that $V^+_d$ can not contain a hyperbolic plane (i.e., a nondegenerate dimension 2 subspace of split type). Suppose this is not the case, then up to an isometry of $V_d$, we have $X_1,Y_1\in V_d^+$. This again means that there is a collection $\{\underline{x}',\underline{y}'\}$ in the $Q$-isometric class of $\{\underline{x},\underline{y}\}$, 
whose line configuration lies strictly above that of $\{\underline{x},\underline{y}\}$. A contradiction to the maximality. 

Since $V_d^+$ has trivial radical, and has no hyperbolic subspace, it can only happen that $V_d^+$ is inert of dimension 2, nondegenerate of dimension 1, or trivial.   
\end{proof}
\begin{lemma}\label{lm:quadmaximal333}
  Suppose that $\{x_1,...,x_n,y_1,...,y_n\}$ is a collection of power series in $\mathbb{F}[\![t]\!]$, which is maximal in its $Q$-isometry class. Suppose that $d=v_t(x_1)=...=v_t(x_n)=v_t(y_1)\leq ...\leq v_t(y_n)$ and $d<v_t(y_n)$. Then for all $a_i,b_i\in \mathbb{F}_p$, the $t$-valuation of $x_n + \sum_{1\leq i\leq n-1} (a_ix_i+ b_iy_i)$ equals $d$.
\end{lemma}
\begin{proof}
Suppose that the $t$-valuation the expression is larger than $d$. Let  
$x'_n= x_n-(\sum_{1\leq i\leq n-1} a_ib_i)y_n+\sum_{1\leq i\leq n-1} (a_ix_i+ b_iy_i)$, $y'_n=y_n$. For $1\leq i<n$, let $x'_i= x_i-b_iy_n$, $y'_i= y_i-a_iy_n$.  Then the new collection $\{\underline{x}',\underline{y}'\}$ is in the same $Q$-isometry class but with strictly larger line configuration. Contradiction. 
\end{proof}
\subsection{Critical points and shape of the $Q$-polygon}
In the following, we understand the behavior of critical points (Definition~\ref{def:lineconfig}(\ref{def:critpoints})) of a line configuration.  The main result is Corollary~\ref{cor:maximalanddegenerate}.
\subsubsection{Dependence results} 
If $s_1,...,s_n\in \mathbb{F}$,  the Moore matrix $(s_j^{(i-1)})_{1\leq i,j\leq n}$ will be denoted by $M(s_1,...,s_n)$. 
\begin{lemma}\label{lm:some_shit}
    Let $\alpha_1,...,\alpha_n,\beta_1,...,\beta_n\in{\mathbb{F}}^*$, $N\in \mathbb{N}$, and consider the quantity $\gamma_r=\sum_{i=1}^n\alpha_i^{p^{N+r}}\beta_i$.
    Suppose that $\gamma_r=0$ for all $0\leq r\leq n-2$.
    \begin{enumerate}[label=\upshape{(\alph*)}]
        \item If $\gamma_r=0$ for $r=n-1$, then $\dim\mathrm{Span}_{\mathbb{F}_p}\{\alpha_1,...,\alpha_n\}<n$. 
        \item If $\gamma_r\neq 0$ for $r=n-1$ , then $\dim\mathrm{Span}_{\mathbb{F}_p}\{\alpha_1,...,\alpha_n\}=\dim\mathrm{Span}_{\mathbb{F}_p}\{\beta_1,...,\beta_n\}=n$.
    \end{enumerate}
\end{lemma}
 \begin{proof} Note that it suffices to establish the case where $N=0$. The general case reduces to $N=0$ case by setting $\alpha_i'=\alpha^{p^N}_i$, and noting that $\alpha_1,...,\alpha_n$ are $\mathbb{F}_p$-linearly dependent if and only if $\alpha'_1,...,\alpha'_n$ are $\mathbb{F}_p$-linearly dependent. 
 \begin{enumerate}[label=\upshape{(\alph*)}]
     \item The condition is equivalent to $M(\alpha_1,...,\alpha_n)\cdot [\beta_1,...,\beta_n]^T=0$. If $\alpha_1,...,\alpha_n$ are $\mathbb{F}_p$-linearly independent, then $\det M(\alpha_1,...,\alpha_n)\neq 0$. This implies that $\beta_i=0$ for all $i$. A contradiction. 
     \item The matrix $P=M(\alpha_1,...,\alpha_n)^{(n-1)}\cdot M(\beta_1,...,\beta_n)^T$ is a lower triangular matrix with diagonal entries $\gamma_{n-1},\gamma_{n-1}^p,...,\gamma_{n-1}^{p^{n-1}}$. In particular,  $P$ is invertible. This shows that $M(\alpha_1,...,\alpha_n)$ and $M(\beta_1,...,\beta_n)$ are both invertible.
 \end{enumerate}
 \end{proof}

\begin{lemma}\label{lm:leadingcoefogus2}
Let $\alpha_1,...,\alpha_n,\beta_1,...,\beta_n\in{\mathbb{F}}^*$ and consider the quantity $q_r=\sum_{i=1}^n(\alpha_i^{p^{r}}\beta_i+\alpha_i\beta_i^{p^{r}})$. Suppose that $q_r=0$ for all $0\leq r\leq n-1$.
\begin{enumerate}[label=\upshape{(\alph*)}]
    \item  $\dim\mathrm{Span}_{\mathbb{F}_p}\{\alpha_1,...,\alpha_n,\beta_1,...\beta_n\}\leq n$. 
    \item If $\beta_i/\alpha_i\in \mathbb{F}_p$ for all $i$, then $\alpha_1,...,\alpha_n$ are $\mathbb{F}_p$-linearly dependent. 
\end{enumerate}
\end{lemma}
\begin{proof}
\begin{enumerate}[label=\upshape{(\alph*)}]
    \item Consider the vector space $\mathbb{F}^{2n}_p$ with standard coordinates, with quadratic form $\langle-,-\rangle$ given by $Q_0$ as in Definition~\ref{def:isometricshuffle}. We now base change to $\mathbb{F}$.  Consider the vector 
$\mathbf{v}=(\alpha_1,...,\alpha_n,\beta_1,...\beta_n)\in \mathbb{F}^{2n}$. We use $\varphi^i\mathbf{v}$ to denote the $i$-th Frobenius twist of $\mathbf{v}$. Let  $V=\mathrm{Span}\{\varphi^{i}\mathbf{v}\}_{0\leq i\leq n-1}$. Then $V$ is totally isotropic. 
\begin{claim}
    $\varphi V=V$. In particular, $V$ descends to an $\mathbb{F}_p$-subspace of $\mathbb{F}^{2n}_p$. 
\end{claim}
The claim follows almost immediately from Proposition~\ref{NPstrata}(c) by first proving $q_r=0$ for all $r\geq n$. However, we give a more elementary proof.\footnote{This proof was found with the help of Gemini 3.0.} Suppose that $\dim V<n$. Let $d< n$ be the smallest integer such that $\varphi^d\mathbf{v}\in \mathrm{Span}_{\mathbb{F}}\{\varphi^i\mathbf{v}\}_{0\leq i\leq d-1}$. Iterating the Frobenius, we find that for all $r\geq d$, $\varphi^r\mathbf{v}\in \mathrm{Span}_{\mathbb{F}}\{\varphi^i\mathbf{v}\}_{0\leq i\leq d-1}$. It follows that $\varphi {V}=V$. Now suppose that $\dim V=n$. Then $V$ and $\varphi V$ are maximal isotropical subspaces of $\mathbb{F}^{2n}$. Clearly $\dim V\cap \varphi V\geq n-1$. It suffices to show that the inequality is strict. Suppose for the sake of contradiction that $\dim V\cap \varphi V= n-1$. Then there is a reflection in $\mathrm{O}(Q_0)(\mathbb{F})\setminus \mathrm{SO}(Q_0)(\mathbb{F})$ that takes $V$ to $\varphi V$. In particular, $V$ and $\varphi V$ are not in the same $\mathrm{SO}(Q_0)(\mathbb{F})$-orbit. On the other hand, since $Q_0$ is split over $\mathbb{F}_p$, the orthogonal Grassmanian $\mathrm{OG}(n,\mathbb{F}^{2n})$ parametrizing maximal isotropic subspaces is an $\mathbb{F}_p$-subscheme of $\mathrm{Gr}(n,\mathbb{F}^{2n}_p)$ with two components, each being an $\mathrm{SO}(Q_0)$-orbit. Since $\varphi$ preserves each component,  $V$ and $\varphi V$ must lie in the same component, hence in the same 
$\mathrm{SO}(Q_0)(\mathbb{F})$-orbit. We immediately arrive at a contradiction. So the claim holds.

Now let $V_0\subseteq\mathbb{F}^{2n}_p$ be the subspace with $V_0\otimes \mathbb{F}= V$. Pick $n$ vectors $\mathbf{w}_1,...,\mathbf{w}_n$ that spans $V_0$. Let $\theta_1,...,\theta_n\in \mathbb{F}$ be such that $\mathbf{v}=\sum_{i=1}^n\theta_i\mathbf{w}_i$. Then $\mathrm{Span}_{\mathbb{F}_p}\{\alpha_1,...,\alpha_n,\beta_1,...\beta_n\}\subseteq \mathrm{Span}_{\mathbb{F}_p}\{\theta_1,...,\theta_n\}$, which has dimension no more than $n$. 
    \item Let $\beta_i=\eta_i \alpha_i$, where $\eta_i\in \mathbb{F}_p^*$. We find that for $0\leq r\leq n-1$, 
$$\sum_{i=1}^n \eta_i\alpha_i^{p^r+1}=0\Rightarrow M(\alpha_1,...,\alpha_n)\begin{bmatrix}
   \alpha_1 \eta_1\\
    \alpha_2\eta_2\\
    \vdots\\
     \alpha_n\eta_n
\end{bmatrix}=0.$$
If $\alpha_1,...,\alpha_n$ are $\mathbb{F}_p$-linearly independent, then $\det M(\alpha_1,...,\alpha_n)\neq 0$. This implies that $\alpha_i\eta_i=0$ for all $i$, a contradiction.  
\end{enumerate}
\end{proof}

\subsubsection{Application to critical points}
\begin{corollary}\label{cor:po}
     Let $\{x_1,x_2,...,x_n,y_1,y_2,...,y_n\}$ be a  $Q_s$-degenerate collection of nonzero power series in $\mathbb{F}[\![t]\!]$, with leading coefficients $\alpha_1,...,\alpha_n,\beta_1,...\beta_n$, respectively. Suppose that $\max\{v_t(x_1), ...,v_t(x_n)\}\leq \min\{v_t(y_1),...,v_t(y_n)\}$. Let $\{\Pi_i\}_{1\leq i\leq n}$ be the associated line configuration.
     \begin{enumerate}[label=\upshape{(\alph*)}]
         \item\label{cor:po1} Suppose that $\ell\in \{\Pi_i\}$ is a non-redundant line \textbf{not} of the form $y=dx+d$, and let
         $\Sigma=\{i\in[1,n]:  \Pi_{i}=\ell\}$. 
         Suppose moreover that $\dim\mathrm{Span}_{\mathbb{F}_p}\{\alpha_{i}\}_{i\in \Sigma}=|\Sigma|$. Then $\ell$ contains at most $\min\{s+1,|\Sigma|-1\}$ edge-like critical points within $[1,p^s]\times \mathbb{R}$. If $\ell$ contains exactly $|\Sigma|-1$ edge-like critical points within $[1,p^s]\times \mathbb{R}$, then $\dim\mathrm{Span}_{\mathbb{F}_p}\{\beta_{i}\}_{i\in \Sigma}=|\Sigma|$.
         \item\label{cor:po2} Suppose that $\ell\in \{\Pi_i\}$ is a non-redundant line of the form $y=dx+d$, and let $\Sigma=\{i\in[1,n]: \Pi_{i}=\ell\}$. Then $\ell$ contains an edge-like critical point with $x$-coordinate 1. Suppose moreover that at least one of the following conditions holds \begin{itemize}
    \item $\dim\mathrm{Span}_{\mathbb{F}_p}\{\alpha_{i},\beta_i\}_{i\in \Sigma}>|\Sigma|$, 
    \item $\dim\mathrm{Span}_{\mathbb{F}_p}\{\alpha_{i}\}_{i\in \Sigma}=|\Sigma|$ and $\alpha_i/\beta_i\in \mathbb{F}_p$ for all $i\in \Sigma$. 
\end{itemize} Then $\ell$ contains at most $\min\{s+1,|\Sigma|-1\}$ edge-like critical points within $[1,p^s]\times \mathbb{R}$.
     \end{enumerate}
\end{corollary}
\begin{proof}Note that the line $\Pi_i$ has equation $y=v_t(x_i)x+v_t(y_i)$, and it is of the form $y=dx+d$ if and only if $v_t(x_i)=v_t(y_i)$.
\begin{enumerate}[label=\upshape{(\alph*)}]
    \item     
    Up to re-index, we can assume that $\Sigma$ is the subset $\{1,2,...,|\Sigma|\}$. Without loss of generality, we can also assume that $|\Sigma|-1\leq s$. Suppose for the sake of contradiction that $\ell$ contains $|\Sigma|$ edge-like critical points within $[p^{b},p^{b+|\Sigma|-1}]\times \mathbb{R}\subseteq [1,p^s]\times \mathbb{R}$. Then over the interval $[p^{b},p^{b+|\Sigma|-1}]$, every line other than $\cL$ is redundant.  Since for any $i\in \Sigma$ we must have $v_t(x_i)<v_t(y_i)$, we find that, for all $0\leq r\leq |\Sigma|-1$, we have $\sum_{i=1}^{|\Sigma|}\alpha_i^{p^{b+r}}\beta_i =0$.
    Lemma~\ref{lm:some_shit}(a) implies that  $\{\alpha_i\}_{i\in \Sigma}$ are $\mathbb{F}_p$-linearly dependent. Contradictory to our assumption. This shows that $\ell$ contains no more than $|\Sigma|-1$ edge-like critical points within $[1,p^s]\times \mathbb{R}$. 

If $\ell$ contains exactly $|\Sigma|-1$ edge-like critical points, say, they lie within $[p^{b},p^{b+|\Sigma|-2}]\times \mathbb{R}\subseteq [1,p^s]\times \mathbb{R}$. Then for all $0\leq r\leq |\Sigma|-2$, we have $\sum_{i=1}^{|\Sigma|}\alpha_i^{p^{b+r}}\beta_i =0$. The same reason as in the last paragraph implies that the collection is not $Q_{|\Sigma|-1}$-degenerate. Now we can apply Lemma~\ref{lm:some_shit}(b) to conclude that $\{\beta_{i}\}_{i\in \Sigma}$ are also $\mathbb{F}_p$-linearly independent.  
\item Suppose that $\ell$ does not contain an edge-like critical point with $x$-coordinate 1, then there exists a different non-redundant line, say $\Pi_n$, that contains a critical point with $x$-coordinate 1. In order for $\Pi_n$ and $\ell$ to be non-redundant, $\Pi_n$ must have bigger slope and smaller $y$-intercept than $\ell$. It follows that $v_t(x_n)>d>v_t(y_n)$, violating our assumption on valuations. So $\ell$ contains an edge-like critical point with $x$-coordinate 1. The rest of the  proof is similar to (a): instead of using Lemma~\ref{lm:some_shit}, we use Lemma~\ref{lm:leadingcoefogus2}(a) for the first bullet point, and Lemma~\ref{lm:leadingcoefogus2}(b) for the second bullet point.
\end{enumerate}
\end{proof}
\begin{corollary}\label{cor:maximalanddegenerate}
     Let $\{x_1,x_2,...,x_n,y_1,y_2,...,y_n\}$ be a $Q_s$-degenerate collection of nonzero power series in $\mathbb{F}[\![t]\!]$, which is maximal in its $Q$-isometry class. Suppose that $\max\{v_t(x_1), ...,v_t(x_n)\}\leq \min\{v_t(y_1),...,v_t(y_n)\}$. Write the line configuration $\{\Pi_i\}_{1\leq i\leq n}$ as a multiset $\{n_j\cdot \ell_j'\}_{1\leq j\leq u}$; cf.  Definition~\ref{def:lineconfig}(\ref{def:multiset}).
     \begin{enumerate}[label=\upshape{(\alph*)}]
         \item  Each $\ell'_i$ contains at most $\min\{s+1,n_j-1\}$ edge-like critical points within $[1,p^s]\times \bR$.
         \item   Let $\mathcal{R}$ be the set of indices $i$ such that $\Pi_i$ is not quasi-redundant over $[1,p^s]$; cf. Definition~\ref{def:lineconfig}(\ref{def:quasi-redundant}), then $|\mathcal{R}|\geq s+2$. 
         \item Suppose that $|\mathcal{R}|= s+2$. Write $\{\Pi_{i}\}_{i\in \mathcal{R}}$ as a multiset $\{m_j\cdot \ell_j\}_{1\leq j\leq w}$. Then \begin{enumerate}[label=\upshape{(\roman*)}]
             \item $\forall 1\leq j,j'\leq w$, $\ell_j\cap \ell_{j'}$ is a vertex-like critical point within $[1,p^s]\times \bR$. 
             \item Every $\ell_j$ contains exactly $m_j-1$ edge-like critical points within $[1,p^s]\times \bR$.
         \end{enumerate}
        (See Figure~\ref{figure2} for a graph of the line configuration when the slope of $\ell_j$'s  are strictly increasing.)
     \end{enumerate} 
\end{corollary}
\begin{proof}Let $\alpha_i,\beta_i$ be the leading coefficients of $x_i,y_i$.\begin{enumerate}[label=\upshape{(\alph*)}]
    \item Since the collection is maximal, by Lemma~\ref{lm:hypermaxiaml}, the condition that ``$\{\alpha_{i}\}_{i\in \Sigma}$ are $\mathbb{F}_p$-linearly independent'' in Corollary~\ref{cor:po}\ref{cor:po1} is met. Similarly, the condition of 
Corollary~\ref{cor:po}\ref{cor:po2} is also met up to a $Q$-isometry that does not change the line configuration: By Lemma~\ref{lm:hypermaxiaml}, we know that $$\dim\mathrm{Span}_{\mathbb{F}_p}\{\alpha_{i}\}_{i\in \Sigma}=\dim\mathrm{Span}_{\mathbb{F}_p}\{\beta_{i}\}_{i\in \Sigma}= |\Sigma|.$$ By 
Lemma~\ref{lm:quadmaximal}, we know that $$\dim\mathrm{Span}_{\mathbb{F}_p}\{\alpha_{i},\beta_i\}_{i\in \Sigma}\geq 2|\Sigma|-2.$$ So if the first condition is not met, i.e., $\dim\mathrm{Span}_{\mathbb{F}_p}\{\alpha_{i},\beta_i\}_{i\in \Sigma}\leq |\Sigma|$, then we must have $|\Sigma|\leq 2$. If $|\Sigma|=1$,  then we must have $\alpha_i/\beta_i\in \mathbb{F}_p$ for the unique $i\in \Sigma$, so the second condition is met. If $|\Sigma|=2$, then the space $V_d^+$ in Lemma~\ref{lm:quadmaximal} is inert. It is clear that by a $Q$-isometry that does not change the line configuration, we can assume that $\alpha_i/\beta_i\in \mathbb{F}_p$, so the second condition is met.
\item Let $\{\Pi_{i}\}_{i\in \mathcal{R}}=\{m_j\cdot \ell_j\}_{1\leq j\leq w}$. Let $N$ be the number of critical points in $[1,p^s]\times \bR$ that lie on $\ell_1,...,\ell_w$. Since every critical point in $[1,p^s]\times \bR$ lies on some $\ell_j$, we have $$N=s+1.$$ On the other hand, each $\ell_j$ contains at most $m_j-1$ edge-like critical points within $[1,p^s]\times \bR$ by (a). In addition, $w$ distinct lines can have at most $w-1$ intersections, so there are at most $w-1$ vertex like critical points within $[1,p^s]\times \bR$. As a result,
\begin{equation}\label{eq:criticalpointcounting}
   N\leq w-1 +\sum_{j=1}^w (m_j-1)=  -1+\sum_{j=1}^w m_j. 
\end{equation}
It follows that $$  -1+\sum_{j=1}^w m_j  \geq N= s+1.$$
So $|\mathcal{R}|=  \sum_{j=1}^w m_j \geq s+2$.
\item If $|\mathcal{R}|=s+2$, then (\ref{eq:criticalpointcounting}) is an equality. This means that any intersection $\ell_j\cap \ell_{j'}$ is a vertex-like critical point within $[1,p^s]\times \bR$, and every $\ell_j$ contains exactly $m_j-1$ edge-like critical points within $[1,p^s]\times \bR$.\end{enumerate}
\end{proof}
 \begin{center}\label{Figure1}

\tikzset{every picture/.style={line width=0.75pt}} 

\begin{tikzpicture}[x=0.75pt,y=0.75pt,yscale=-1,xscale=1, scale = 0.8]

\draw  [dash pattern={on 0.84pt off 2.51pt}]  (211.41,247.86) -- (204.93,258.56) -- (199,269) -- (196,274) -- (199,269) ;
\draw  [dash pattern={on 0.84pt off 2.51pt}]  (300,163.5) -- (266,191.5) ;
\draw  [dash pattern={on 0.84pt off 2.51pt}]  (411,116) -- (366,126) ;
\draw    (486,104) -- (182,132) ;
\draw [line width=0.75]    (512,93) -- (411,116) ;
\draw    (176,265.5) -- (266,191.5) ;
\draw    (428,61.5) -- (300,163.5) ;
\draw    (316,75.5) -- (211.41,247.86) ;
\draw    (158,336.5) -- (199,269) ;
\draw  [fill={rgb, 255:red, 7; green, 1; blue, 1 }  ,fill opacity=1 ] (222,225.5) .. controls (222,223.84) and (223.34,222.5) .. (225,222.5) .. controls (226.66,222.5) and (228,223.84) .. (228,225.5) .. controls (228,227.16) and (226.66,228.5) .. (225,228.5) .. controls (223.34,228.5) and (222,227.16) .. (222,225.5) -- cycle ;
\draw  [fill={rgb, 255:red, 7; green, 1; blue, 1 }  ,fill opacity=1 ] (447,107) .. controls (447,105.34) and (448.34,104) .. (450,104) .. controls (451.66,104) and (453,105.34) .. (453,107) .. controls (453,108.66) and (451.66,110) .. (450,110) .. controls (448.34,110) and (447,108.66) .. (447,107) -- cycle ;
\draw  [fill={rgb, 255:red, 7; green, 1; blue, 1 }  ,fill opacity=1 ] (211,243) .. controls (211,241.34) and (212.34,240) .. (214,240) .. controls (215.66,240) and (217,241.34) .. (217,243) .. controls (217,244.66) and (215.66,246) .. (214,246) .. controls (212.34,246) and (211,244.66) .. (211,243) -- cycle ;
\draw  [fill={rgb, 255:red, 7; green, 1; blue, 1 }  ,fill opacity=1 ] (193,274) .. controls (193,272.34) and (194.34,271) .. (196,271) .. controls (197.66,271) and (199,272.34) .. (199,274) .. controls (199,275.66) and (197.66,277) .. (196,277) .. controls (194.34,277) and (193,275.66) .. (193,274) -- cycle ;
\draw  [fill={rgb, 255:red, 7; green, 1; blue, 1 }  ,fill opacity=1 ] (183,290) .. controls (183,288.34) and (184.34,287) .. (186,287) .. controls (187.66,287) and (189,288.34) .. (189,290) .. controls (189,291.66) and (187.66,293) .. (186,293) .. controls (184.34,293) and (183,291.66) .. (183,290) -- cycle ;
\draw  [fill={rgb, 255:red, 7; green, 1; blue, 1 }  ,fill opacity=1 ] (239,211) .. controls (239,209.34) and (240.34,208) .. (242,208) .. controls (243.66,208) and (245,209.34) .. (245,211) .. controls (245,212.66) and (243.66,214) .. (242,214) .. controls (240.34,214) and (239,212.66) .. (239,211) -- cycle ;
\draw  [fill={rgb, 255:red, 7; green, 1; blue, 1 }  ,fill opacity=1 ] (174,306) .. controls (174,304.34) and (175.34,303) .. (177,303) .. controls (178.66,303) and (180,304.34) .. (180,306) .. controls (180,307.66) and (178.66,309) .. (177,309) .. controls (175.34,309) and (174,307.66) .. (174,306) -- cycle ;
\draw  [fill={rgb, 255:red, 7; green, 1; blue, 1 }  ,fill opacity=1 ] (256,197) .. controls (256,195.34) and (257.34,194) .. (259,194) .. controls (260.66,194) and (262,195.34) .. (262,197) .. controls (262,198.66) and (260.66,200) .. (259,200) .. controls (257.34,200) and (256,198.66) .. (256,197) -- cycle ;
\draw  [fill={rgb, 255:red, 7; green, 1; blue, 1 }  ,fill opacity=1 ] (318,147) .. controls (318,145.34) and (319.34,144) .. (321,144) .. controls (322.66,144) and (324,145.34) .. (324,147) .. controls (324,148.66) and (322.66,150) .. (321,150) .. controls (319.34,150) and (318,148.66) .. (318,147) -- cycle ;
\draw  [fill={rgb, 255:red, 7; green, 1; blue, 1 }  ,fill opacity=1 ] (423,113) .. controls (423,111.34) and (424.34,110) .. (426,110) .. controls (427.66,110) and (429,111.34) .. (429,113) .. controls (429,114.66) and (427.66,116) .. (426,116) .. controls (424.34,116) and (423,114.66) .. (423,113) -- cycle ;
\draw  [fill={rgb, 255:red, 7; green, 1; blue, 1 }  ,fill opacity=1 ] (472,105) .. controls (472,103.34) and (473.34,102) .. (475,102) .. controls (476.66,102) and (478,103.34) .. (478,105) .. controls (478,106.66) and (476.66,108) .. (475,108) .. controls (473.34,108) and (472,106.66) .. (472,105) -- cycle ;
\draw  [dash pattern={on 0.84pt off 2.51pt}]  (321,147) -- (336,140) -- (381,120) ;
\draw  [fill={rgb, 255:red, 7; green, 1; blue, 1 }  ,fill opacity=1 ] (301,160.5) .. controls (301,158.84) and (302.34,157.5) .. (304,157.5) .. controls (305.66,157.5) and (307,158.84) .. (307,160.5) .. controls (307,162.16) and (305.66,163.5) .. (304,163.5) .. controls (302.34,163.5) and (301,162.16) .. (301,160.5) -- cycle ;
\draw    (151,223) -- (336,140) ;
\draw    (381,120) -- (471,79) ;
\draw    (117,374) -- (581,373) ;

\draw (140.16,309.28) node [anchor=north west][inner sep=0.75pt]  [rotate=-1.16]  {$\ell_{w}$};
\draw (144,263) node [anchor=north west][inner sep=0.75pt]   [align=left] {$\displaystyle \ell_{w-1}$};
\draw (157,212) node [anchor=north west][inner sep=0.75pt]   [align=left] {$\displaystyle \ell_{w-2}$};
\draw (163,121) node [anchor=north west][inner sep=0.75pt]   [align=left] {$\displaystyle \ell_{1}$};
\draw (428,282) node [anchor=north west][inner sep=0.75pt]   [align=left] {$\displaystyle \Omega $};
\draw (152,162) node [anchor=north west][inner sep=0.75pt]   [align=left] {$\displaystyle \vdots$};
\draw (174,380) node [anchor=north west][inner sep=0.75pt]   [align=left] {1};
\draw (204,374) node [anchor=north west][inner sep=0.75pt]   [align=left] {$\displaystyle p^{2}$};
\draw (189,380) node [anchor=north west][inner sep=0.75pt]   [align=left] {$\displaystyle p$};
\draw (223,380) node [anchor=north west][inner sep=0.75pt]   [align=left] {$\displaystyle \cdots $};

\end{tikzpicture}

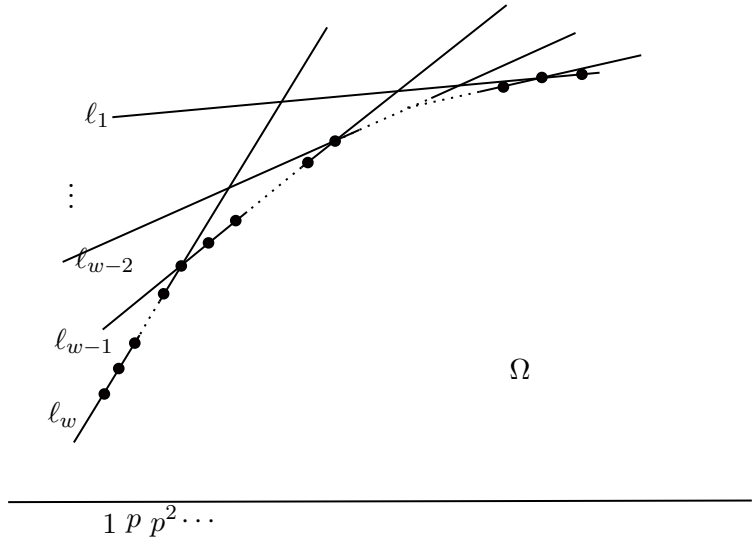
\captionof{figure}{Line configuration of $\{m_j\cdot \ell_j\}_{1\leq j\leq w}$ with critical points.}\label{figure2}

\end{center}
\subsection{Pattern of the first order decay}\label{sub:patternfirstorderdecay}
Let recall the setup in \S\ref{sub:setuplocalcurve}: we have a collection \begin{equation}\label{eq:collection}
\{x_1,...,x_m, y_1,...,y_m\}
\end{equation}of nonzero series in $\mathbb{F}[\![t]\!]$ given by the equation of the formal curve $C$. We also have $Q_0(t)=...=Q_{a-1}(t) = 0$ and $ Q_a(t)\neq 0$. We have defined  $h_i=v_t(Q_i(t))$. Furthermore,  $a_1 = a < a_2 \hdots < a_k$ is the maximal sequence of integers such that $a_k \leq m-1$, and $h_{a_1} \geq h_{a_2} \hdots \geq h_{a_k}$. Note that any collection in the $Q$-isometry class of (\ref{eq:collection}) is $Q_{a_k-1}$-degenerate. 
\begin{setup}
\label{setup: pattern of first order decay}
We will always work with a collection \eqref{eq:collection} which is maximal in its $Q$-isometry class. 
\end{setup} 
\begin{definition}
Let $\Lambda\subseteq\cL$ denote a sublattice of $\cL$. We will put a bar on $\Lambda$ to denote the mod $p\cL$ reduction of $\Lambda$, regraded as a sublattice of $\overline{\cL}$. For example, $\overline{\mathcal{L}_1}=\mathcal{L}_1/p\cL=\mathrm{Span}_{\bZ_p}\{e_1,...,e_m, f_1,...,f_m\}\otimes \mathbb{F}_p$. We denote by $\overline{e'},\overline{f'}, \overline{e_i}, \overline{f_i}$ the image of $e',f',e_i,f_i$ in $\overline{\cL}$. We will also naturally identify $ \overline{\mathcal{L}_1}$ with $V_{\mathrm{frame}}$; cf. \S\ref{def:isometricshuffle}, by $\overline{e_i}\rightarrow X_i$ and $\overline{f_i}\rightarrow Y_i$. 
\end{definition}
\begin{definition}\label{def:somebadass}~\begin{enumerate}
     \item Let $\{\Pi_i\}_{1\leq i\leq m}$ be the line configuration associated to the collection (\ref{eq:collection}). Let $r$ be the number of quasi-redundant lines over $[1,p^{a_k}]$, and let $r_0\leq r$ be the number of redundant lines over $[1,p^{a_k}]$. Let $n=m-r$, and $n_0=m-r_0$. 
     \item  If $n\leq a_k+1$, we say that the collection $\{\underline{x},\underline{y}\}$ is \textit{thin}. Otherwise, it is \textit{thick}.
     \item A \textit{borderline} is a line $\Pi_i$ which is not quasi-redundant over $[1,p^{a_k}]$, and has equation $y=dx+d$. If a borderline exists, then it is unique. 
     \item Define the \textit{redundant subspace} $\overline{\mathcal{L}}_{\text{rd}}\subseteq\overline{\mathcal{L}_{1}}$ to be the subspace spanned by all $\overline{e_i},\overline{f_i}$ for those $i$ such that $\Pi_i$ is quasi-redundant over $[1,p^{a_k}]$.

\item\label{def:boarderllinemodp} Define the \textit{borderline subspace} $\overline{\mathcal{L}}_{\text{bd}}\subseteq \overline{\mathcal{L}}_{\mathrm{rv}}$ as follows: if a borderline exists, and let $d$ be the border, then $\overline{\mathcal{L}}_{\text{bd}}$ is the subspace $V^+_d \subseteq V_{\mathrm{frame}}=\overline{\cL_1}$ in Lemma~\ref{lm:quadmaximal}, which is either $\mathbf{0}$, or nondegenerate of dimension 1, or inert of dimension 2. If a borderline does not exist, then we set $\overline{\mathcal{L}}_{\text{bd}}=\mathbf{0}$.
\item Define $\overline{\mathcal{L}}_{\text{brd}}:=\overline{\mathcal{L}}_{\text{bd}}\oplus \overline{\mathcal{L}}_{\text{rd}}$. Here we are using the notation $\overline{\mathcal{L}}_{\text{brd}}$ instead of $\overline{\mathcal{L}_{\text{brd}}}$ since there is no apriori defined $\bZ_p$-lattice  ${\mathcal{L}_{\text{brd}}}$. The reader is encouraged to ignore this notational nuance: later on, we will define a $\bZ_p$-lattice  ${\mathcal{L}_{\text{brd}}}$ whose mod $p\cL$ reduction is $\overline{\mathcal{L}}_{\text{brd}}$, and then $\overline{\mathcal{L}_{\text{brd}}}=\overline{\mathcal{L}}_{\text{brd}}$. 

  \end{enumerate} 
 \end{definition}
\begin{assumption}\label{ass:redundantlines}
Starting from a maximal collection $\{\underline{x},\underline{y}\}$, we may  perform finitely many steps of re-indexing and swapping so that $\{\underline{x},\underline{y}\}$ satisfies one of the two following two assumptions (the two assumptions \textit{may not} be mutually compatible):  
    \begin{enumerate}[label=({A}{{\arabic*}})]
        \item \label{ass:red1} $\Pi_m,\Pi_{m-1},...,\Pi_{n_0+1}$ are redundant over $[1,p^{a_k}]$,
        
\noindent        $\forall i\in [1,m]$, $v_t(x_i)\leq v_t(y_i)$, 
        
       \noindent  $v_t(x_1)\leq v_t(x_2)\leq ...\leq v_t(x_{n_0})$.
        \item \label{ass:red2} $\Pi_m,\Pi_{m-1},...,\Pi_{n_0+1}$ are redundant over $[1,p^{a_k}]$,
        
    \noindent $\Pi_{n_0},\Pi_{n_0-1},...,\Pi_{n+1}$ are quasi-redundant but not redundant over $[1,p^{a_k}]$,
    
    \noindent $\forall i\in [1,m]$, $v_t(x_i)\leq v_t(y_i)$,
    
    \noindent $v_t(x_1)\leq v_t(x_2)\leq ...\leq v_t(x_{n})$.
    \end{enumerate}

\end{assumption}
\begin{remark}
As the property of being maximal is preserved under re-indexing and swapping, imposing one of the above assumptions on $\{\underline{x},\underline{y} \}$ is consistent maximality. 
\end{remark}

\begin{lemma}\label{lm:beforelemma}Under Assumption~\ref{ass:redundantlines}\ref{ass:red1}, we have $v_t(x_{n_0}) \leq v_t(y_{i})$ for all $1\leq i\leq m$. Furthermore, we have 
\begin{equation}\label{eq:reindexthin}
        v_t(x_1)\leq v_t(x_2)\leq ...\leq v_t(x_{n_0}) \leq v_t(y_{n_0})\leq v_t(y_{n_0-1})\leq ...\leq v_t(y_1).\end{equation}
\end{lemma}
\begin{proof} We first prove that $v_t(x_{n_0}) \leq v_t(y_{i})$ for all $i$. Suppose that this is not the case, let $j$ be such that $v_t(y_j)=\min\{v_t(y_1),...,v_t(y_{m})\}$. Then $v_t(y_j)<v_t(x_{n_0})$. Then $j\neq n_0$ and $v_t(x_j)\leq v_t(y_j)<v_t(x_{n_0})\leq v_t(y_{n_0})$. This means that $\Pi_{n_0}\succ \Pi_j$, rendering $\Pi_{n_0}$ totally redundant. A contradiction. To show that (\ref{eq:reindexthin}) holds, note that if there is some $v_t(y_i)>v_t(y_{i-1})$, then $\Pi_i\succ \Pi_{i-1}$, rendering $\Pi_i$ totally redundant. A contradiction.

\end{proof}

  We have the following two key propositions.
\begin{proposition}\label{prop:thincase!!!}Assume \ref{ass:redundantlines}\ref{ass:red2}.\begin{enumerate}[label=\upshape{(\alph*)}]        
        \item\label{it:prop:thincase!!!1} We have $v_t(x_{n}) \leq v_t(y_{i})$ for all $1\leq i\leq m$, and $v_t(x_{n}) < v_t(y_{i})$ if $n<i\leq m$. Furthermore, 
\begin{equation}\label{eq:reindexthin3}
        v_t(x_1)\leq v_t(x_2)\leq ...\leq v_t(x_{n}) \leq v_t(y_{n})\leq v_t(y_{n-1})\leq ...\leq v_t(y_1).
     \end{equation}
     
     \item\label{thineqs} Suppose that the collection is thin. Then $n=a_k+1$. Moreover, we have \begin{equation}\label{eq:thinv}
         v_t(y_{n-i-1})-v_t(y_{n-i})= p^{i}[v_t(x_{n-i})-v_t(x_{n-i-1})],\;\;0\leq i\leq n-2.
     \end{equation}
 \begin{equation}\label{eq:thinh}
 h_{a_k}\geq v_t(y_1)+ p^{a_k}v_t(x_1).
 \end{equation}
\item\label{fateqs} Suppose that the collection is thick. Then \begin{equation}\label{eq:fatv}
     p^{i}[v_t(x_{n-i})-v_t(x_{n-i-1})]\leq v_t(y_{n-i-1})-v_t(y_{n-i})\leq p^{i+1}[v_t(x_{n-i})-v_t(x_{n-i-1})],\;\;0\leq i\leq n-2
\end{equation}
\begin{equation}\label{eq:fath}
 h_{a_k}\geq v_t(y_1)+ p^{a_k}v_t(x_1)
\end{equation}
 \end{enumerate}
\end{proposition}
\begin{proof}\begin{enumerate}[label=\upshape{(\alph*)}] 
\item Ignore the assertion $v_t(x_n)< v_t(y_i)$ for $n<i\leq m$, the statement follows directly from Lemma~\ref{lm:beforelemma} by throwing away quasi-redundant indices. Now we treat the case $n<i\leq m$ separately. It suffices to show that $v_t(x_n)\neq v_t(y_i)$ for all $n<i\leq m$. Suppose this is not the case, i.e., there exists some such $i$ such that $v_t(x_i)\leq v_t(y_i)= v_t(x_n)\leq v_t(y_n)$. Then $\Pi_{n}\succeq \Pi_i$. Since $\Pi_i$ is quasi-redundant over $[1,p^{a_k}]$, $\Pi_{n}$ must also be quasi-redundant. Contradiction.
\item Note that the claim ``$n=a_{k}+1$'' only depends on the multiset of lines, so it suffices to prove $n\geq a_k+1$ under Assumption~\ref{ass:redundantlines}\ref{ass:red1}. Note that the subcollection $\{x_i,y_i\}_{1\leq i\leq n_0}$ is again maximal and $Q_{a_k-1}$-degenerate. Apply Corollary~\ref{cor:maximalanddegenerate} to $\{x_i,y_i\}_{1\leq i\leq n_0}$ with $s=a_k-1$. Note that all conditions of the corollary are met (the condition $\max\{v_t(x_1), ...,v_t(x_{n_0})\}\leq \min\{v_t(y_1),...,v_t(y_{n_0})\}$ is satisfied because of Lemma~\ref{lm:beforelemma}). It follows that $n\geq a_k+1$ by the second part Corollary~\ref{cor:maximalanddegenerate}. Since being ``thin'' means $n\leq a_k+1$, we find that $n=a_k+1$. 

For the formulas (\ref{eq:thinv}) and (\ref{eq:thinh}), we use the third part of Corollary~\ref{cor:maximalanddegenerate}. To help visualize it, we note that the critical points and the non-quasi-redundant lines are arranged in Figure~\ref{figure2}. Now switch back to Assumption~\ref{ass:redundantlines}\ref{ass:red2}. It is clear that $\Pi_{n-i-1}\cap\Pi_{n-i}$ contains a point with $x$-coordinate $p^i$. As a result, 
$$p^iv_t(x_{n-i-1})+v_t(y_{n-i-1}) =p^iv_t(x_{n-i})+v_t(y_{n-i}).$$
This proves (\ref{eq:thinv}). 

Note that $\Pi_1$ lies below other lines in the collection $\{\Pi_i\}_{1\leq i\leq n}$ over the interval $[p^{a_k},\infty)$. If $\Pi_i$ is quasi-redundant, then it surely lies above $\Pi_1$, at least over the point $x=p^{a_k}$. It follows that \begin{equation}\label{eq:hV1}
h_{a_k}\geq \Pi_1(p^{a_k})=v_t(x_1)p^{a_k}+ v_t(y_1).
\end{equation}
\item The thick case can be deduced in a manner similar to the thin case. In fact, the graph of $\{\Pi_i\}_{1\leq i\leq n}$ again looks like Figure~\ref{figure2}, except that the intersection of two lines may no longer contain a vertex-like critical point. So the outcome is weaker: $\Pi_{n-i}\cap\Pi_{n-i-1}$ contains an intersection point with $x$-coordinate in $[p^i,p^{i+1}]$. Using this, we get (\ref{eq:fatv}). By a similar reason as in the thin case, we have $h_{a_k}\geq v_t(x_1)p^{a_k}+ v_t(y_1)$.
\end{enumerate}
\end{proof}
\begin{proposition}\label{prop:L+thin} 
Assume \ref{ass:redundantlines}\ref{ass:red2}. Let $d=v_t(x_n)= \max\{v_t(x_i)\}_{1\leq i\leq n}$. We have the following description of $\overline{L'_{P,r}\otimes \bZ_p}$:
        \begin{enumerate}[label=\upshape{(\alph*)}]
            \item If $r\leq d$, then $$\overline{L'_{P,r}\otimes \bZ_p}=\mathrm{Span}_{\mathbb{F}_p}\{\overline{e'},\overline{f'},\overline{f_1},\overline{f_2},...,\overline{f_m}\}\oplus \mathrm{Span}_{\mathbb{F}_p} \{\overline{e_i}: v_t(x_i)\geq r\}_{1\leq i\leq m}.$$ 
        \item\label{it:prop:L+thin2} If $d<r\leq h_{a_k}$, then $$\overline{L'_{P,r}\otimes \bZ_p}=\mathrm{Span}_{\mathbb{F}_p}\{\overline{e'},\overline{f'},\overline{f_i}:v_t(y_i)\geq r\}_{1\leq i \leq n}\oplus \overline{\Delta}_r,$$ 
    where $$\overline{\Delta}_r\subseteq \overline{\mathcal{L}}_{\mathrm{brd}}\oplus\mathrm{Span}_{\mathbb{F}_p}\{\overline{f_i}:v_t(y_i)< r\}_{1\leq i \leq n}$$ is a subspace whose projection to $\overline{\mathcal{L}}_{\mathrm{brd}}$ is an isometrical embedding. 
    
 \item If $r> h_{a_k}$ and $a_k=a$, then 
   $$\overline{L'_{P,r}\otimes \bZ_p} \subseteq \overline{\mathcal{L}}_{\mathrm{brd}}\oplus\mathrm{Span}_{\mathbb{F}_p}\{\overline{e'},\overline{f'},\overline{f_1},\overline{f_2},...,\overline{f_n}\}$$ is a subspace whose projection to $\overline{\mathcal{L}}_{\mathrm{brd}}$ is an isometrical embedding. 
        \end{enumerate}
\end{proposition}
\begin{proof} Recall that we are working in the setup that $\{\Pi_i\}_{1\leq i\leq m}$ is maximal in its $Q$-isometry class. By Lemma~\ref{lm:hypermaxiaml}, $x_1,...,x_m$ are $\mathbb{F}_p$-hyper-independent, the same holds for $y_1,...,y_m$.  By Proposition~\ref{prop:thincase!!!}\ref{it:prop:thincase!!!1}, we also have $d \leq v_t(y_{i})$ for all $1\leq i\leq m$ and $d<v_t(y_i)$ for $i>n$. In addition, $\overline{e'},\overline{f'}\in \overline{L'_{P,r}\otimes \bZ_p}$ for all $r\leq h_{a_k}$. These results will be freely used in the proof without mentioning. 
\begin{enumerate}[label=\upshape{(\alph*)}]
    \item The $r\leq d$ case follows easily from the above mentioned results. 
    \item Suppose now that $d<r\leq h_{a_k}$. Up to a $Q$-isometry that keeps Assumption~\ref{ass:redundantlines}\ref{ass:red2}, we can assume that $\overline{\cL}_{\mathrm{bd}}=\mathrm{Span}_{\mathbb{F}_p}\{\overline{e_n}-\gamma_n\overline{f_n}\}$ (when $\dim \overline{\cL}_{\mathrm{bd}}=1$) or $\mathrm{Span}_{\mathbb{F}_p}\{\overline{e_{n-1}}-\gamma_{n-1}\overline{f_{n-1}},\overline{e_n}-\gamma_n\overline{f_n}\}$ (when $\dim \overline{\cL}_{\mathrm{bd}}=2$). We then define the following \textit{relevant subspace} which is complementary to $\overline{\cL}_{\mathrm{bd}}$: $$\overline{\cL}_{\mathrm{rv}}:=\mathrm{Span}_{\mathbb{F}_p}\{\overline{e'},\overline{f'},\overline{e_j},\overline{f_i}\}_{1\leq i \leq n,1\leq j\leq n-\dim \overline{\cL}_{\mathrm{bd}}}.$$
Since $y_1,...,y_m$ are $\mathbb{F}_p$-hyper-independent, we see that $\overline{L'_{P,r}\otimes \bZ_p}\cap \overline{\cL}_{\mathrm{rv}}=\mathrm{Span}_{\mathbb{F}_p}\{\overline{e'},\overline{f'},\overline{f_i}:v_t(y_i)\geq r\}_{1\leq i \leq n}$. Suppose that $e\in \overline{L'_{P,r}\otimes \bZ_p}$ is an element that is not contained in $\overline{\cL}_{\mathrm{rv}}$. After removing its components in $\overline{e'},\overline{f'}$, we can assume that $e\in \mathrm{Span}_{\mathbb{F}_p}\{\overline{e_i},\overline{f_i}\}_{1\leq i \leq m}$. We can uniquely write $e=e_{\mathrm{brd}}+e_{\mathrm{rv}}$, where $0\neq e_{\mathrm{brd}}\in \overline{\cL}_{\mathrm{brd}}$ and $e_{\mathrm{rv}}\in \overline{\cL}_{\mathrm{rv}}$. More explicitly:
$$e_{\mathrm{rv}}=\sum_{j=1}^{n-\dim \overline{\cL}_{\mathrm{bd}}}  a_j\overline{e_j}+\sum_{j=1}^{n}b_j\overline{f_j},\;\;\;e_{\mathrm{brd}}= \heartsuit+\sum_{i=n+1}^{m}(c_i\overline{e_i}+d_i\overline{f_i}),$$
where $\heartsuit$ is 0, $c(\overline{e_{n}}-\gamma_{n}\overline{f_{n}})$ or a linear combination of $\overline{e_{n-1}}-\gamma_{n-1}\overline{f_{n-1}},\overline{e_{n}}-\gamma_{n}\overline{f_{n}}$, depending on what $\overline{\cL}_{\mathrm{bd}}$ is. Let $x\in \mathbb{F}[\![t]\!]$ be the power series corresponding to $e$, similar for $x_{\mathrm{brd}}$ and $x_{\mathrm{rv}}$. 

In the following, we claim that $e_{\mathrm{rv}}\in \mathrm{Span}_{\mathbb{F}_p}\{\overline{f_i}\}_{1\leq i \leq n}$. In other words, $a_j=0$ in the above expression. Note that (\ref{prop:L+thin}) follows immediately from the claim. 

We prove the claim by contradiction. Suppose the claim does not hold. Then there exists a smallest integer $j_0$ such that $a_{j_0}\neq 0$. Let $\lambda\leq j_{0}$ be the smallest integer that $v_t(x_{\lambda})=v_t(x_{j_0})$. Then $v_t(x_{\mathrm{rv}})=v_t(x_{\lambda})\leq d<r$.  On the other hand, we know that $v_t(x)\geq r$. It follows that $v_t(x_{\mathrm{brd}})=v_t(x_{\mathrm{rv}})$ and the leading terms of $x_{\mathrm{brd}}$ and $x_{\mathrm{rv}}$ must cancel out. Let $d'=\min_{1\leq i\leq m}\{v_t(y_i)\}$, this quantity is no smaller than $d$. Suppose that $v_t(x_{\lambda})<d'$. Then any $\mathbb{F}_p$-combination of $y_i$'s has degree bigger than $v_t(x_{\lambda})$. On the other hand, $\overline{e_{n+1-i}}-\gamma_{n+1-i}\overline{f_{n+1-i}}$ for $1\leq i\leq \dim\overline{\cL}_{\mathrm{bd}}$ all have degree no less than $d+1>v_t(x_{\lambda})$. As a result, if the leading terms of $x_{\mathrm{brd}}$ and $x_{\mathrm{rv}}$ cancel, then $x_1,...,x_m$ are not  $\mathbb{F}_p$-hyper-independent. So this case can not happen. Now suppose that $v_t(x_{\lambda})=d'$. Then $d=d'$. By Proposition~\ref{prop:thincase!!!}\ref{it:prop:thincase!!!1} and the non-redundancy of $\{\Pi_i\}_{1\leq i\leq n}$, we have  $d=v_t(x_{\lambda})=v_t(x_{\lambda+1})=...=v_t(x_{n})=v_t(y_{n})=v_t(y_{n-1})=...=v_t(y_{\lambda})>v_t(y_{\lambda-1})$. 
Since $v_t(y_i)> d$ for all $n<i\leq m$, we find that 
\begin{equation}\label{eq:cancellationxx}
    \sum_{j=\lambda}^{n-\dim \overline{\cL}_{\mathrm{bd}}}  a_jx_j+\sum_{j=\lambda}^{n}b_jy_j + \sum_{\substack{i>n\\v_t(x_i)=d}} c_ix_i 
\end{equation}
has degree bigger than $d$. Now up to a re-indexing we can assume that the set $\{i>n:v_t(x_i)=d\}$ is $\{n+1,n+2,...,n+l\}$, and furthermore $v_t(y_{n+1})\leq ...\leq v_t(y_{n+l})$. Applying Lemma~\ref{lm:quadmaximal333} to the collection $\{x_{\lambda},x_{\lambda+1},...,x_{n},...x_{n+l};  y_{\lambda},y_{\lambda+1},...,y_{n},...,y_{n+l}\}$, we find that (\ref{eq:cancellationxx}) has degree equal to $d$. This leads to a contradiction.
\item Take $r=h_{a_k}$ in \ref{it:prop:L+thin2} and use the fact that $h_{a_k}>v_t(y_1)$ from Proposition~\ref{prop:thincase!!!}, we see that $\overline{L'_{P,h_{a_k}}\otimes \bZ_p}=\overline{\cL_0}\oplus \overline{\Delta}_{h_{a_k}}$, where $\overline{\Delta}_{h_{a_k}}\subseteq \overline{\mathcal{L}}_{\mathrm{brd}}\oplus \mathrm{Span}_{\mathbb{F}_p}\{\overline{f_i}\}_{1\leq i \leq n}$ whose projection to $\overline{\mathcal{L}}_{\mathrm{brd}}$ is an isometrical embedding. 

Now suppose that $r>h_{a_k}$ and $a_k=a$. The lattice $\cL_0$ decays rapidly to arbitrary order. So $\overline{L'_{P,r}\otimes \bZ_p}\cap \overline{\cL_0}=\mathbf{0}$. Since $\overline{L'_{P,r}\otimes \bZ_p}\subseteq \overline{L'_{P,h_{a_k}}\otimes \bZ_p}$, the desired claim follows easily.  
\end{enumerate}
\end{proof}

\subsection{Estimating $S'(M)$: preparation}
We continue to use the setup in \S\ref{sub:patternfirstorderdecay}, and will assume that our maximal collection satisfies Assumption~\ref{ass:redundantlines}\ref{ass:red2}.

In this section, we make various preparations for computing $S'(M)$. First, we split out cases from three perspectives. Second, we define auxiliary lattices $L_{P,r}''$ generalizing the construction in the toy example \S\ref{subsub:toyexamplem=2}, based on a cases by case study (however, the explicit  computation of $L_{P,r}''$, and the relation between $S'(M)$ and $S''(M)$ will be carried out in the next two sections \S\ref{sub:hardsitu}$\sim$\S\ref{sub:easysitu}). Third, we compute local density for $L_{P,r}''$. Fourth, we estimate a quantity $D_0$ which will appear in various computation later on. 

\subsubsection{Split out cases and set up notations} We have three different rules to split out cases: depending on whether $a_k>a$, we have two cases; depending on the dimension of 
$\overline{\cL}_{\mathrm{bd}}$, we have three cases; depending on whether $n>a_k+1$, we have two cases. Let's spell them out: \begin{enumerate}
    \item If $a_k=a$ (resp. $a_k>a$), we say that we are in  \textit{tight} (resp. \textit{loose}) \textit{case}.
    \item If $\dim\overline{\cL}_{\mathrm{bd}}=i\in \{0,1,2\}$, we say that we are in \textit{borderline case} $i$. 
\item If $n> a_k+1$ (resp. $n= a_k+1$), we say that we are in \textit{thick} (resp. \textit{thin}) \textit{case}. This is already defined in Definition~\ref{def:somebadass}, and the fact that $n\leq a_k+1$ implies $n=a_k+1$ is proved in Proposition~\ref{prop:thincase!!!}.
\end{enumerate}
The three different rules do not depend on each other, so we have $2\times 3\times 2=12$ subcases in total. For example, when $a_k=a$, $\dim\overline{\cL}_{\mathrm{bd}}=2$ and $n>a_k+1$, we are in tight thick borderline case 2. 


\begin{notation}
\label{not: c_i's and d_i's}
    To easy notation, from now on we will write $c_i=v_t(x_i)$ and $d_i=v_t(y_i)$ for all $i\in [1,m]$. We rewrite the formulae (\ref{eq:reindexthin3}), (\ref{eq:thinv}), (\ref{eq:fatv}), (\ref{eq:thinh}) and (\ref{eq:fath}) from Proposition \ref{prop:thincase!!!} as follows:
\begin{equation}
       \text{thin \& thick cases:\;\;\;\;} c_1\leq c_2\leq ...\leq c_n\leq d_n\leq d_{n-1}\leq ...\leq d_1.
     \end{equation}
\begin{equation}\label{eq:thinv2}
      \text{thin case:\;\;\;\;}   d_{i}-d_{i+1}= p^{n-1-i}(c_{i+1}-c_{i}),\;\;1\leq i\leq n-1.
     \end{equation}
 \begin{equation}\label{eq:fatv2}
      \text{thick case:\;\;\;\;}  p^{n-1-i}(c_{i+1}-c_{i}) \leq  d_{i}-d_{i+1} \leq    p^{n-i}(c_{i+1}-c_{i}),\;\;1\leq i\leq n-1.
     \end{equation}
      \begin{equation}\label{eq:hv2}
      \text{thin \& thick cases:\;\;\;\;}  h_{a_k}\geq p^{a_k}c_1+d_1.
     \end{equation}
\end{notation}

\subsubsection{Auxiliary lattices: the tight case}\label{subsub:auxlatticetight} Assume that $a_k=a$. 
In this section, we define a sequence of auxiliary lattices $L_{P,r}''$, generalizing the construction in the toy example \S\ref{subsub:toyexamplem=2}.

Recall we have the borderline subspace $\overline{\mathcal{L}}_{\mathrm{bd}}$ and redundant subspace $ \overline{\mathcal{L}}_{\mathrm{rd}}$ as per Definition~\ref{def:somebadass}. We now define  $\bZ_p$-lattices $ \mathcal{L}_{\mathrm{bd}}$ and $\mathcal{L}_{\mathrm{rd}}$ as follows. We define $\mathcal{L}_{\mathrm{rd}}:=\mathrm{Span}_{\mathbb{F}_p}\{e_j,f_j\}_{n<j\leq m}$. If $\dim\overline{\mathcal{L}}_{\mathrm{bd}}=0$, we define ${\mathcal{L}}_{\mathrm{bd}}=\mathbf{0}$.  If $\dim\overline{\mathcal{L}}_{\mathrm{bd}}=1$ (resp. 2), then up to a $Q$-isometry that keeps Assumption~\ref{ass:redundantlines}\ref{ass:red2}, we can assume $\overline{\mathcal{L}}_{\mathrm{bd}}=\{\overline{e_n}-\gamma_n\overline{f_n}\}$ (resp. $\{\overline{e_{n-1}}-\gamma_{n-1}\overline{f_{n-1}}, \overline{e_n}-\gamma_n\overline{f_n}\}$), and we define ${\mathcal{L}}_{\mathrm{bd}}=\mathrm{Span}_{\bZ_p}\{e_{n}-\gamma_{n}f_{n}\}$ (resp. $\mathrm{Span}_{\bZ_p}\{e_{n-1}-\gamma_{n-1}f_{n-1}, e_{n}-\gamma_{n}f_{n}\}$), where by abuse of notation $\gamma_{n-1},\gamma_{n}\in \bZ_p$ are Teichmuller lifts of $\gamma_{n-1},\gamma_{n}\in \mathbb{F}_p$. Let $\mathcal{L}_{\mathrm{brd}}:=\mathcal{L}_{\mathrm{bd}}\oplus \mathcal{L}_{\mathrm{rd}}$. For $0\leq i\leq n$, we define  \begin{align*}
    &\mathcal{L}_{> i}=\mathrm{Span}_{\mathbb{Z}_p}\{e',f',{e}_{j},{f}_{j'}\}_{i<j\leq n, 1\leq j' \leq n }.\\&\mathcal{L}'_{\leq  i}=\mathrm{Span}_{\mathbb{Z}_p}\{e',f',{f}_{j'}\}_{1\leq j' \leq i}.
\end{align*}
Note that by definition, we have $\mathcal{L}_{> n}=\cL'_{\leq n}$. These lattices are closely related to the description in Proposition~\ref{prop:L+thin}. 

Moreover, for each $j\geq 1$, define $H_{j}:=h_P(1+p^{a_k+1}+...+p^{(a_k+1)(j-1)})$ and $H_{0}:=0$.

Finally, define $L''_{P,r}$ to be the $\bZ$-lattice such that $L''_{P,r}\otimes \bZ_l=L'_{P,r}\otimes \bZ_l$ for $l\neq p$, and 
\begin{equation*}
 L''_{P,r}\otimes\bZ_p= \left\{\begin{aligned}
    & \mathcal{L}_{\mathrm{brd}}+p^j\cL_{>0}+ p^{j+1}\mathcal{L},\,\,\,\, &r\in (H_{j},H_{j}+p^{(a_k+1)j}c_1],\\
     & \mathcal{L}_{\mathrm{brd}}+ p^{j}\mathcal{L}_{>1}+ p^{j+1}\mathcal{L},\,\,\,\, &r\in (H_{j}+p^{(a_k+1)j}c_1,H_{j}+p^{(a_k+1)j}c_2],\\
&\;\;\;\;\;\;\;\;\;\;\;\;\;\;\;\;\;\;\;\vdots\;\;\;\;\;\;\;\;\;\;\;\;\;\;,\;\;\;\;\;\;\;\;\;\;\;&\vdots\;\;\;\;\;\;\;\;\;\;\;\;\;\;\;\;\;\;\;\;\;\;\;\;\;,\\
&\mathcal{L}_{\mathrm{brd}}+ p^{j}\mathcal{L}_{>n-1}+ p^{j+1}\mathcal{L},\,\,\,\, &r\in (H_{j}+p^{(a_k+1)j}c_{n-1},H_{j}+p^{(a_k+1)j}c_n],\\
&\mathcal{L}_{\mathrm{brd}}+ p^{j}\mathcal{L}'_{\leq n}+ p^{j+1}\mathcal{L},\,\,\,\, &r\in (H_{j}+p^{(a_k+1)j}c_{n},H_{j}+p^{(a_k+1)j}{d_n}],\\
&\mathcal{L}_{\mathrm{brd}}+ p^{j}\mathcal{L}'_{\leq n-1}+ p^{j+1}\mathcal{L},\,\,\,\, &r\in (H_{j}+p^{(a_k+1)j}d_{n},H_{j}+p^{(a_k+1)j}{d_{n-1}}],\\
&\;\;\;\;\;\;\;\;\;\;\;\;\;\;\;\;\;\;\;\vdots\;\;\;\;\;\;\;\;\;\;\;\;\;\;,&\vdots\;\;\;\;\;\;\;\;\;\;\;\;\;\;\;\;\;\;\;\;\;\;\;\;\;,\\
&\mathcal{L}_{\mathrm{brd}}+ p^{j}\mathcal{L}'_{\leq 0}+ p^{j+1}\mathcal{L},\,\,\,\, &r\in (H_{j}+p^{(a_k+1)j}d_{1},H_{j+1}].\\
 \end{aligned}\right.    
\end{equation*} 
\begin{remark}\label{rmk:bigremarkonL'L''}
We have $L'_{P,r}=L''_{P,r}$ for $r\leq  c_n$. We also have $L'_{P,r}=L''_{P,r}$ in the special case where all vectors in $\mathcal{L}_{\mathrm{brd}}$ never decay. Note that $\mathcal{L}_{\mathrm{brd}}$ is saturated in $\mathcal{L}$ with rank $2(m-n+1)\leq 2(m-a)$. The equality holds only in the thin case.
\end{remark} 
\begin{remark} The chain of lattices  $L''_{P,r}\otimes\bZ_p$ decay in a pattern similar to Example~\ref{exp:Rondo}. 
\end{remark} 

\subsubsection{Auxiliary lattices: the loose case} Assume that $a_k>a$. We will again construct a sequence of auxiliary lattices $L_{P,r}''$. First, define  $\mathcal{L}_{\mathrm{brd}}$, $\mathcal{L}_{> i},\mathcal{L}'_{\leq i}$ and $H_j$ in the same way as \S\ref{subsub:auxlatticetight}. For $r\leq H_1=h_P$, we define $L''_{P,r}$ the same way as \S\ref{subsub:auxlatticetight}. We won't need auxiliary lattices when $r>h_{a_k}$. 

\subsubsection{Local density computation}\label{subsub:ldc} Below we state a local density formula, using the following notation. Let \(\mathsf{Q}\) be the set of isomorphism classes of self-dual quadratic forms over \(\mathbb{F}_p\). Define \(\chi : \mathsf{Q} \to \{-1,0,1\}\) by setting \(\chi = 0\) for odd-dimensional forms, \(1\) (resp. $-1$) for even-dimensional split (resp. nonsplit) forms. For \(\alpha \in \mathbb{F}_p^\times\), let \(\mathbb{F}_p(\alpha)\) denote \(\mathbb{F}_p\) equipped with the quadratic form \(q(x) = \alpha x^2\). For \(W \in \mathsf{Q}\), write \(W(\alpha) = W \otimes \mathbb{F}_p(\alpha)\).

\begin{lemma}\label{lm:computedelta}
    Let \(\Lambda\) be a self-dual quadratic lattice over \(\mathbb{Z}_p\) and \(M\) a positive integer coprime to \(p\). Write $\overline{\Lambda}$ for $\Lambda \tensor \bF_p$. Then
    \begin{equation}
        \delta(p, \Lambda, M) = 
        \bigl(1 - \chi(\overline{\Lambda})\, p^{-\mathrm{rk\,}(\Lambda)/2}\bigr)
        \bigl(1 + \chi\bigl(\mathbb{F}_p(M) \oplus \overline{\Lambda}(-1)\bigr)\, p^{(1 - \mathrm{rk\,}(\Lambda))/2}\bigr).
    \end{equation}
\end{lemma}
\begin{proof}
    Use the fact that $\delta(p,\Lambda,M)=p^{1-\mathrm{rk} \overline{\Lambda}} \#\{x\in \overline{\Lambda}|\overline{Q}(x)= M\}$ and apply \cite[Thm.~1.3.2]{Ki93}.
\end{proof}

Now we are ready to compute the local density of $L''_{P,r}\otimes\bZ_p$ (in both tight and loose cases). 
\begin{definition}\label{def:localdensitydelta}
Let $M$ be coprime to $p$. For $0\leq i\leq n$, define $\delta_i=\delta(p, \cL_{\mathrm{brd}}+\cL_{>i}, M)$ and $\delta_i'=\delta(p, \cL_{\mathrm{brd}}+\cL_{\leq i}', M)$.  
\end{definition}
\begin{remark}\label{rmk:localdensityofL''_r!} It is clear from the definition that $\delta(p,L''_{P,r}\otimes\bZ_p,M)$ equals $ \delta_0$ when $r\in (0,c_1]$, equals $\delta_i$ when $r\in(c_i,c_{i+1}]$ and $i\in [1,n-1]$, equals $\delta_n'$ when $r\in(c_n,d_{n}]$, equals $\delta_{j}'$ when $r\in (d_{j+1},d_{j}]$ and $j\in [1,n-1]$, equals $\delta'_0$ when $r\in (d_1,h_{a_k}]$. In the tight case, we furthermore have $\delta(p,L''_{P,r}\otimes\bZ_p,M)=\delta'_0$ for all  $r\in (d_1,\infty)$. 
\end{remark}
\begin{corollary}\label{cor:localdensity}~\begin{enumerate}
    \item In borderline case 2,  $\delta_i= 1-p^{i-m}$ for $i\in [0,n-1]$ and $\delta_j'=1+p^{n-m-1}$ for $j\in [0,n]$.
    \item In borderline case 1, $\delta_i= 1-p^{i-m}$  for $i\in [0,n-1]$ and $\delta_j'\leq 1+ p^{n-m}$ for $j\in [0,n]$.
    \item In borderline case 0, $\delta_n=1-p^{i-m}$ for $i\in [0, n-1]$ and $\delta_{j}'= 1-p^{n-m}$ for $j\in [0,n]$.
\end{enumerate}
\end{corollary}
\begin{proof}
    Direct computation using Lemma~\ref{lm:computedelta}.
\end{proof}

\subsubsection{The mysterious quantity $D_0$} \label{subsub: D_0}
Define the following quantity  
$$D_0=d_n+ p^{n-1}c_1+\sum_{i=1}^{n-1} p^{n-i-1}(c_{i+1}-c_{i}).$$
Later on, this quantity will appear mysteriously in various computations, but we don't have a very good geometric explanation of it.  
\begin{lemma}\label{lm:D_0}
  In thin case, we have $D_0\leq h_{a_k}$. In thick case, we have $$D_0+(p-1)(c_n-c_1)+(p-1)d_n\leq p^{n-(a_k+1)}h_{a_k}.$$
\end{lemma}
\begin{proof}
Recall that $h_{a_k}\geq p^{a_k}c_1+d_1$ by (\ref{eq:hv2}). We can use (\ref{eq:thinv2}) or the first inequality of $(\ref{eq:fatv2})$ to get 
$$d_1= d_n+\sum_{i=1}^{n-1}(d_i-d_{i+1})\geq d_n+\sum_{i=1}^{n-1}p^{n-1-i}(c_{i+1}-c_i).$$
Using (\ref{eq:hv2}), we get 
$$h_{a_k}\geq  d_n+p^{a_k}c_1+\sum_{i=1}^{n-1} p^{n-i-1}(c_{i+1}-c_{i})$$
In the thin case, we have $a_k=n-1$. So $h\geq D_0$. 
In the thick case, we have 
\begin{alignat*}{3}
    p^{n-(a_k+1)}h_{a_k}&\geq  p^{n-(a_k+1)}d_n+p^{n-1}c_1+p^{n-(a_k+1)}\sum_{i=1}^{n-1} p^{n-i-1}(c_{i+1}-c_{i})\\
    &\geq (p-1)d_n+d_n+p^{n-1}c_1+\sum_{i=1}^{n-1} p^{n-i-1}(c_{i+1}-c_{i}) + (p-1)\sum_{i=1}^{n-1} (c_{i+1}-c_{i})\\
   & \geq (p-1)d_n+(p-1)(c_n-c_1)+D_0.
\end{alignat*}
\end{proof}

\subsection{Estimating $S'(M)$: the tight case}\label{sub:hardsitu} In this section we assume $a_k=a$. We will compare $S'(M)$ to $S''(M)$ and explicitly compute $S''(M)$, and use this to prove Theorem~\ref{lm:localglobalEisensteinSSP} in the tight case. Following Remark~\ref{rmk:explicit q_L'}, we can write $S'(M)$ as the following series (we will use $h_a$ instead of $h_P$, the two quantities are equal; cf. Notation~\ref{not: define h_i and a_i}):
$$S'(M)=\frac{p^{a+1}-1}{h_a}\frac{1}{p(1+p^{-(m+1)})}\sum_{r=1}^\infty\frac{\delta(p,L'_{P,r},M)}{|L'_{P,1}\otimes\bZ_p:L'_{P,r}\otimes\bZ_p|}.$$
We then define 
$$S''(M):=\frac{p^{a+1}-1}{h_a}\frac{1}{p(1+p^{-(m+1)})}\sum_{r=1}^\infty\frac{\delta(p,L''_{P,r},M)}{[L''_{P,1}\otimes\bZ_p:L''_{P,r}\otimes\bZ_p]},$$
where the auxiliary lattices $L''_{P,r}$ are as in \S\ref{subsub:auxlatticetight}. In the following we prove a partial generalization of Lemma~\ref{lm:isometric embedding in the m=2 case} that compares $S'(M)$ to $S''(M)$:
\begin{lemma}\label{lm:isometric embedding in the tight case partial 1}
Notation as above.  For all $M$ coprime to $p$ we have $S'(M)\leq S''(M)$. 
\end{lemma}
    \begin{proof}
The strategy is the same as Lemma~\ref{lm:isometric embedding in the m=2 case}, and we use Proposition~\ref{prop:L+thin} as the main input. In the following, we show that for all $r\geq 1$,
\begin{equation}\label{eq:localdensitycomparison2}
\frac{\delta(p,L'_{P,r},M)}{[L'_{P,1}\otimes\bZ_p:L'_{P,r}\otimes\bZ_p]}\leq \frac{\delta(p,L''_{P,r},M)}{[L''_{P,1}\otimes\bZ_p:L''_{P,r}\otimes\bZ_p]}.   
\end{equation}
By Remark~\ref{rmk:bigremarkonL'L''}, it suffices to treat the case where $r>c_n$. By the Decay Lemma \ref{thm: superspecial rapid decay} (also see Remark \ref{rem: ak is a}), the \textit{relevant lattice} $$\cL_{\mathrm{rv}}:=\mathrm{Span}_{\mathbb{F}_p}\{{e'},{f'},{e_j},{f_i}\}_{1\leq i \leq n,1\leq j\leq n-\dim \overline{\cL}_{\mathrm{bd}}}$$ decays rapidly to arbitrary order. So by construction we always have \begin{equation}\label{eq:Lgeq L''2}
[L'_{P,1}\otimes\bZ_p:L'_{P,r}\otimes\bZ_p]\geq [L''_{P,1}\otimes\bZ_p:L''_{P,r}\otimes\bZ_p].\end{equation} Suppose that (\ref{eq:Lgeq L''2}) is a strict inequality. If $\delta(p,L''_{P,r},M)\neq 0$, then (\ref{eq:localdensitycomparison2}) holds and is a strict inequality; if $\delta(p,L''_{P,r},M)= 0$ (this happens only when $\mathcal{L}_{\mathrm{brd}}=\mathbf{0}$), then the isometrical embedding results in  Proposition~\ref{prop:L+thin} imply that 
$\delta(p,L'_{P,r},M)=0$, hence   (\ref{eq:localdensitycomparison2}) holds and is an equality. Suppose that (\ref{eq:Lgeq L''2}) is an equality, then the isometrical embedding in Proposition~\ref{prop:L+thin} is an isomorphism. It follows that $\delta(p,L'_{P,r},M)=\delta(p,L''_{P,r},M)$. So (\ref{eq:localdensitycomparison2}) is an equality. As a result,  (\ref{eq:localdensitycomparison2}) always holds, and we have $S'(M)\leq S''(M)$.  
\end{proof}
In the rest of this section we conduct a case by case study. Even though $a_k=a$, we will still keep the notation $a_k$. This is because some computations hold even when $a_k>a$, and will be reused in the loose situation. To ease the computation of $S''(M)$, we decompose it to two parts: 
\begin{equation}\label{eq:decompositionofS''(M)}
    S''(M)=S''_1(M)+ S''_{>1}(M),
\end{equation}
where $S''_1(M)$ is the sub-sum from $r=1$ to $h_{a_k}$, while $S''_{>1}(M)$ is the sub-sum from $r=h_{a_k}+1$ to $\infty$.
\subsubsection{Borderline case 2}\label{subsub;Tinfat2}
The case where $\dim\overline{\cL}_{1,\mathrm{bd}}=2$ is harder than others, since there are potentially more special endomorphisms that resist decay. We will treat this case first. Once this is done, the remaining cases are nothing but easier variants. 

First we have $c_{n-1}=c_n=d_{n}=d_{n-1}$. We then define 
\begin{align*}
    &D_\delta:=c_1\delta_0+\sum_{i=1}^{n-1}(c_{i+1}-c_i){p}^{-i}
\delta_{i}+\sum_{i=1}^{n-1}  (d_i-d_{i+1})p^{2+i-2n}\delta_{i}'+ (h_{a_k}-d_1)p^{2-2n}\delta_{0}',\\
&D:=c_1+\sum_{i=1}^{n-1}(c_{i+1}-c_i){p}^{-i}+\sum_{i=1}^{n-1}  (d_{i}-d_{i+1})p^{2+i-2n}+ (h_{a_k}-d_1)p^{2-2n}.
\end{align*}
Here $\delta_i,\delta_i'$ are from Definition~\ref{def:localdensitydelta}. Then it is easy to see from Remark~\ref{rmk:localdensityofL''_r!} that 
\begin{equation}\label{eq:S1''S>1''border2}
    \begin{aligned}
&S_1''(M)= \frac{p^{a+1}-1}{h_{a}}\frac{ D_\delta}{p(1+p^{-{(m+1)}})}.\\
&S_{>1}''(M)=\frac{p^{a+1}-1}{h_{a}}\frac{\delta_0'D}{p(1+p^{-{(m+1)}})} \sum_{j=1}^\infty\frac{p^{(a_k+1)j}}{p^{2nj}}.
\end{aligned}
\end{equation}

\begin{lemma}\label{lm;Destimate}
We have $D\leq p^{-a_k}h_{a_k}$ and $D_\delta\leq p^{-a_k}h_{a_k}$.   
\end{lemma}
\begin{proof} 
We first treat the thin case. Using the computation in Corollary~\ref{cor:localdensity}.
\begin{alignat}{5}
\label{eq:thinD0}D&\stackrel{(\ref{eq:thinv2})}{=}c_1+c_1p^{2-2n+a_k}+\sum_{i=1}^{n-1}(c_{i+1}-c_i){p}^{-i}+p^{1-n}\sum_{i=1}^{n-1}(c_{i+1}-c_i)\\
 \label{eq:thinD1}  & = c_1+c_1p^{2-2n+a_k}+ p^{1-n}(c_n-c_1)+p^{1-n}\sum_{i=1}^{n-1}(c_{i+1}-c_i){p}^{n-i-1}\\
  \label{eq:thinD2}   & =p^{1-n}(c_1(p^{a_k+1-n}-1)+c_n+ p^{n-1}c_1+\sum_{i=1}^{n-1}(c_{i+1}-c_i){p}^{n-i-1})\\
\label{eq:thinD5}   &\stackrel{(c_n=d_n)}{=}p^{1-n}(c_1(p^{a_k+1-n}-1)+D_0)\\
 \label{eq:thinD6}  &\stackrel{(n=a_k+1)}{=} p^{1-n}D_0\stackrel{\text{Lemma~\ref{lm:D_0}}}{\leq}p^{-a_k}h_{a_k}.
\end{alignat}
\begin{alignat*}{4}
   D_{\delta}&\stackrel{(\ref{eq:thinv2})}{=} c_1(1-p^{-m})+c_1p^{2-2n+a_k}(1+p^{n-m-1})+ \sum_{i=1}^{n-1}(c_{i+1}-c_i)(p^{-i}+ p^{1-n})\\
   &=  c_1(1-p^{-m})+c_1p^{2-2n+a_k}(1+p^{n-m-1})+ p^{1-n}(c_n-c_1)+\sum_{i=1}^{n-1}p^{-i}(c_{i+1}-c_i)\\
   &=c_1(p^{a_k+1-n-m}-p^{-m}+p^{2-2n+a_k}-p^{1-n})+c_1+ p^{1-n}c_n+\sum_{i=1}^{n-1}p^{-i}(c_{i+1}-c_i)\\
   &\stackrel{(n=a_k+1)}{=}p^{1-n}D_0\stackrel{\text{Lemma~\ref{lm:D_0}}}{\leq} p^{-a_k}h_{a_k}.
\end{alignat*} 
Now we treat the thick case: the computation can be reduced to the thin case. To begin with, we have 
\begin{alignat}{2}
\label{2122}D&\stackrel{(\ref{eq:fatv2})}{\leq}c_1+c_1p^{2-2n+a_k}+\sum_{i=1}^{n-1}(c_{i+1}-c_i){p}^{-i}+p^{2-n}\sum_{i=1}^{n-1}(c_{i+1}-c_i)\\\label{2123}&= c_1+c_1p^{2-2n+a_k}+\sum_{i=1}^{n-1}(c_{i+1}-c_i){p}^{-i}+p^{1-n}\sum_{i=1}^{n-1}(c_{i+1}-c_i) + p^{1-n}(p-1)\sum_{i=1}^{n-1}(c_{i+1}-c_i).
\end{alignat}
As compared to (\ref{eq:thinD0}), we have $p^{2-n}$ rather than $p^{1-n}$ in the last summation of (\ref{2122}). But we can put it into the form as in (\ref{2123}), leaving the computation of the first half of (\ref{2123}) identical to (\ref{eq:thinD1})$\sim$(\ref{eq:thinD5}). As a result, we get

\begin{alignat}{4}
D&\leq 
   p^{1-n}(c_1(p^{a_k+1-n}-1)+D_0)+ p^{1-n}(p-1)(c_{n}-c_1)\\
   & \leq p^{1-n}(c_1(p^{a_k+1-n}-1)+(p-1)(c_{n}-c_1)+D_0)\\
   &\stackrel{(n>a_k+1)}{<} p^{1-n}((p-1)(c_{n}-c_1)+D_0)\\
   &\stackrel{\text{Lemma~\ref{lm:D_0}}}{\leq} p^{-a_k}h_{a_k}.
\end{alignat}
Similarly strategy works for $D_\delta$ in the thick case. We get 
\begin{alignat*}{5}
   D_{\delta}&= c_1(p^{a_k+1-n-m}-p^{-m}+p^{2-2n+a_k}-p^{1-n})+p^{1-n}\left[D_0 +  (p-1)(1+p^{n-m-1})(c_n-c_1)\right]\\
   &\stackrel{(n>a_k+1)}{<}p^{1-n}\left[D_0 +  (p-1)(1+p^{n-m-1})(c_n-c_1)\right]\\
   &\stackrel{(n\leq m)}{<}p^{1-n}\left[D_0 +  2(p-1)(c_n-c_1)\right]\stackrel{\text{Lemma~\ref{lm:D_0}}}{\leq} p^{-a_k}h_{a_k}.
\end{alignat*}
\end{proof}
\begin{remark}\label{rmk:hardtoeasy} 
Lemma~\ref{lm;Destimate} still holds when $a<a_k$, and will also be used in the easy situation.
\end{remark}
Now plug lemma~\ref{lm;Destimate} into (\ref{eq:S1''S>1''border2}), and equate $a_k=a$. We obtain 
\begin{align}
\label{eq1case22final}S''(M)=S_1''(M)+S_{>1}''(M)&\leq  \frac{1-p^{-a-1}}{1+p^{-{(m+1)}}}\left(1+ \frac{1+p^{n-m-1}}{p^{2n-a-1}-1}\right)\\
&\label{eq2case22final}=1-\frac{ (p^n - p^{a+1}) (p^{a + m + 1} + p^{a + n} + p^a + p^{m + n})}{p^a(p^{m + 1} + 1) (p^{2 n} - p^{a + 1})}.
\end{align}
In the thick case, (\ref{eq2case22final}) is strictly smaller than $1$. By Lemma~\ref{lm:isometric embedding in the tight case partial 1}, we get $S'(M)\leq S''(M)<1$. In the thin case,  (\ref{eq2case22final}) equals $1$. Therefore we need the following extra input (which is  a partial generalization of the second half of Lemma~\ref{lm:isometric embedding in the m=2 case}):
\begin{lemma}
Suppose that we are in the tight thin borderline case 2, and that $\cA_{C^{/P}}$ does not admit a rank $2(m-a)$ saturated lattice of formal special endomorphisms.  Then $S'(M)< S''(M)$. 
\end{lemma}
\begin{proof}
    The proof is similar to that of Lemma~\ref{lm:isometric embedding in the m=2 case}. In the tight thin borderline case 2, $\delta(p,L''_{P,r},M)\neq 0$. So if (\ref{eq:Lgeq L''2}) is a strict inequality, then so is (\ref{eq:localdensitycomparison2}). Therefore $S'(M)= S''(M)$ implies that (\ref{eq:Lgeq L''2}) is an equality for all $r$. It then follows from Proposition~\ref{prop:L+thin} that for $r>h_{a_k}$, that the projection of $\overline{L'_{P,r}\otimes \bZ_p}\subseteq \overline{\mathcal{L}}_{\mathrm{brd}}\oplus \mathrm{Span}_{\mathbb{F}_p}\{\overline{e'},\overline{f'},\overline{f_1},\overline{f_2},...,\overline{f_n}\}$ to $ \overline{\mathcal{L}}_{\mathrm{brd}} $ is an isometric isomorphism. In particular, $\overline{L'_{P,r}\otimes \bZ_p}$ does not depend on $r$. We then pick $2(m-a)$ vectors in $\cL$ whose reductions form a basis of $\overline{L'_{P,r}\otimes \bZ_p}$. A limit argument as in Lemma~\ref{lm:isometric embedding in the m=2 case} yields a rank $2(m-a)$ saturated lattice that is contained every $\overline{L'_{P,r}\otimes \bZ_p}$. Therefore $\cA_{C^{/P}}$ admits a rank $2(m-a)$ saturated lattice of formal special endomorphisms. This contradiction shows that we must have  $S'(M)< S''(M)$. 
\end{proof}
The following consequence is immediate: 
\begin{corollary}
    \label{cor:main in tight borderline case 2}
Theorem~\ref{lm:localglobalEisensteinSSP} holds in the tight borderline case 2.

\end{corollary}

\subsubsection{Borderline case 1} 
Suppose that $\dim\overline{\cL}_{1,\mathrm{bd}}=1$. We then have $c_n=d_{n}$. Define 
\begin{align*}
    &D_\delta=c_1\delta_0+\sum_{i=1}^{n-1}(c_{i+1}-c_i){p}^{-i}
\delta_{i}+\sum_{i=1}^{n-1}  (d_i-d_{i+1})p^{1+i-2n}\delta_{i}'+ (h_{a_k}-d_1)p^{1-2n}\delta_{0}',\\
  & D=c_1+\sum_{i=1}^{n-1}(c_{i+1}-c_i)p^{-i}+ \sum_{i=1}^{n-1} (d_i-d_{i+1})p^{1+i-2n}+ (h_{a_k}-d_1)p^{1-2n}. 
\end{align*}
Then $ S_1''(M)$ and $ S_{>1}''(M)$ can still be computed by the formula (\ref{eq:S1''S>1''border2}) (with the meaning of $D_\delta$ and $D$ replaced by the new definition given above). 
\begin{lemma}\label{lm;Destimate1}
$D< p^{-a_k}h_{a_k}$ and $D_{\delta}< p^{-a_k}h_{a_k}$.
\end{lemma}
\begin{proof}
    The computation is completely similar to the proof of Lemma~\ref{lm;Destimate}. Of course one can directly do the dirty computation, but we can also use the computation already done in Lemma~\ref{lm;Destimate} to relive the burden. In fact, in the computation of Lemma~\ref{lm;Destimate}, we only treated $c_i$ and $d_i$ as symbols and didn't use specific properties of them, except in (\ref{eq:thinD5}) we have used the fact $c_n=d_n$ (which is also true in this case). Let's now do the thin case: note that 
\begin{align*}
    D < c_1+\sum_{i=1}^{n-1}(c_{i+1}-c_i){p}^{-i}+\sum_{i=1}^{n-1}  (d_{i}-d_{i+1})p^{2+i-2n}+ (h_{a_k}-d_1)p^{2-2n}. 
\end{align*}
  But according to the computation (\ref{eq:thinD0})$\sim$(\ref{eq:thinD6}), the RHS equals $p^{1-n}D_0$. It follows immediately from Lemma~\ref{lm:D_0} that $D<p^{-a_k}h_{a_k}$. Similarly, 
$$ D_\delta=c_1\delta_0+\sum_{i=1}^{n-1}(c_{i+1}-c_i){p}^{-i}\delta_i
+\sum_{i=1}^{n-1}  (d_i-d_{i+1})p^{2+i-2n}(p^{-1}+p^{n-m-1})+ (h_{a_k}-d_1)p^{2-2n}(p^{-1}+p^{n-m-1}),$$
and one verifies that the RHS is smaller than the $D_{\delta}$ computed in  Lemma~\ref{lm;Destimate}. So we immediately get $D_\delta<p^{-a_k}h_{a_k}$ as well. 
The thick case can be treated similarly.
\end{proof}
Use Lemma~\ref{lm;Destimate1} and equate $a_k=a$. We have 
\begin{align*}
S''(M)&= S_1''(M)+S_{>1}''(M)
\stackrel{(n\geq a+1)}{<} \frac{1-p^{-a-1}}{1+p^{-{(m+1)}}}\left(1+ \frac{2}{p^{a+2}-1}\right) <1.
\end{align*}
\begin{corollary}
    \label{cor:main in tight borderline case 1}
Theorem~\ref{lm:localglobalEisensteinSSP} holds in the tight borderline case 1.

\end{corollary}

\subsubsection{Borderline case 0}
Define \begin{align*}
    & D_\delta=c_1\delta_0+\sum_{i=1}^{n-1}(c_{i+1}-c_i){p}^{-i}
\delta_{i}+(d_n-c_n)p^{-n}\delta_n+\sum_{i=1}^{n-1}  (d_i-d_{i+1})p^{i-2n}\delta_{i}'+ (h_{a_k}-d_1)p^{-2n}\delta_{0}',\\
& D=c_1+\sum_{i=1}^{n-1}(c_{i+1}-c_i)p^{-i}+ (d_n-c_n)p^{-n}+\sum_{i=1}^{n-1} (d_i-d_{i+1})p^{i-2n}+ (h_{a_k}-d_1)p^{-2n}. 
\end{align*}
Then $ S_1''(M)$ and $ S_{>1}''(M)$ can still be computed by the formula (\ref{eq:S1''S>1''border2}) (with the meaning of $D_\delta$ and $D$ replaced by the new definition given above). 
\begin{lemma}\label{lm;Destimate2}
$D< p^{-a_k}h_{a_k}$ and $D_{\delta}< p^{-a_k}h_{a_k}$.
\end{lemma}
\begin{proof}
   This can be dealt with the strategy explained in Lemma~\ref{lm;Destimate1}, i.e., reduce to the computation already done in Lemma~\ref{lm;Destimate}. For the thin case, we have 
\begin{align*}
    D < c_1+\sum_{i=1}^{n-1}(c_{i+1}-c_i){p}^{-i}+\sum_{i=1}^{n-1}  (d_{i}-d_{i+1})p^{2+i-2n}+ (h_{a_k}-d_1)p^{2-2n} + (d_n-c_n)p^{-n}.
\end{align*}
According to the computation (\ref{eq:thinD0})$\sim$(\ref{eq:thinD2}) (note that we cannot use (\ref{eq:thinD5}) since $c_n$ may not be equal to $d_n$ in this case), the RHS equals
  \begin{alignat*}{2}
      &p^{1-n}(c_n+ p^{n-1}c_1+\sum_{i=1}^{n-1}(c_{i+1}-c_i){p}^{n-i-1})+ (d_n-c_n)p^{-n}\\
      &< p^{1-n}(d_n+ p^{n-1}c_1+\sum_{i=1}^{n-1}(c_{i+1}-c_i){p}^{n-i-1})=p^{1-n} D_0\stackrel{\text{Lemma~\ref{lm:D_0}}}{\leq} p^{1-n}h_{a_k}.
  \end{alignat*}  
The computation for $D_\delta$, as well as the thick case can be dealt in the same way. 
\end{proof}
Use Lemma~\ref{lm;Destimate2} and equate $a_k=a$, we find that 
\begin{align*}
S''(M)=S_1''(M)+S_{>1}''(M)&
\stackrel{(n\geq a+1)}{<}\frac{1-p^{-a-1}}{1+p^{-(m+1)}}\left(1+\frac{2}{p^{a+3}-1}\right)<1.
\end{align*}
\begin{corollary}
    \label{cor:main in tight borderline case 0}
Theorem~\ref{lm:localglobalEisensteinSSP} holds in the tight borderline case 0.
\end{corollary}
\subsection{Estimating $S'(M)$: the loose case}\label{sub:easysitu} Suppose that $a_k>a$. As compared to the tight situation, this is much easier, and we can treat all borderline cases all at once. 

We will make the decomposition $ S'(M)=S'_1(M)+ S'_{>1}(M)$ similar to (\ref{eq:decompositionofS''(M)}), where $S'_1(M)$ is the sub-sum of $S'(M)$ from $r=1$ to $h_{a_k}$, while $S'_{>1}(M)$ is the sub-sum from $r=h_{a_k}+1$ to $\infty$. On the other hand, since $L''_{P,r}$ is only defined for $r\leq h_{a_k}$, we can only make sense of $S''_1(M)$. 

\begin{lemma}\label{lm:isometric embedding in the loose case partial 1}
Notation as above.  For all $M$ coprime to $p$ we have $S'_1(M)\leq S''_1(M)$. 
\end{lemma} 
\begin{proof}
    The proof is similar to Lemma~\ref{lm:isometric embedding in the tight case partial 1}. 
\end{proof}
Define the quantity $D_\delta$ exactly the same way as in the tight situation. As noted in Remark~\ref{rmk:hardtoeasy}, we still have the bound $D_\delta\leq p^{-a_{k}}h_{a_k}$ (in all thin and tight and all borderline cases). Note that $h_a\geq h_{a_k}$. The same computation as in \S\ref{sub:hardsitu} shows that in all cases we have 
$$S'_1(M)\leq S_1''(M)\leq \frac{p^{a+1}-1}{h_a}\frac{D_\delta}{p(1+p^{-{(m+1)}})}  \leq \frac{p^{a-a_k}-p^{-a_k-1}}{1+p^{-{(m+1)}}}\frac{h_{a_k}}{h_a}<\frac{p^{-1}}{1+p^{-{(m+1)}}}.$$
On the other hand, it is easy to estimate $S_{>1}'(M)$: one can just use the coarsest bounds. By Theorem~\ref{thm: superspecial rapid decay}, we have a rank $a+a_k+2$ saturated sublattice of rapid decay. Each primitive vector in this sublattice can at most lift to $h_{a_k}$. Upon using the coarsest upper bound for the local density (which is $2$), we find that  
$$S_{> 1}'(M)\leq \frac{p^{a+1}-1}{h_a}\frac{2}{p(1+p^{-(m+1)})} \sum_{j=1}^\infty\frac{p^{(a_k+1)j}h_{a_k}}{p^{(a+a_k+2)j}}\leq \frac{2p^{-1}}{1+p^{-(m+1)}}.$$
As a result, we have 
$$S'(M)=S_1'(M)+S_{>1}'(M) < \frac{3p^{-1}}{1+p^{-(m+1)}} \stackrel{(p\geq 3)}{<} 1. $$
\begin{proof}[Proof of Theorem~\ref{lm:localglobalEisensteinSSP}]    
The above paragraph treats the loose case. The tight case is established in Corollaries~\ref{cor:main in tight borderline case 2}, \ref{cor:main in tight borderline case 1} and \ref{cor:main in tight borderline case 0}. 
\end{proof}
\section{Decay of special endomorphisms: supersingular but not superspecial}\label{sec: supersingular}
We closely follow the exposition and notation in \cite[Sections 4 and 5]{MST}. The computations are similar to the ones in \emph{loc. cit.} and the ones carried out in the superspecial case above, and so we will be content with only providing an outline of the calculations.

We work at a supersingular point $P$. Set $n = \frac{t_P}{2}$ and $m = (b-2n)/2$. To ease exposition, we set $n' = n-1$. Note that the case $n' =0$ is precisely the superspecial case, and so we have that $n' >0$. As recalled in \cref{subsec: loc supersingular}, we have a decomposition of $W$-lattices $\bL = \bL_0 \oplus \bL_1$ and hence a decomposition of their Tate modules $\cL = \cL_0 \oplus \cL_1$. Recall that $\cL_1$ is self-dual over $\bZ_p$, so $\mathrm{rank}_W \bL_0 = \mathrm{rank}_{\bZ_p} \cL_0 = 2n$ and $\mathrm{rank}_W \bL_1 = \mathrm{rank}_{\bZ_p} \cL_1 = 2m$. In order to keep notation consistent within our paper, we diverge from the notation in \cite{MST} and make the following changes.

\begin{table}[H]
\centering
\caption{Changes in notation between current paper and \cite{MST}}
\label{tab:notation-changes}
\renewcommand{\arraystretch}{1.3}
\begin{tabular}{|c|c|c|}
\hline
\textbf{Object} & \textbf{Notation in \cite{MST}} & \textbf{Notation in current paper} \\
\hline
Basis for $\mathcal{L}_0$ & $\{e_i, f_i : 1 \le i \le n \}$ & $\{e'_i, f'_i : 1 \le i \le n \}$\\
\hline
Basis for $\mathcal{L}_1$ & $\{e'_j, f'_j : 1 \le j   \le m \}$ & $\{e_j, f_j : 1 \le j \le m \}$\\
\hline
Basis for $\bL_0$ & $\{v_i, w_i : 1 \le i \le n \}$ & $\{v'_i, w'_i : 1 \le i \le n \}$ \\
\hline
Coordinates (1st set) & $\{x_i, y_i : 1 \le i \le n' \}$ & $\{x'_i, y'_i : 1 \le i \le n' \}$ \\
\hline
Coordinates (2nd set) & $\{x'_j, y'_j : 1 \le j \le m \}$ & $\{x_j, y_j : 1 \le j \le m \}$ \\
\hline
Function $Q_i, 0 \le i \le n' - 1$ & $y_{i + 1}$ & $y'_{i + 1} = Q_i$ \\
\hline
Function $Q_{n'}$ & $y_n$ & $Q = \sum_{i=1}^{n'} x'_i y'_i + \sum_{i=1}^m x_i y_i = Q_{n'}$ \\
\hline
Function $Q_{n'+a}$ $(a > 0)$ & $y_{n+1}, y_{n+2}, \dots$ & $Q_{n'+a} = \sum_{i=1}^m (x_i y_i^{p^{a}} + x_i^{p^{a}}y_i)$ \\
\hline
\end{tabular}
\end{table}

Note that $\{e_j, f_j : 1 \le j \le m \}$ is a basis for $\bL_1$ as well, and the basis $\{v'_i, w'_i : 1 \le i \le n \}$ for $\bL_0$ is provided by \cite[Lem.~4.5]{MST}. Again let $R := W[\![x'_i, y'_i, x_j, y_j]\!]_{1 \le i \le n', 1 \le j \le m}$ be the structure ring of $\shS^{/P}_W$. The Frobenius operator $\Frob$ on $\bL_{\cris}(R) \simeq \bL \tensor R$ is given by $(uB) \circ \sigma$ with respect to $\{ v'_i, w'_i, e_j, f_j \}$, where $u$ is a matrix with entries in $R$ and $B$ is a matrix with entries in $\frac{1}{p} W \subset w[1/p]$. The general formulas for $u$ and $B$ for a supersingular $P$ that extends \cref{eqn: u and B} in \S\ref{subsub:explicitcoordinatheonL} for the superspecial case is spelled out in \cite[\S4.8]{MST}.

Combining these formulae with the argument of \cref{NPstrata}, we find that the local equations of the newton strata through $P$ are defined by vanishing of the $Q_i$'s. 
\begin{proposition}[Ogus]
    \begin{enumerate}[label=\upshape{(\alph*)}]
        \item For $0\leq i < m+n' -1$, the Newton stratum of height $i+2$ is defined by the vanishing of the equations $Q_0, \hdots Q_i$. 

        \item The supersingular stratum is defined by the vanishing of the equations $Q_0, Q_1, \hdots Q_{n'+m-1}$. 

        \item For $i \geq n'+m$, we have $Q_i $ is in the ideal generated by $Q_0\hdots Q_{n'+m-1}$.
    \end{enumerate}
\end{proposition}

We will recall notation for future use. Let $A$ (resp. $B$,$D$) denote the top left $2n \times 2n$ (resp. top right $2n \times 2m$, bottom left $2m \times 2n$) block of $u' - I$ where $u'$ is as in \cite[Section 4.11]{MST}. As in \cite[Section 4.6]{MST}, we let $S_0$ be the change-of-coordinates matrix from $\{e'_i,f'_i \}$ to $\{ v'_i,w'_i\}$, i.e. the first (resp. last) $n$ columns of $S_0$ are the coordinates of the $v'_i$ (resp. $w'_i$) expressed in terms of the $e'_i,f'_i$. As in \emph{loc. cit.}, there is a unique matrix $S'_0 \in \textrm{GL}_{2n}(W)$ such that\[
S_0 = S'_0 
\left[
\begin{array}{c|c}
p^{-1} I & 0 \\ \hline
0 & I
\end{array}
\right].
\] 

\subsection{Setup with local curve}\label{sec: local setup supersingular}
We will assume that $Q_0(t), \hdots Q_{a-1}(t) = 0$ and that $Q_a(t)$ is the first non-zero function of $t$. We refer to \cite[Section 5.7]{MST} for the definition of the matrices $K_i$. The matrices $K_i$ are defined for $i = 1, \hdots n+1$ in \emph{loc. cit.}. We define $K_i$ for every $i$. Indeed, set $A_{n+i} = CD^{(i)}$. This is the $2n\times 2n$ matrix which is zero everywhere except for the $(n,2n)$-entry, and this entry is $p^{-1}Q_{n'+i}$. Define $K_{n+i} = S'_0 A_{n+i} ({S'}_0^{-1})^{(i)}$. We note that \cite[Lemma 5.9 (2)]{MST} is true for $K_i$ with values of $i$ greater than $n+1$ as well, and the same proof goes through. The discussion immediately after \emph{loc. cit.} also holds for arbitrary $i$. We set $h_i := v_t(Q_i(t))$. 
\begin{remark}
    The valuations of these functions were denoted by $a_i$ in \cite{MST}. 
\end{remark}
Note that we have $h_i = \infty$ for $i< a$. Let $a_1 = a < a_2 \hdots < a_k$ be the maximal sequence of integers such that $a_k \leq n+m-1$, and $h_{a_1} \geq h_{a_2} \hdots \geq h_{a_k}$. There are three main cases. 

\begin{enumerate}
    \item $a_1<n', a_k \leq n'$. 

    \item $n' \leq  a_1$. 

    \item $a_1 < n' < a_k$. 
\end{enumerate}

Cases 2 and 3 are treated identically to establish rapid decay of a subspace having rank at least $2a+2$, and Cases 1 and 3 are treated identically to prove Theorem \ref{thm: main local version}. We will therefore only analyze Cases 1 and 2, and only briefly comment on Case 3. 
\subsection{Rapid decay}
\subsection*{Case 1}
In this case, we will show that a sufficiently large-rank sublattice of $\cL_0$ decays, and this will be sufficient for our purposes.
\begin{definition}
    We say that a primitive vector $v \in \cL_0$ decays rapidly if $p^r v$ does not lift to an endomorphism modulo $t^{1+ h_a (1 + p^{a+1} + p^{2(a + 1)} + \hdots p^{r(a + 1)})}$. We say that a saturated sublattice of $\cL_0$ decays rapidly if every vector in that sublattice decays rapidly. 
\end{definition}

\begin{theorem}\label{thm: supersingular rapid decay}
    There is a saturated sublattice of $\cL_0$ having rank $2a_1 + 2$ that decays rapidly. 
\end{theorem}

\subsubsection{Proof of rapid decay in Case 1}
The treatment of rapid decay for Case 1 is almost identical to the analysis carried out in \cite[Section 5]{MST}. Indeed, it suffices to solely work with $K_i$ with $a_{1}+1 \leq i \leq a_{k}+1 $. The results in \emph{loc. cit.} through Lemma 5.12 are just general and therefore hold. Lemma 5.13 of \emph{loc. cit.} has the following strengthening. 
\begin{lemma}\label{lem: stronger lemma 5.13 MST}
    Let $\alpha_{a_1} \hdots \alpha_{a_k} \in \mathbb{F}$ be not all zero. Consider the linear combination $\bar{R}  = \sum_{i=a_1}^{a_k} \alpha_i \sigma^{i} (\overline{{R}_{n+1}})$. Then $\dim_{\bF_p} \{ v \in \bF_p^{2n}: \bar{R}v = 0\} \leq a_k - a_1$.
\end{lemma}
The proof is the same. 

Now, the proof of the decay lemma goes through exactly as in Theorem 5.2 and Section 5.14 of \cite{MST}. We note that 5.14 requires $a_k \leq n'$ -- indeed, without this assumption, one must a-priori work with $K_i$ for all values of $i$ bounded above by $n'+m-1$ (and this is precisely what we do in Cases 2 and 3). Therefore, \cite[Theorem 5.15]{MST} is only true in the setting of this case. However, the proof of \cite[Theorem 5.2]{MST} together with Lemma \ref{lem: stronger lemma 5.13 MST} in place of \cite[Lemma 5.13]{MST} actually shows that a saturated sub-lattice of $\cL_0$ having rank at least $2n - (a_k-a_1)$ decays rapidly. We have that $2n - (a_k - a_1) \geq 2a_1+2$ (with equality holding only when $a_k = a_1 = n'$). 

\subsection*{Cases 2 and 3}
We will work in the setting of Case 2. We again note that the argument for rapid decay that we outline holds verbatim for Case 3. 

The setup with words, valuations, etc is the same as in the superspecial case. Specifically, the definitions, lemmas, propositions, and proofs from Section \ref{Sec: word properties decay defn} all hold with no change. 

The results regarding isotropic vectors in $\cL_1$ also hold with little change in the statements and no change in the proofs. Indeed, we have the following results.

\begin{lemma}\label{lemma: nonsuperspecial height valuation}
    The number of distinct $i$ such that $\nu_t(x_i) \geq h_{a_k}$ is at most $m-1 - (a_k-n') = m+n-2-a_k$.
\end{lemma}
The proof is identical to the proof of Lemma \ref{lemma: height valuation}. 

\begin{corollary}\label{cor: nonsuperspecial isotropic subspace first order decay}
    \begin{enumerate}[label=\upshape{(\alph*)}]
        \item Let $\cL'\subset \cL $ be a saturated sub-lattice having rank $m-(a_k-n') = m+n-1-a_k$ such that $\cL' \bmod p \subset \cL_1 \bmod p$, and such that $\cL' \bmod p$ is isotropic. Then there is at least one primitive vector that decays rapidly to first order. In other words, there is a primitive $v\in \cL$ such that $v_t(x_v) < h_{a_k}$.


        \item There exists a saturated sublattice $\cL' \subset \cL_1$ having rank $2(a_k-n')$ such that every primitive vector decays rapidly to first order.
    \end{enumerate}
     
\end{corollary}
The proof is idential to the proof of Corollary \ref{cor: isotropic subspace first order decay}

We are now ready to outline the proof of the decay lemma in this setting. 
\begin{proof}
    The setup and notation is identical, and we have the exact analogues of Claims \ref{claim 1} and \ref{claim 2}. The proof of Claim \ref{claim 1} goes through in this setting. The setting of Claim \ref{claim 2} in the superspecial setting is somewhat degenerate. But the analysis in \cite[Proof of Theorem 5.2]{MST} goes through to show the existence of a saturated sublattice of $\cL_0$ with rank $2n - T_r(a_{h_k})$ that decays rapidly. Indeed, in the notation of \emph{loc. cit.}, $T_r(a_{h_k})$ is just $\mu_{J} - \mu_{J}$, which is the only thing required. 

    All of this proves the existence of a saturated sublattice of $\cL$ having rank $2n + 2(a_k - n') - (a_k-a_1)=  2a_k+2 - (a_k-a_1) \geq 2a_1 + 2$ that decays rapidly to $r$th order. An argument identical to the one used in the superspecial case applies to deduce rapid-decay to all orders. 
\end{proof}

\subsection{Local-global comparison}
\subsubsection{Cases 1 and 3}
The analysis in Cases 1 and 3 to compare the local and global Eisenstein series is identical. Indeed, there is no first-order decay analysis required. We are now ready to prove Theorem \ref{thm: main local version}, which we restate for the reader's convenience. 
\begin{theorem}\label{localglobal comparison: Case 3}
    Let $P\in C(\overline{\bF}_p)$ be a supersingular point. Suppose that we are in the setting of Case 1 or 3, i.e. that $a\leq n-1$. Then, there is a positive real number $\alpha < 1$ such that for every integer $M$ that is relatively prime to $p$ and is represented by $L$, we have $S'(M)=\sum_{r = 1}^{\infty} \frac{q_{L_r'}(M)}{g_P(M)} < \alpha$.
\end{theorem}
\begin{remark}
    The local and global contribution at a supersingular point is effected both by the generic newton stratum of the curve and the Artin invariant of the supersingular point. The $p$-power part of the discriminant of the lattice $\cL$ grows as the Artin invariant grows. On the other hand, we need faster rates of decay when the generic Newton stratum is close to being supersingular. When $a < n'$, the saving we obtain from the Artin invariant is large enough that it suffices to prove that a sufficiently large-rank lattice decays rapidly without requiring a first-order analysis. We do not require an algebraicity input either, as the numerics dictate that the formal curve cannot sit inside a formal exceptional locus. Indeed, being contained inside a formal exceptional locus would require that $C$ admits formal special endomorphisms by a self-dual rank $2m + 2n - (2a-2)$ lattice. However, any such lattice must be contained in $\cL_1$, which has rank $2m$ which is strictly smaller than $2m + 2n - (2a-2)$ by our assumptions on $a$. 
\end{remark}

\begin{proof}
    By the decay lemma in this case, we have that $\sqrt{|(L_r'\otimes \bZ_p)^{\vee}/(L_r'\otimes \bZ_p)|}\geq p^{n}$ if $r\leq h_a$, $\geq p^{n + 2a+2}$ if $h_a< r\leq h_a(p^{a+1}+1)$, $\geq p^{r+4a}$ if $h_a(p^{a+1} + 1) < r \leq h_a(p^{2a+2} + p^{a+1} + 1)$, etc. Using the coarse estimate of \cite[Lemma 7.16]{MST}, we then obtain 
    \begin{align*}
        \frac{h_a}{p^{a+1}-1}\sum_{r = 1}^{\infty} \frac{q_{L_r'}(M)}{g_P(M)} &\leq (\frac{2}{1-p^{(-b-2)/2}}) (\frac{h_a(1+1/p)}{p^n} + \frac{h_ap^{a+1}}{p^{2a+2+n}} + \frac{h_ap^{2a+2}}{p^{4a+4+n}} + \cdots) \\ &= \frac{2h_a}{p^n(1-p^{(-b-2)/2})}(1/p + \frac{1}{1-p^{-(a+1)}}).    
    \end{align*}
    As we are in the non-ordinary locus, we have $a>0$. Substituting the coarse bound $3/2 > \dfrac{1}{1-3^{-2}} \geq \dfrac{1}{1-p^{-(a+1)}} $, we obtain a strict upper bound of $\dfrac{3h_a}{p^n(1-p^{-n})}$. It suffices to prove that
    $$\frac{3h_a}{p^n(1-p^{-n})} \leq \frac{h_a}{p^{a+1} -1} \Leftrightarrow  \frac{3}{p^n-1} \leq \frac{1}{p^{a+1}-1} \Leftrightarrow 3 p^{a+1}-3 \leq p^n-1. $$
    This follows directly from the fact that $a<n'$, or $a+1 < n$.  
\end{proof}

\subsubsection{Case 2}
The first order analysis in the superspecial case goes through verbatim in this setting. Indeed, the case that does not follow immediately from the arguments used in Cases 1 and 3 is when $a_k = a_1 \geq n'$. In this setting, the sublattice $\cL_0$ decays rapidly. The analysis of first-order decay for vectors in $\cL_1$ is identical to the superspecial case. The analysis in Section \ref{sec: first step superspecial} yields the following result. 

\begin{theorem}\label{lm:localglobalEisensteinsupersingular}
 Let $C$ be as in Section \ref{sec: local setup supersingular} and let $P$ be the closed point of $C$. If $\mathcal{A}_{C^{/P}}$ does not admit a saturated lattice of formal special endomorphisms of rank $2(m-a)$, then there exists a constant $\alpha< 1$, depending only on $C$ and $P$, such that for $M\in \bN$ coprime to $p$, we have  $S'(M)=\sum_{r=1}^\infty \frac{{q_{L_{r,P}'}(M)}}{g_P(M)}<\alpha .$
\end{theorem}

\section{The algebraization theorem}\label{sec: algebraization}
The main result of this section is to finish the proof of Theorem \ref{thm: main local version}. We recall the setting. We have a supersingular point $P\in C$, with local parameter $t$. Let $k[\![t]\!]$ denote the complete local ring of $C$ at $P$. We have shown the existence of $\alpha < 1$ such that the conclusion of Theorem \ref{thm: main local version} holds \emph{unless} the module of formal special endomorphisms at $k[\![t]\!]$ has rank $b-2a$. Recall that the Decay Lemma (Theorem \ref{thm: superspecial rapid decay}, Theorem \ref{thm: supersingular rapid decay}) implies that the module of formal special endomorphisms at $k[\![t]\!]$ has rank at most $b-2a$. We will now prove that under the assumptions of Theorem \ref{mainShimura}, the module of formal special endomorphisms at $k[\![t]\!]$ has rank strictly less than $b-2a$. 

\begin{remark}
    For the rest of the paper, $m$ will index special divisors.
\end{remark}

\begin{remark}
    We note that in the inert case (see \ref{trichotomy} for the definition of the inert case), the quantity $b-2a$ is zero if our Newton stratum contains the supersingular locus as a divisor. Indeed, in this setting, we necessarily have $$\sum_{r=1}^{\infty} \frac{q_{L'_{P,r}}(m)}{g_P(m)} = 1. $$
This equality is of course expected, as we know that the reduced locus of the intersection of every special divisors with this Newton stratum is contained in the supersingular locus. 

\end{remark}

We now introduce our setup before stating our result.
\begin{setup}\label{setup:splittingcurve}
Let $U$ be a connected smooth affine curve over $\mathbb{F}$ with a nonconstant morphism $\iota:U\rightarrow \mathscr{S}_{a+1,\mathbb{F}}\setminus \mathscr{S}_{a+2,\mathbb{F}}$, i.e., the image of $U$ lies in the height $a+1$ but not height $a+2$ Newton stratum. Let $\iota_0: U_0\rightarrow \mathscr{S}_{a+1,\mathbb{F}_q}\setminus\mathscr{S}_{a+2,\mathbb{F}_q}$ be a  model of $\iota$ over a sufficiently large finite field, i.e., a map defined 
over $\mathbb{F}_q$ whose base change to $\mathbb{F}$ is $\iota$. 
\end{setup}
Since $\mathbb{L}_{\cris,U_0}$ has constant Newton polygon,  $\mathbb{L}_{\cris,U_0}$ admits a slope filtration in the category of $F$-crystals over $U_0$:
\begin{equation}\label{eq:slopefiltration}
    0=\Fil_{-2}\subseteq \Fil_{-1}\subseteq \Fil_{0}\subseteq \Fil_{1}=\mathbb{L}_{\cris,U_0},
\end{equation}
where $\Fil_{-1}$ is the slope $-1/(a+1)$ part, $\Fil_{0}$ is the slope $\leq 0$ part. The rank of $\gr_{-1},\gr_0$ and $\gr_{1}$ are $a+1$, $b-2a$ and $a+1$, respectively. 

\subsection*{Formal Brauer groups} Note that the $F$-crystal $\mathbb{L}_{\textrm{EB}} : = \Fil_{0}/\Fil_{-1}$ is a Dieudonné crystal, the main result of \cite{DJ95} implies that $\mathbb{L}_{\textrm{EB}}=\mathbb{D}_{}(\mathbf{\Psi})$ where $\mathbb{D}$ is the contravariant crystalline Dieudonné functor and $\mathbf{\Psi}$ is a $p$-divisible group over $U_0$, called the \emph{enlarged formal Brauer group}. For a K3 surface over finite field, this recovers the enlarged formal Brauer group introduced in \cite[\S IV]{Artin-Mazur} and \cite[\S 3]{Nygaard-Ogus}, and the connected subgroup of this group is just the classical formal Brauer group associated to the K3 surface. 

\begin{theorem}
 \label{T: algebraicC}
Let $U$ be in Setup~\ref{setup:splittingcurve}. Let $D$ be the function field of $U$ and $E$ be the completion of $D$ at a place (not necessarily a place corresponding to a point in $U$). Suppose that 
$\mathbb{L}_{\cris,E}$ admits a rank $b-2a$ lattice of formal special endomorphisms. Then $\End (\mathcal{A}_D) \supsetneq \End(\mathcal{A})$, where $\mathcal{A}$ is the universal Kuga--Satake abelian scheme.

Further, if the quadratic form on this lattice of formal special endomorphisms is self-dual, then the slope filtration of the enlarged  formal Brauer group over $D$ (and in fact, over $U$) splits. 

 \end{theorem}

Given the results of Sections \ref{sec: decay superspecial}, \ref{sec: first step superspecial} and \ref{sec: supersingular}, Theorem \ref{thm: main local version} (and therefore Theorem \ref{mainShimura}) follows directly from Theorem \ref{T: algebraicC}. 

\subsection{Splittings}
Let $X$ be a $U_0$-scheme, we say that $\mathbb{L}_{\cris,X}$ \textit{splits in the middle}, if 
\begin{equation}\label{eq:splittingK31}
\mathbb{L}_{\cris,X}\simeq \gr_{0,X}\oplus \mathscr{E}_X,
\end{equation}
where $\mathscr{E}_X$ is an extension of $\Fil_{-1,X}$ by $\gr_{1,X}$. Rationally, we also say that $\mathbb{L}_{\cris,X}[p^{-1}]$ \textit{splits in the middle}, if (\ref{eq:splittingK31}) holds after up to inverting $p$.     

 Note that the $p$-divisible group $\mathbf{\Psi}$ sits inside the classical connected-étale sequence \begin{equation}\label{ecs}
    1\rightarrow \mathbf{\Psi}_{\text{con}}\rightarrow \mathbf{\Psi}\rightarrow \mathbf{\Psi}_{\text{ét}} \rightarrow 1.
\end{equation}
Applying $\mathbb{D}(-)$ to (\ref{ecs}), we recover \begin{equation}\label{ecs2}
    0\rightarrow \gr_0 \rightarrow \bL_{\textrm{EB}}\rightarrow \gr_{1} \rightarrow 0.
\end{equation}
\begin{lemma}\label{lm:splitinthemiddle} Suppose that $X$ is a regular excellent scheme. The following are equivalent: \begin{enumerate}
    \item $\bL_{\cris,X}$ splits in the middle, 
   \item the projection $\Fil_{0,X}\rightarrow \gr_{0,X}$ admits a section, 
    \item the embedding $\gr_{0,X} \rightarrow \bL_{\mathrm{EB},X}$ admits a retraction, 
    \item $\mathbf{\Psi}_X$ splits as  $\mathbf{\Psi}_{\mathrm{con},X}\times \mathbf{\Psi}_{\text{é}\mathrm{t},X}$.
\end{enumerate}
\end{lemma}
\begin{proof}
The equivalence of (1), (2), (3) follows easily from the existence of quadratic pairing on $\bL_{\cris, X}$. The equivalence of (3) and (4) follows from the fully-faithfulness of the crystalline Dieudonné functor; cf. \cite[Theorem 4.6]{DJ99}.
\end{proof}
\begin{remark}\label{rmk:rationalsplittingversion}
    The above lemma has an obvious rational version, i.e., $\bL_{\cris,X}[p^{-1}]$ splits in the middle if and only if $\Fil_{0,X}[p^{-1}]\rightarrow \gr_{0,X}[p^{-1}]$ admits a section, if and only if $\gr_{0,X}[p^{-1}] \rightarrow \bL_{\mathrm{EB},X}[p^{-1}]$ admits a retraction, if and only if $\mathbf{\Psi}_X$ splits as  $\mathbf{\Psi}_{\mathrm{con},X}\times \mathbf{\Psi}_{\text{é}\mathrm{t},X}$ up to isogeny. The proof is similar. 
\end{remark}
\begin{construction}\label{const:vspl}
  Suppose that $ \bL_{\cris,X}[p^{-1}]$ splits in the middle, we can construct an endomorphism of $\bH_{\cris,X}[p^{-1}]$ which is canonical up to scaling. Let's recall how to do this: since $\bL_{\cris,X}[p^{-1}]$ splits, we have $$\det\gr_{0,X}[p^{-1}]= \bigwedge^{b-2a}\gr_{0,X}[p^{-1}]\hookrightarrow\bigwedge^{b-2a}\bL_{\cris,   X}[p^{-1}].$$
Now $\gr_{0,X}[p^{-1}]$ is self-dual (due to the quadratic pairing), we have $\det\gr_{0,X}[p^{-1}]=\mathbbm{1}_X$, the constant object. By Remark~\ref{rmk:embwedge}, we have $$\mathbbm{1}_X\hookrightarrow \bigwedge^{b-2a}\bL_{\cris,   X}[p^{-1}]\hookrightarrow\End(\bH_{\cris, X}[p^{-1}]).$$ 
So any global section of $\mathbbm{1}_X$ gives rise to a global section of $\End(\bH_{\cris, X}[p^{-1}])$. This is equivalent to saying that $\bH_{\cris,X}[p^{-1}]$ admits an endomorphism corresponding to that global section, canonical up to scaling. We will call this endomorphism $\nu_{\mathrm{spl},X}$.   
\end{construction} 



\begin{proof}[Proof of Theorem \ref{T: algebraicC}]
  The fact that $\mathbb{L}_{\cris,  E}$ admits a rank $b-2a$ lattice of formal special endomorphisms implies that the slope 0 constant $F$-isocrystal 
$\mathbbm{1}^{\oplus (b-2a)}_{E}$ embeds into $\Fil_{0,E}[p^{-1}]\subseteq  \mathbb{L}_{\cris,  E}[p^{-1}]$. This implies that $\gr_{0,E}[p^{-1}]\simeq \mathbbm{1}^{ \oplus (b-2a)}_{E}$, and the quotient map $\Fil_{0,E}[p^{-1}]\rightarrow \gr_{0,E}[p^{-1}]$ admits a section. So by Lemma \ref{lm:splitinthemiddle} and the remark following that, $\mathbf{\Psi}_E\simeq \mathbf{\Psi}_{\mathrm{con},E}\times \mathbf{\Psi}_{\text{é}\mathrm{t},E}$ up to isogeny. Since the coordinate ring of $U_0$ is $p$-integrally closed in $E$ in the sense of \cite[\S 5.1]{RJ23}, we can apply Proposition 5.2 of \textit{loc.cit} to deduce that $\mathbf{\Psi}_{U_0}\simeq \mathbf{\Psi}_{\mathrm{con},U_0}\times \mathbf{\Psi}_{\text{é}\mathrm{t},U_0}$ up to isogeny. 

Let $D_0$ be the function field of $U_0$, which is a finitely generated field over $\mathbb{F}_q$. By Lemma \ref{lm:splitinthemiddle} again, we find that $\bL_{\cris,D_0}[p^{-1}]$ splits in the middle. Therefore Construction~\ref{const:vspl} gives rise to an endomorphism $\nu_{\mathrm{spl},D_0}\in \End(\bH_{\cris,D_0})\otimes{\bQ_p}$. By crystalline isogeny theorem; cf. \cite[Theorem 2.6]{DJ98}, we have $\End(\bH_{\cris,D_0})\otimes{\bQ_p}= \End(\mathcal{A}_{D_0})\otimes\bQ_p$. On the other hand, the existence of $\nu_{\mathrm{spl},D_0}$ accounts for the splitting of $\bL_{\cris,D_0}[p^{-1}]$ in the middle, since $\bL_{\cris}[p^{-1}]$ does not split in the middle over the generic point of $\mathscr{S}_{a+1,\mathbb{F}_q}$, we have 
$\nu_{\mathrm{spl},D_0}\notin\End(\mathcal{A})\otimes\bQ_p$. It follows that $\End(\mathcal{A})\otimes\bQ_p\subsetneq \End(\mathcal{A}_{D_0})\otimes\bQ_p$. Therefore $\End(\mathcal{A})\subsetneq \End(\mathcal{A}_{D_0})$.

Furthremore, suppose that the quadratic form on this rank $b-2a$ lattice of formal special endomorphisms is self-dual, then it is saturated. It follows that  
$\gr_{0,E}$ is a trivial $F$-crystal, and the quotient map $\Fil_{0,E}\rightarrow \gr_{0,E}$ admits a section (integrally). A similar argument as in the first paragraph implies that 
$\mathbf{\Psi}_{U_0}\simeq \mathbf{\Psi}_{\mathrm{con},U_0}\times \mathbf{\Psi}_{\text{é}\mathrm{t},U_0}$. 
\end{proof}

\appendix

\section{Index of notations}\label{sec: appendix}

\begin{supertabular}{c|c|l}
    notation & appearance & biref explanation \\ \hline
    $\shS$ & \S\ref{subsec: set up SV} & integral model of an orthogonal Shimura variety \\
    $\cA$  & \S\ref{subsec: set up SV} & Kuga-Satake abelian scheme over $\shS$ \\
    $b$ & \S\ref{subsec: set up SV} & dimension of $\shS$, or $\mathrm{rk\,} L - 2$ \\
    $\bH_?, \bL_?$ & \S\ref{subsec: set up SV} & $\bH_? = H^1_?(\cA)$ and $\bL_?$ is a direct sumamnd of $\underline{\End}(\bH_?)$ \\
    $L(\cA_T)$  & Def.~\ref{def: special end} & special endomorphisms on $\cA_T$ for a $\shS$-scheme $T$ \\
    $\cZ(m), \overline{\cZ(m)}$ & Def.~\ref{def: special end} & special divisor indexed by $m$ and its closure \\ 
    $\gamma_{-1}, \gamma_0$ & Eq.~\ref{eqn: gamma -1} & Frobenius isomorphisms in F-zip structures \\
    $(\cF_h, E_h), \shS_h$ & Def.~\ref{def: F_h and E_h} & Newton stratifications on $\shS_\bF$   \\  
    $E_0$ & \S\ref{subsub:gEs} & an Eisenstein series \\ 
    $q_L(m)$ & Eq.~\ref{eqn: Fourier expansion} & $q_L(m) = q_L(m, 0)$, Fourier cofficient of $E_0$ \\
    $i_P$ & Def.~\ref{def: intersect at P} & intersection multiplicity at $P$ \\
    $\delta(l, L, m)$ & Def.~\ref{def: local density} & local density of lattice $L$ representing $m$ over $\bZ_\ell$ \\
    $Z(m), \overline{Z(m)}$ & \S\ref{subsub: arithmetic intersection theory} & mod $p$ reductions of $\cZ(m), \overline{\cZ(m)}$ \\
    $\shS^{\mathrm{tor}}$ & \S\ref{subsub: arithmetic intersection theory} & a toroidal compactification of $\shS$ \\
    $\omega$ & Lem.~\ref{lem: Frob and E} & Hodge bundle of $\shS$, i.e., $\Fil^1 \bL_\dR$ \\
    $a$ & \S\ref{sub: intersection numbers} & index of the smallest Newton stratum $\cF_{a+1}$ containing $C$ \\
    $h_P$ & \S\ref{sub: intersection numbers} & $i_P(C \cdot \sF_{a + 2})$ \\
    $g_P(m)$ & Def.~\ref{def: global int num} & global intersection number at $P$ \\
    $L_{P, r}$ & Def.~\ref{def: L_P, r} & sublattice of $L(\cA_P)$ of those that deform to $r$th order \\
    $L'_{P, r}$ & Def.~\ref{def: L_P, r} & a modification of $L_{P, r}$ at prime-to-$p$ places \\
    $S'(m)$ or $S'(M)$ & Thm.~\ref{thm: main local version}  & a series comparing the local and global intersection numbers\\
    $\bL$ & \S\ref{subsub:explicitcoordinatheonL} & a shorthand for $\bL_{\cris, P}$ \\
    $\cL, \cL_0, \cL_1$ &\S\ref{subsub:explicitcoordinatheonL} & $\cL = \bL^{\Frob = 1}$, its non-self-dual part and self-dual part \\
    $\lambda$ &\S\ref{subsub:explicitcoordinatheonL} & a certain number in $W(\bF_{p^2})^\times$ such that $\sigma(\lambda) = - \lambda$ \\
    $t_P, a_P$ & \S\ref{subsub:explicitcoordinatheonL}&  $t_P = \mathrm{rank\,}\cL_0$ and $a_P = t_P/2$ (the Artin invariant of $P$)  \\
    $\{v', w'\}, \{e', f' \}$ & \S\ref{subsub:explicitcoordinatheonL} & (respectively) basis of $\bL_0$ and $\cL_0$ for $P$ superspecial \\
    $\{e_i, f_i\}_{1 \le i \le m}$ & \S\ref{subsub:explicitcoordinatheonL} & basis for both $\cL_1$ and $\bL_1$ \\
    $Q, Q_j$ & \S\ref{subsub:explicitcoordinatheonL} & power series on $\shS^{/P}_W \simeq \Spf(W[\![x_1, \cdots, x_m, y_1, \cdots, y_m ]\!])$ \\
    $x_v \in W[\![x_i, y_i]\!]$ & Def.~\ref{eq:v_defspace} & obstruction to deforming $v$.  \\
    $F_\infty$ & \S\ref{subsec: F infty} & matrix for infinite iterate of Frobenius \\
    $\tilde{v}$ & Eq.~\ref{eqn: def tilde v} & $\tilde{v} = F_\infty v$.  \\ 
    $\mathbf{c}, \wt{\mathbf{c}} \in \bZ_p^{2m + 2}$  & Eq.~(\ref{eqn: define c and c tilde}) & the coordinates of $v, \wt{v}$ in basis $\{ e', f', e_i, f_i \}$. \\
    $Q_{l_1 < l_2 < \cdots < l_{2n}}$  & Eq.~\ref{eqn: define Q <} & a product of twists of $Q_j$'s. \\
    $\delta_{l_1 < l_2 < \cdots < l_{2n}}$ & Eq.~\ref{eqn: define delta <} &  a sign, used together with $Q_{l_1 < l_2 < \cdots < l_{2n}}$.\\ 
    $F_n$ & Lem.~\ref{lm:Finftyexplicit} & the coefficient matrix of $p^{-n}$ in $F_\infty$\\
    $X_n^+, Y_n^+$ & Lem.~\ref{lm:Finftyrecursive} &  the parts of $X_n$ and $Y_n$ with even leading term. \\
    $\mathbf{U}_n$ & Lem.~\ref{lm:Finftyrecursive} &  the first row of $\begin{bmatrix}
        X_n^+ & Y_n^+
    \end{bmatrix}$. \\
    $D_n(v), \overline{D}_n(v)$ & Eq.~(\ref{eqn: defined D_n's}) & power series controlling the deformation of $v$ \\
    $d_n(v)$ & \S\ref{sub:setuplocalcurve} & the $n$th order rate of decay of $v$ with respect to $C$  \\ 
    $h_i$ & Not.~\ref{not: define h_i and a_i} & $t$-adic valuation $v_t(Q_i)$ \\
    $a_1,\cdots, a_k$ & Not.~\ref{not: define h_i and a_i} & index of a subsequence of $h_i$'s \\
    $\bI_r, \bI_r(v)$ & Def.~\ref{def: bI_r's} & certain words in $Q_i$'s and $x_v$ of length $r$ \\
    $\nu(-)$ & Def.~\ref{def: bI_r's} & $t$-adic valuation of words in $\bI_r$ or $\bI_r(v)$ \\
    $\nu^{\min}_r, \nu^{\min}(v)$ & Def.~\ref{def: bI_r's} &  the minimal $t$-adic valuation of words in $\bI_r, \bI_r(v)$) \\
    $\bI^{\min}_r, \bI^{\min}(v)$ & Def.~\ref{def: bI_r's} & words with minimal $t$-adic valuation in $\bI_r, \bI_r(v)$ \\
    $I_r^{\min}(v)$ & Def.~\ref{def: intervals I_r(v)} &  a certain interval in $[1, k]$ defined by $ \bI^{\min}_r(v)$ \\
    $\overline{\cL}_1$  & \S\ref{setup: pattern of first order decay} & mod $p$ reduction of the $\bZ_p$ lattice $\cL_1$ \\
    $\overline{\cL}_{\mathrm{rd}}$ & Def.~\ref{def:somebadass} & redundant subspace of $\overline{\cL_1}$ \\
    $\overline{\cL}_{\mathrm{bd}}$ & Def.~\ref{def:somebadass} & borderline subspace of $\overline{\cL_1}$ \\
     $\overline{\cL}_{\mathrm{brd}}$ & Def.~\ref{def:somebadass} & $\overline{\cL}_{\mathrm{bd}}\oplus \overline{\cL}_{\mathrm{rd}}$ \\
    $c_i, d_i$ & Not.~\ref{not: c_i's and d_i's} & $v_t(x_i), v_t(y_i)$ \\
    $\cL_{> i}, \cL'_{\le i}$ & \S\ref{subsub:auxlatticetight} & some auxilliary lattices \\
    $H_j$ & \S\ref{subsub:auxlatticetight} & $h_P(1 + p^{a_k + 1} + \cdots + p^{(a_k + 1)(j - 1)})$ \\
    $L''_{P, r}$ & \S\ref{subsub:auxlatticetight} & some auxilliary lattices modified from $L'_{P, r}$ \\
    $\delta_i, \delta'_i$ & Def.~\ref{def:localdensitydelta} & local densities $\delta(p, \cL_{\mathrm{brd}} + \cL_{> i}, M)$ and $\delta(p, \cL_{\mathrm{brd}} + \cL'_{\le i}, M)$\\
    $D_0$ & \S\ref{subsub: D_0} & a mysterious quantity \\
    $S''(M)$ & \S\ref{sub:hardsitu} & a variant of $S'(M)$ with $L'_{P, r}$ replaced by $L''_{P, r}$ \\
    $n'$ & \S\ref{sec: supersingular} & $n - 1$, where $n = t_P / 2$. \\
    $\Psi, \bL_{\mathrm{EB}}$ & S\ref{sec: algebraization} & enlarged formal Brauer group and its Dieudonn\'e module \\
    $\nu_{\mathrm{spl, -}}$ & Const.~\ref{const:vspl} & extra endomorphism on $\bH_{\cris, -}$ \\
\end{supertabular}

\bibliography{newref}
\bibliographystyle{alpha}

\noindent Ruofan Jiang {\footnotesize\,\,\, University of California, Berkeley, Department of Mathematics, Evans Hall,
Berkeley, CA 94720-3840, USA \,\,\,  Email: \url{ruofanjiang@berkeley.edu}}\\

\noindent Ananth Shankar {\footnotesize\,\,\, Northwestern University, Department of Mathematics, Lunt Hall,
Evanston, Il 60208, USA \,\,\,  Email: \url{ananth@northwestern.edu}}\\

\noindent Ziquan Yang {\footnotesize \,\,\, The Institute of Mathematical Sciences and Department of Mathematics, The Chinese University of Hong Kong, Shatin, N.T., Hong Kong.  \,\,\,  Email: \url{zqyang@cuhk.edu.hk}}

\end{document}